\documentclass{amsart}

\usepackage{amsmath}
\usepackage{amssymb}
\usepackage{graphicx}
\usepackage{amsmath,amscd}
\usepackage{tikz}
\usepackage{tikz-cd}
\usetikzlibrary{positioning,arrows}
\usepackage[utf8]{inputenc} 
\usepackage[T1]{fontenc}
\usepackage{enumitem}
\usepackage{hyperref}

\numberwithin{equation}{section}
\theoremstyle{plain}

\newtheorem{lem}[equation]{Lemma}

\newtheorem{prop}[equation]{Proposition}
\newtheorem{thm}[equation]{Theorem}
\newtheorem{cor}[equation]{Corollary}
\newtheorem{que}[equation]{Question}

\theoremstyle{definition}
\newtheorem{definition}[equation]{Definition}
\newtheorem{remark}[equation]{Remark}
\newtheorem{example}[equation]{Example}

\newcommand*{\inc}{\ensuremath{\mathcal{I}}}

\newcommand{\A}{{\mathcal A}}

\newcommand{\F}{{\mathcal F}}

\renewcommand{\P}{{\mathcal P}}

\newcommand{\R}{\mathbb R}

\renewcommand{\S}{{\mathcal S}}

\newcommand{\V}{{\mathcal V}}

\newcommand{\Z}{\mathbb Z}

\newcommand{\acts}{\curvearrowright}
\newcommand{\al}{\alpha}

\newcommand{\Ga}{\Gamma}

\newcommand{\out}{\operatorname{Out}}

\newcommand{\Stab}{\operatorname{Stab}}

\begin{document}

\title[QUASI-ISOMETRIC CLASSIFICATION OF RIGHT-ANGLED ARTIN GROUPS II]{QUASI-ISOMETRIC CLASSIFICATION OF RIGHT-ANGLED ARTIN GROUPS II: SEVERAL INFINITE OUT CASES}

\author{JINGYIN HUANG}
\address{The Department of Mathematics and Statistics\\
McGill University\\
Burnside Hall, 805 Sherbrooke W.\\
Montreal, QC, H3A 0B9, Canada}

\email{jingyin.huang@mcgill.ca}
\begin{abstract}
We are motivated by the question that for which class of right-angled Artin groups (RAAGs), the quasi-isometric classification coincides with commensurability classification. This is previously known to hold for RAAGs with finite outer automorphism groups. In this paper, we identify two classes of RAAGs, where their outer automorphism groups are allowed to contain adjacent transvections and partial conjugations, hence are infinite. If $G$ belongs to one of these classes, then any other RAAG $G'$ is quasi-isometric to $G$ if and only if $G'$ is commensurable with $G$. We also show that in such cases, there exists an algorithm to determine whether two RAAGs are quasi-isometric by looking at their defining graphs. Compared to the finite out case, the main issue we need to deal with here is that one may not be able to straighten the quasi-isometries in a canonical way. We introduce a deformation argument, as well as techniques from cubulation to deal with this issue.

\smallskip 

\noindent
\textbf{AMS classification numbers}.  20F65, 20F67, 20F69, 05C25 %52B30, 55N25, %57N65\\%57R18, 57R55, \\

\smallskip

\end{abstract}

\maketitle
\tableofcontents
\setcounter{tocdepth}{2}

\section{Introduction}

\subsection{Motivation and Background}
Recall that a \emph{quasi-isometry} $q:X\to Y$ between two metric spaces $X$ and $Y$ is a map such that there exist constants $L,A>0$ with the following properties:
\begin{enumerate}
	\item $L^{-1}d(x,y)\le d(f(x),f(y))\le Ld(x,y)$ for any $x,y\in X$;
	\item each point in $Y$ is at most distance $A$ from a point in $q(X)$.
\end{enumerate}

Given a quasi-isometry $q:X\to Y$ between two metric spaces, one common scheme of understanding $q$ is the following. In step one we specify a collection of subspaces of $X$ and $Y$ such that they are stable under $q$, and encode the coarse intersection pattern of the these subspaces of $X$ (or $Y$) in a combinatorial object $C_X$ (or $C_Y$). Then $q$ induces a \textquotedblleft isomorphism\textquotedblright\ $q_{\ast}:C_X\to C_Y$. In step two we understand whether an isomorphism between $C_X$ and $C_Y$ implies whether $X$ and $Y$ are isometric, or at least share some interesting geometric feature. Here are two examples of these scheme:
\begin{itemize}
\item When $X=Y=SL(n,\Bbb R)/SO(n)$ for $n\ge 3$, in step one one shows that $q$ preserves the intersection pattern of maximal flats in $X$, hence induces an automorphism of the spherical building at infinity, which is a simplicial complex encoding the intersecting pattern of these flats. Moreover, this automorphism is continuous with respect to the cone topology. In step two we use the fundamental theorem of projective geometry to deduce that such automorphism of the building actually comes from a homothety of $X$. Then one deduces that every such $q$ is of bounded distance from a homothety. This is a special case of the results in  \cite{kleiner1997rigidity,eskin1997quasi}.
\item When $X$ and $Y$ are the mapping class groups of oriented closed surfaces of genus $\ge 2$, $q$ preserves the intersection pattern of Dehn twist flats \cite{hamenstaedt2005geometry,MR2928983}, hence induces an automorphism of the curve complex, which is a simplicial complex encoding the intersecting pattern of the Dehn twist flats. However, Ivanov's theorem tells us any automorphism of the curve complex is induced by a mapping class, hence $q$ is of bounded distance from a left multiplication.
\end{itemize}

While for several other classes of groups and spaces, it is tempting to follow such a scheme to study quasi-isometric classification and rigidity, one cannot expect we are always as lucky in step two where some analogue of fundamental theorem of projective geometry or Ivanov's theorem would hold. 

Now we look at the class of right-angled Artin groups (RAAG). Given a finite simplicial graph $\Gamma$ with vertex set $\{v_{i}\}_{i\in I}$, the right-angled Artin group with defining graph $\Gamma$, denoted by $G(\Gamma)$, is given by the following presentation:
\begin{center}
	\{$v_i$, for $i\in{I}\ |\ [v_i,v_j]=1$ if $v_{i}$ and $v_{j}$ are joined by an edge\}
\end{center}

For RAAGs there are no combinatorial objects as in the above two cases such that on one hand they are quasi-isometry invariants, and on the other hand they satisfy a strong analogue of Ivanov's theorem; as if such objects exist, it would imply the quasi-isometry groups of RAAGs are rather small. On the other hand, even for the most rigid class of RAAGs, their quasi-isometry groups are quite large \cite{raagqi1}.

Under certain conditions, quasi-isometries between RAAGs preserve a collection of subspaces, called standard flats. Kim and Koberda \cite{kim2013embedability} introduced the notion of extension complexes, which is a simplicial complex encoding the coarse intersecting pattern of standard flats in RAAGs. They are quasi-isometry invariants for large class of RAAGs.

For RAAGs with finite outer automorphism groups considered in \cite{MR2421136,raagqi1}, it is proved that any automorphism $\alpha$ of the extension complex induces a canonically defined bijection $\alpha'$ of the underlying RAAG which preserves enough structure for application to quasi-isometric classification. We will refer the map $\alpha'$ as the ``reconstruction map'', as the map is closely related to reconstructing the RAAG from intrinsic combinatorial structure of its extension complex (the precise definition of $\alpha'$ is in Section~\ref{sec:visible}). This is also a natural extension of classical reconstruction problems (e.g. a theorem of Darboux says that any straight line preserving bijection of Euclidean spaces must be affine) to the context of RAAGs. Note however the map $\alpha'$ is typically very far away from being a left multiplication or an isometry, so this can be viewed as a weak analogue of what happens in step two of the two cases discussed above. 

For RAAGs with infinite outer automorphism, the situation could be much worse:
\begin{enumerate}
\item It is possible that there are no well-defined reconstruction maps (in the sense of Section~\ref{sec:visible}).
\item Even if the reconstruction map exists, it may not preserve as much structure as before. This is due to the fact that RAAGs may not \textquotedblleft branch\textquotedblright\ as much as symmetric spaces, thick Euclidean buildings or mapping class groups. Extra conditions are needed to make the reconstruction map \textquotedblleft nice\textquotedblright, and such cases are studied in \cite{MR2421136,raagqi1}.
\end{enumerate}

In this paper we show that quasi-isometric classification results could still be obtained despite these two challenges. We study two classes of RAAGs in this paper. The first class is the largest class of RAAGs such that the reconstruction map exists with respect to automorphisms of extension complexes. Then we will introduce another class of RAAGs, where the reconstruction map fails to exist, and indicate how to get around this issue. It turns out that ideas from cubulation theory are relevant.

The previous quasi-isometric classification results of RAAGs fall into two classes with strong contrast in their conclusions. \cite{behrstock2008quasi} identifies a class of RAAGs whose quasi-isometry types do not depend on the defining graphs, while \cite{MR2421136} identifies another class of RAAGs such that two RAAGs in this class are quasi-isometric if and only if they are isomorphic. Higher dimensional generalizations of these two cases are in \cite{MR2727658} and \cite{raagqi1} respectively. We intend to understand this strong contrast by \textquotedblleft interpolating\textquotedblright\ between these two cases. The classes of RAAGs discussed in this paper serve as an initial step towards this goal.

%Compared to the quasi-isometry rigidity of lattices and mapping class groups, the form of rigidity obtained in the case of RAAG is weaker and subtle even if one consider the most \textquotedblleft rigid\textquotedblright\ RAAGs. The quasi-isometry group $QI(G(\Gamma))$ is very huge, see the discussion about quasi-isometry flexibility in \cite{MR2421136} and \cite{raagqi1}.

\subsection{Main results and open questions}
We denote the RAAG with defining graph $\Ga$ by $G(\Ga)$. Our search for appropriate classes of RAAGs is roughly guided by the outer automorphism group $\out(G(\Gamma))$. Namely if a property is true for all elements in $\out(G(\Gamma))$, then we ask whether it is also true for all quasi-isometries of $G(\Ga)$. See Section \ref{basics about raag} for a review of $\out(G(\Gamma))$. Since we are mainly interested in the case where $\out(G(\Gamma))$ is infinite, we need to focus on the 3 types of generators of $\out(G(\Gamma))$ which are of infinite order, namely the adjacent transvections, non-adjacent transvections and partial conjugations. Adjacent transvections happen inside free abelian subgroups, so they have relatively nice behaviour compared other types. We deal with it first.

\begin{definition}
\label{1.1}
$G(\Gamma)$ is of \textit{weak type I} if
\begin{enumerate}
\item $\Gamma$ is connected and does not contain any separating closed star.
\item There do not exist vertices $v,w\in\Gamma$ such that $d(v,w)=2$ and $\Gamma=St(v)\cup St(w)$.
\end{enumerate}
\end{definition}

We caution the reader that in this paper, the closed star of a vertex $v$, which we denote by $St(v)$, is defined to be the full subgraph spanned by $v$ and vertices adjacent to $v$. This definition is slightly different from the usual one. Similarly, $lk(v)$ is defined to be the full subgraph spanned by vertices adjacent to $v$.

It turns out that $G(\Gamma)$ is of weak type I if and only if one can always reconstruct a map from $G(\Ga)$ to itself from a given isomorphism of its extension complex in the sense of Definition \ref{visible}, see Theorem \ref{7.28} for a precise statement. In particular, all RAAGs with finite outer automorphism group are of weak type I.

If $G(\Gamma)$ is of weak type I, then $\out(G(\Gamma))$ does not contain non-adjacent transvections and partial conjugations, however, $\out(G(\Ga))$ may contain adjacent transvections. For example, we can take $\Gamma$ to be the graph which is made of a 5-cycle and a 3-cycle glued along an edge.

%If $G(\Gamma)$ is 2-dimensional, then $\out(G(\Gamma))$ contains adjacent transvections if and only if $\out(G(\Gamma))$ contains non-adjacent transvections, so in this case, $G(\Gamma)$ is of weak type I if and only if $\out(G(\Gamma))$ is finite.

\begin{thm}[=Theorem~\ref{7.26}]
\label{1.2}
If $G(\Gamma)$ and $G(\Gamma')$ are of weak type I, then they are quasi-isometric if and only if they are isomorphic.
\end{thm}

Having weak type I is not a quasi-isometry invariant (cf. \cite[Example 1.4]{MR2421136}). However, the following weaker version of Theorem \ref{1.2} is true when only $G(\Ga_1)$ is of weak type I.

\begin{thm}[=Theorem~\ref{6.12}]
\label{1.3}
Suppose $G(\Gamma_{1})$ is of weak type I. Then the following are equivalent:
\begin{enumerate}
\item $G(\Gamma_{2})$ is quasi-isometric to $G(\Gamma_{1})$.
\item $G(\Gamma_{2})$ is isomorphic to a subgroup of finite index in $G(\Gamma_{1})$.
\item $G(\Gamma_{2})$ is isomorphic to a special subgroup of $G(\Gamma_{1})$.
\end{enumerate}
\end{thm}

We refer to Section \ref{special subgroup} for the definition of special subgroups. 

%Given an arbitrary RAAG $G(\Gamma)$ (not necessarily of weak type I) and pick a standard generating set $S$ for $G(\Gamma)$, let $d_{S}$ be the word metric on $G(\Gamma)$ with respect to $S$. A subset $K\subset G(\Gamma)$ is \textit{$S$-convex} if and only if for any three points $x,y\in K$ and $z\in G(\Gamma)$ such that $d_{S}(x,y)=d_{S}(x,z)+d_{S}(z,y)$, we must have $z\in K$. Every finite $S$-convex subset $K$ naturally gives rise to a finite index RAAG subgroup $G\le G(\Gamma)$ such that $K$ is the fundamental domain of the left action $G\curvearrowright G(\Gamma)$. For example, if $G(\Gamma)=\Bbb Z\oplus\Bbb Z$ and pick $K$ to be a rectangle of size $n$ by $m$, then the corresponding subgroup is of form $n\Bbb Z\oplus m\Bbb Z$. The detailed construction of this is given in \cite[Section 6.1]{raagqi1}. $G$ is called a \textit{$S$-special} subgroup of $G(\Gamma)$. A subgroup of $G(\Gamma)$ is \textit{special} if and only if it is $S$-special for some standard generating set $S$. 

\begin{remark}\
\label{1.4}
\begin{enumerate}
\item As we will see later, in general the collection of special subgroups of $G(\Ga)$ depends on the choice of standard generators of $G(\Ga)$. However, the isomorphism types of the special subgroups does not depend on the choice of standard generators (by \cite[Section 6]{raagqi1}, all special subgroups are RAAGs and the isomorphism types of their defining graphs only depend on $\Ga$). For example, let $G(\Ga)\cong \Z\oplus\Z$. Given a base $e_1,e_2\in \Z\oplus\Z$, then special subgroups with such choice of standard generators are of form $\{n e_1+me_2\}_{n,m\in\Z}$. However, no matter what base we choose, all special subgroups are isomorphic to $n\Z\oplus m\Z$. However, we have more rigidity when $\out(G(\Gamma))$ is finite. In this case, the definition of special subgroup does not depend on the choice of standard generators and all finite index RAAG subgroups of $G(\Gamma)$ are special subgroups (\cite[Theorem 1.4]{raagqi1}). 
\item By \cite[Theorem 1.3]{raagqi1}, $G(\Gamma_{2})$ is quasi-isometric to $G(\Ga_1)$ if and only if their extension complexes (Section \ref{basics about raag}) are isomorphic, given $\out(G(\Gamma_{1}))$ is finite. The if only direction is still true in the case of weak type I group, but the other direction is not clear.
\end{enumerate}
\end{remark}

Next we deal with partial conjugations. Since if $\out(G(\Ga))$ contains partial conjugations, then $\Ga$ contains separating closed stars, one may want to cut $\Gamma$ into good pieces along separating closed stars. However, this is not well-defined in general. Then one may try the opposite way and look at graphs obtained by gluing good pieces along vertex stars in a nice way. By studying such examples, we identify the following class of RAAGs. 

\begin{definition}
\label{1.5}
$G(\Gamma)$ is of \textit{type II} if $\Gamma$ is connected and for every pair of distinct vertices $v,w\in\Gamma$, $lk(v)\cap lk(w)$ does not separate $\Gamma$.
\end{definition}

This condition has a geometric interpretation. Note that $lk(v)$ corresponds to hyperplanes in the universal covering of the Salvetti complex, as these hyperplanes are invariant under a conjugate of subgroups of $G(\Gamma)$ generated by vertices in $lk(v)$. So $lk(v)\cap lk(w)$ corresponds to the intersection of hyperplanes. Definition \ref{1.5} can be roughly interpreted as \textquotedblleft hyperplanes of codimension 2 do not coarsely separate\textquotedblright.

A model example is taking $\Gamma$ to be the union of a 5-cycle and a 6-cycle identified along a closed vertex star. If $G(\Gamma)$ is of type II, then $\out(G(\Gamma))$ may contain partial conjugations and adjacent transvections, but not non-adjacent transvections. 

A similar but different condition, called SIL, has been studied in \cite{MR2776995}. The SIL condition also plays a role in the study of right-angled Coxeter groups \cite{cunningham2016recognizing,cunningham2016cat}. 

\begin{thm}[=Theorem~\ref{8.17}]
\label{1.6}
If $G(\Gamma_{1})$ is a right-angled Artin group of type II, then $G(\Gamma_{2})$ is quasi-isometric to $G(\Gamma_{1})$ if and only if $G(\Gamma_{2})$ is commensurable with $G(\Gamma_{1})$. Moreover, there exists a right-angled Artin group $G(\Gamma)$ such that $G(\Gamma_{1})$ and $G(\Gamma_{2})$ are isomorphic to special subgroups in $G(\Gamma)$.
\end{thm}

The following is a consequence of Theorem \ref{1.3}, Theorem \ref{1.6} and \cite[Section 6.3]{raagqi1}.
\begin{cor}
Let $G(\Ga)$ be a right-angled Artin group of type II or weak type I. Then there is an algorithm to determine whether a given right-angled Artin group $G(\Ga')$ is quasi-isometric to $G(\Ga)$ or not.
\end{cor}

We close this section with several comments and open questions. A RAAG of weak type I is not necessarily of type II. The following class contains RAAGs of both weak type I and type II (see Lemma \ref{7.25}).

\begin{definition}
$G(\Ga)$ is said to have \emph{weak type II} if $\Gamma$ is connected and for vertices $v,w\in\Gamma$ such that $d(v,w)=2$, $\Gamma\setminus (lk(v)\cap lk(w))$ is connected.
\end{definition}

It turns out that weak type II is a quasi-isometry invariant for RAAGs, see Corollary \ref{7.22}. Though a large portion of our discussion also generalize to RAAGs of weak type II, the following question remains open.

\begin{que}
Suppose $G(\Ga)$ is of weak type II and $G(\Ga')$ is quasi-isometric to $G(\Ga)$. Is $G(\Ga')$ commensurable with $G(\Ga)$?
\end{que}

The techniques in this paper does not seem to apply effectively to the case when there are non-adjacent transvections in the outer automorphism group. Indeed, in this case, there is serious breakdown of the above form of rigidity. For example, there exist two tree RAAGs which are quasi-isometric but not commensurable (\cite{behrstock2008quasi}). This leads to the following question:
\begin{que}
Suppose $\Ga$ is connected and does not admit a non-trivial join decomposition. Suppose $\out(G(\Ga))$ contains non-trivial non-adjacent transvection. Does there exist $\Ga'$ such that $G(\Ga)$ and $G(\Ga')$ are quasi-isometric, but not commensurable?
\end{que}

\subsection{Comments on the Proof}
We refer to Section \ref{basics about raag} for definitions of relevant terms. The Salvetti complex of $G(\Gamma)$ is denoted by $S(\Gamma)$, the universal covering of $S(\Gamma)$ is denoted by $X(\Gamma)$, and flats in $X(\Gamma)$ that cover standard tori in $S(\Gamma)$ are called standard flats. Two standard flats are \textit{coarsely equivalent} if they have finite Hausdorff distance. Let $\P(\Ga)$ be the extension complex of $X(\Ga)$. The $k$-dimensional simplices in $\P(\Ga)$ are in 1-1 correspondence with coarse equivalent classes of $(k+1)$-dimensional standard flats in $X(\Ga)$. Thus $\P(\Ga)$ captures the coarse intersection pattern of standard flats in $X(\Ga)$. It turns out to be a quasi-isometry invariant for a large class of RAAGs.

\begin{thm}
\label{1.7}
Let $q:G(\Gamma_{1})\to G(\Gamma_{2})$ be a quasi-isometry. Suppose $\out(G(\Gamma_{i}))$ does not contain any non-adjacent transvection for $i=1,2$. Then $q$ preserves maximal standard flats up to finite Hausdorff distance. Moreover it induces a simplicial isomorphism $q_{\ast}:\mathcal{P}(\Gamma_{1})\to\mathcal{P}(\Gamma_{2})$. 
\end{thm}

The assumption of Theorem \ref{1.7} is motivated by the observation that any automorphism of $G(\Ga)$ preserves maximal standard flats up to finite Hausdorff distance if and only if there is no non-adjacent transvection in $\out(G(\Ga))$. This observation can be easily checked by going through the list at the end of Section \ref{basics about raag} (note that maximal standard flats are in 1-1 correspondence to maximal standard abelian groups in $G(\Ga)$).

One can try to reconstruct a \textquotedblleft straightening\textquotedblright\ of $q$ from $q_{\ast}$ as follows. Pick vertex $x\in X(\Ga_1)$ and let $\{F_i\}_{i\in I}$ be the collection of maximal standard flats containing $x$. Under a mild condition on $\Gamma$, we have $x=\cap_{i\in I}F_i$. Each $F_i$ is associated with a maximal standard flat $F'_{i}\subset X(\Ga_2)$ by Theorem \ref{1.7}. It is natural to define $\bar{q}:G(\Gamma_{1})\to G(\Gamma_{2})$ such that $\bar{q}(x)=\cap_{i\in I}F'_i$. However, it is possible that $\cap_{i\in I}F'_i=\emptyset$. 

\subsubsection{The weak type I case} It turns out that this is exactly the case that we always have $\cap_{i\in I}F'_i\neq\emptyset$. Under an extra mild condition we can deduce $\cap_{i\in I}F'_i$ is actually a point. Then the map $\bar{q}$ is well-defined, and it preserves all the maximal standard flats. A priori, $\bar{q}$ may not preserve standard flats which are not maximal, and the key to prove Theorem \ref{1.3} is to deform $\bar{q}$ such that it preserves all standard flats.

A standard flat is \textit{rigid} if $\bar{q}$ sends its vertex set to the vertex set of another standard flat, otherwise it is \textit{non-rigid}. For example, all intersections of maximal standard flats are rigid, but the converse may not be true.

We will deform $\bar{q}$ in an inductive way. The first step is to show one can deform with respect to minimal rigid flats such that any standard flat contained in a minimal rigid flat is preserved by $\bar{q}$. To continue the induction argument, note that inside a (not necessarily minimal) rigid flat, there are directions which are rigid and directions which are not rigid. So we need to perform the deformation such that each move does not undo the previous moves, and does not place obstructions to the moves after. The second point is non-trivial, since rigid flats may intersect each other in a complicated pattern. To describe the deformation, we introduce an atlas for $G(\Gamma)$, where the vertex sets of standard flats are consistently labelled by free abelian groups. The detail is discussed in Section 5.

\subsubsection{The type II case}
In this case, the map $\bar{q}$ may fail to exist. For example, one can take $q$ to be a partial conjugation. 

Instead of reconstructing maps, we ask whether one can reconstruct the space $X(\Ga)$ from $\P(\Ga)$. Note that $X(\Ga)$ is a $CAT(0)$ cube complex. In general, the collection of halfspaces in a $CAT(0)$ cube complex, and their intersection pattern contains the complete information needed to reconstruct the complex itself. This can be formalized in the language of pocset (see Definition \ref{2.3} and Theorem \ref{2.5}).

Then we ask whether we can put a pocset structure on $\P(\Ga)$ such that it is the right pocset structure to recover $X(\Ga)$. This can always be done. Roughly speaking, one can embed $\P(\Ga)$ into the Tits boundary of $X(\Ga)$. Moreover, the collection of subsets of $\P(\Ga)$ which are the intersections of $\P(\Ga)$ and the Tits boundary of halfspaces of $X(\Ga)$ has a natural pocset structure. 

Briefly speaking, $X(\Ga)$ is equivalent to $\P(\Ga)$ with some decorations on $\P(\Ga)$. In general, these decorations depend on how one embeds $\P(\Ga)$ into the Tits boundary, so they do not come from intrinsic properties of $\P(\Ga)$. Thus the rigidity of $X(\Ga)$ depends on the amount of non-intrinsic decorations we need to put on $\P(\Ga)$. For example, in the most rigid case when $G(\Ga)$ has finite outer automorphism group, the amount of decorations needed is minimal. The worst case is when $X(\Ga)$ is tree, then $\P(\Ga)$ is just a discrete set.  

If $G(\Ga)$ is of type II, then the amount of extra decorations is reasonably small (see Corollary \ref{7.14} (1) for a precise statement). We prove Theorem \ref{1.6} in two steps. The pocset structure on $\P(\Ga)$ is defined in terms of a certain partition of $\P(\Ga)$. First we show it is possible to refine this partition to obtain a new pocset which does not admit any reasonable further refinement. It turns out the new pocset gives rise to another RAAG which is commensurable with the original one. Such RAAGs are called \textit{prime} RAAGs (Definition \ref{8.3}). Then we show two prime RAAGs are quasi-isometric if and only if they are isomorphic, which finishes the proof of Theorem \ref{1.6}. We caution the reader that in order to avoid some technicality, we work with pocset on $\Ga$ rather than $\P(\Ga)$ in Section 6. However, the idea is similar.

\subsection{Organization of the Paper}
In Section 2 we summarize and generalize several results from \cite{raagqi1} about $CAT(0)$ cube complexes, right-angled Artin groups and extension complexes. In particular, Theorem \ref{1.7} will be proved in Section 2.2.

In Section 3 we study the structure of the extension complex $\mathcal{P}(\Gamma)$ and prove Theorem \ref{1.2} at the end of Section 3. In Section 4, we prove some extra properties for extension complex for later use in Section 6.

The goal of Section 5 is to prove Theorem \ref{1.3}. We will introduce a notion of atlas for right-angled Artin group in Section 5.1 and use this in Section 5.2 as an effective language to describe the deformation argument mentioned above. 

We prove Theorem \ref{1.6} in Section 6. Section 6 does not depend on Section 5.

\subsection{Index of notation}
\label{sec_index_of_notation}
\begin{itemize}
	\item $St(v,K)$: the closed star of $v$ in $K$, or $St(v)$ if $K$ is clear (Section~\ref{subsec_notation}).
	\item $Lk(x,X)$ or $Lk(c,X)$: the link of a vertex $x$ or a cell $c$ in a polyhedral complex $X$ (Section~\ref{subsec_notation}).
	\item $\Ga_1\circ\Ga_2$: the join of two graphs (Section~\ref{subsec_notation}).
	\item $K_1\ast K_2$: the join of two simplicial complexes (Section~\ref{subsec_notation}).
	\item $V^\perp $: collection of vertices which are adjacent to each vertex in $V$ (Section~\ref{subsec_notation}).
	\item $G(\Ga)$ the right-angled Artin group with defining graph $\Ga$ (Section~\ref{basics about raag}).
	\item $F(\Gamma)$: the flag complex of a simplicial graph $\Gamma$.
	\item $X(\Ga)\to S(\Ga)$ the universal covering of the Salvetti complex (Section~\ref{basics about raag}).
	\item $\Gamma_K$: the support of a subcomplex $K$ in $X(\Gamma)$ (Definition~\ref{notation}).
	\item $V_K$: vertex set of $\Gamma_K$ (Definition~\ref{notation}).
	\item $\inc(C_{1},C_{2})=(Y_{1},Y_{2})$: nearest point sets between $Y_1$ and $Y_2$ defined in Lemma~\ref{2.6} and the paragraph after.
	\item $\mathcal{P}(\Ga)$: the extension complex 
	(Section~\ref{basics about raag}).
	\item $\Delta(K)$: for each subcomplex $K\subset X(\Gamma)$, $\Delta(K)$ is a subcomplex of $\mathcal{P}(\Gamma)$ defined in the paragraph before Lemma~\ref{boundary of std subclex}.
	\item For a vertex $p$ of $X(\Gamma)$, the map $i_p:F(\Gamma)\to \P(\Ga)$ is defined right before Lemma~\ref{isometric embedding}.
	\item $(F(\Ga))_p$: the image of $i_p$, see the paragraph before Lemma~\ref{isometric embedding}.
	\item $\pi:\mathcal{P}(\Ga)\to F(\Ga)$: label-preserving canonical projection defined in the 3rd paragraph after Lemma~\ref{coarse contain of standard subcomplexes}.
	\item $v(M)$: the set of vertices in a subset $M$ of a complex.
	\item $\pi_{v}$ or $\pi_{\Delta(\ell)}$ where $\ell$ is a standard line with $\Delta(\ell)=v$: a projection map defined in Definition~\ref{def:proj}.
	\item $\pi_\ell:X(\Gamma)\to\ell$: CAT(0) projection from $X(\Gamma)$ to a standard geodesic $\ell$ (Section~\ref{subsec:tiers}).
	\item $P_v$: the parallel set of a standard geodesic $\ell$ with $\Delta(\ell)=v$ (see the paragraph after Example~\ref{ex:type II}).
	\item $\partial B$: boundary of a subcomplex $B$, defined in the paragraph after Lemma~\ref{7.10}.
	\item $\bar B$: the full subcomplex spanned by $B$ and $\partial B$ (Section~\ref{sec:correspondence}).
	\item $\Pi$: this is a map which associated a component of $\P(\Ga)\setminus St(v)$ with a component of $F(\Gamma)\setminus St(\bar v)$, explained in the paragraph before Definition~\ref{def:6.4}.
\end{itemize}

\subsection{Acknowledgement}
The author thanks Ruth Charney, Bruce Kleiner, and Thomas Koberda for helpful discussions and Do-Gyeom Kim for valuable comments. The author is very grateful to the anonymous referees for very detailed comments.

\section{Preliminaries}
\subsection{Notations and conventions}
\label{subsec_notation}
Notations here are consistent with \cite[Section 2.1]{raagqi1}. All graphs in this paper are simplicial. The flag complex of a graph $\Gamma$ is denoted by $F(\Gamma)$, i.e. $F(\Gamma)$ is a flag complex such that its 1-skeleton is $\Gamma$.

Let $K$ be a polyhedral complex. 
\begin{enumerate}
\item By viewing the 1-skeleton of $K$ as a metric graph with edge lengths 1, we obtain a metric defined on the 0-skeleton of $K$, which we denote by $d$.
\item A subcomplex $K'\subset K$ is \textit{full} if $K'$ contains all the subcomplexes of $K$ which have the same vertex set as $K'$. If $K$ is 1-dimensional, then we also call $K'$ a \textit{full subgraph}. 
\item We use $\circ$ to denote the join of two graphs, namely $\Gamma_1\circ \Gamma_2$ is a graph obtained by adding to $\Gamma_1\sqcup \Gamma_2$ an edge between each vertex of $\Gamma$ and each vertex of $\Gamma_2$ and $\ast$ to denote the join of two polyhedral complexes.
\item For a set of vertices $V\subset K$, $V^{\perp}$ is defined to be collection of vertices which are adjacent to each vertex in $V$.
\item Let $v\in K$ be a vertex. The \textit{link} of $v$ in $K$, denoted by $lk(v,K)$ or $lk(v)$ when $K$ is clear, is defined to be the full subcomplex spanned by $v^{\perp}$. The \textit{closed star} of $v$ in $K$, denoted by $St(v,K)$ or $St(v)$ when $K$ is clear, is defined to be the full subcomplex spanned by $\{v,v^{\perp}\}$. 
\item Let $M\subset K$ be an arbitrary subset. We denote the collection of vertices inside $M$ by $v(M)$.
\end{enumerate}

We will be using the following simple observation repeatedly.

\begin{lem}
\label{connect}
Let $K$ be a simplicial complex and let $K^{(1)}$ be the 1-skeleton of $K$. Suppose $L\subset K$ is a full subcomplex. Then there is a 1-1 correspondence between connected components of $K\setminus L$ and of $K^{(1)}\setminus L^{(1)}$. Moreover, the intersection of each component of $K\setminus L$ with $K^{(1)}$ is a component of $K^{(1)}\setminus L^{(1)}$.
\end{lem}

Let $X$ be a metric space. We use $d_{H}$ to denote the Hausdorff distance and use $N_{R}(Y)$ to denote the open $R$-neighbourhood of a subspace $Y\subseteq X$. Two subspaces $A$ and $B$ are \textit{coarsely equivalent} if they have finite Hausdorff distance. $A$ is \emph{coarsely contained} in $B$ if $A$ is contained in the $R$-neighborhood of $B$ some $R>0$. A subspace $V\subseteq X$ is the \textit{coarse intersection} of subspaces $Y_1$ and $Y_2$ if $V$ is at finite Hausdorff distance from $N_R(Y_1)\cap N_{R}(Y_2)$ for all sufficiently large $R$. In general the coarse intersection of two subspaces might not exist.

\subsection{$CAT(0)$ spaces and $CAT(0)$ cube complexes}
We mention several relevant facts here and refer to \cite{MR1744486} and \cite{sageevnotes} for more background on $CAT(0)$ spaces and $CAT(0)$ cube complexes. The reader can also check \cite[Section 2.2]{raagqi1}. 

Let $(X,d)$ be a $CAT(0)$ space and let $C\subset X$ be a convex subset. We denote the nearest point projection from $X$ to $C$ by $\pi_{C}:X\to C$. Denote the Tits boundary of $X$ by $\partial_{T}X$. If $C'\subset X$ be another convex set, then $C'$ is \textit{parallel} to $C$ if $d(\cdot,C)|_{C'}$ and $d(\cdot,C')|_{C}$ are constant functions. We define the \textit{parallel set} of $C$, denoted by $P_{C}$, to be the union of all convex subsets of $X$ parallel to $C$.

Now we turn to $CAT(0)$ cube complexes. All cube complexes in this paper are assumed to be finite dimensional. There are two common metrics on a $CAT(0)$ cube complex $X$, namely the $CAT(0)$ metric and the $\ell^{1}$-metric. In this paper, we will mainly use the $CAT(0)$ metric unless otherwise specified.

A \textit{geodesic segment, geodesic ray} or \textit{geodesic} in a $CAT(0)$ cube complex $X$ is an isometric embedding of $[a,b]$, $[0,\infty)$ or $\Bbb R$ into $X$ with respect to the $CAT(0)$ metric. A \textit{combinatorial geodesic segment, combinatorial geodesic ray} or \textit{combinatorial geodesic} is a $\ell^{1}$-isometric embedding of $[a,b]$, $[0,\infty)$ or $\Bbb R$ into $X^{(1)}$ such that its image is a subcomplex.

The collection of convex subcomplexes in a $CAT(0)$ cube complex enjoys the following version of Helly's property (\cite{gerasimov1998fixed}):

\begin{lem}
\label{2.1}
Let $X$ be a finite-dimensional $CAT(0)$ cube complex and let $\{C_{i}\}_{i=1}^{k}$ be a collection of convex subcomplexes. If $C_{i}\cap C_{j}\neq\emptyset$ for any $1\le i\neq j\le k$, then $\cap_{i=1}^{k}C_{i}\neq\emptyset$.
\end{lem}

\begin{lem}
\label{combinatorial convex}
\cite{haglund2008finite} Let $X$ be a $CAT(0)$ cube complex and let $Y\subset X$ be a convex subcomplex. Then $Y$ is also combinatorially convex in the sense that any combinatorial geodesic segment joining two vertices in $Y$ is contained in $Y$.
\end{lem}

\begin{lem}[Lemma 2.10 of \cite{quasiflat}]
\label{2.6}
Let $X$ be a $CAT(0)$ cube complex of dimension $n$ and let $C_{1}$, $C_{2}$ be convex subcomplexes. Denote $\bigtriangleup=d(C_{1},C_{2})$. Put $Y_{1}=\{y\in C_{1}\mid d(y,C_{2})=\bigtriangleup\}$ and $Y_{2}=\{y\in C_{2}\mid d(y,C_{1})=\bigtriangleup\}$. Then:
\begin{enumerate}
\item $Y_{1}$ and $Y_{2}$ are not empty.
\item $Y_{1}$ and $Y_{2}$ are convex; $\pi_{C_{1}}$ maps $Y_{2}$ isometrically onto $Y_{1}$ and $\pi_{C_{2}}$ maps $Y_{1}$ isometrically onto $Y_{2}$; the convex hull of $Y_{1}\cup Y_{2}$ is isometric to $Y_{1}\times [0,\bigtriangleup]$.
\item $Y_{1}$ and $Y_{2}$ are subcomplexes, and $\pi_{C_{2}}|_{Y_{1}}$ is a cubical isomorphism with its inverse given by $\pi_{C_{1}}|_{Y_{2}}$.
\item There exists $A=A(\Delta,n,\epsilon)$ such that if $p_{1}\in C_{1}$, $p_{2}\in C_{2}$ and $d(p_{1},Y_{1})\geq \epsilon>0$, $d(p_{2},Y_{2})\geq \epsilon>0$, then 
\begin{equation}
\label{2.7}
d(p_{1}, C_{2})\geq \bigtriangleup + A\cdot d(p_{1},Y_{1});\ \ \  d(p_{2}, C_{1})\geq \bigtriangleup + A\cdot d(p_{2},Y_{2})\,.
\end{equation}
\end{enumerate}
\end{lem}

The above lemma implies $Y_{1}$ and $Y_{2}$ are coarsely equivalent, and $Y_1$ (or $Y_2$) is the coarse intersection of $C_1$ and $C_2$. We use $\inc(C_{1},C_{2})=(Y_{1},Y_{2})$ to describe this situation, where $\mathcal{I}$ stands for \textquotedblleft intersect\textquotedblright. 

Pick edge $e\subset X$. It turns out the parallel set $P_e$ is a convex subcomplex (actually $P_e$ is made of cubes which contain an edge parallel to $e$). There is a natural splitting $P_e=e\times h_e$. The \textit{hyperplane dual to $e$} is defined to be the subset of $P_e$ of form $\{m\}\times h_e$ where $m$ is the middle point of $e$. Each hyperplane separates $X$ into exactly two connected components. The closure of these components are called \textit{halfspaces}. The sets $Y_{1}$ and $Y_{2}$ in Lemma \ref{2.6} can be characterized in terms of hyperplanes. 

\begin{lem}
\label{2.9}
$($\cite[Lemma 2.6]{raagqi1}$)$ Let $X$, $C_{1}$, $C_{2}$, $Y_{1}$ and $Y_{2}$ be as in Lemma \ref{2.6}. Pick an edge $e$ in $Y_{1}$ $($or $Y_{2})$ and let $h$ be the hyperplane dual to $e$. Then $h\cap C_{i}\neq\emptyset$ for $i=1,2$. Conversely, if a hyperplane $h'$ satisfies $h'\cap C_{i}\neq\emptyset$ for $i=1,2$, then $$\inc(h'\cap C_{1},h'\cap C_{2})=(h'\cap Y_{1},h'\cap Y_{2})$$ and $h'$ comes from the dual hyperplane of some edge $e'$ in $Y_{1}$ $($or $Y_{2})$.
\end{lem}

The collection of halfspaces in $X$ contains enough information to recover $X$. More generally, we can view $X$ as a space with walls, and \cite{MR1347406,ed1998simplicit} introduces a way to construct a $CAT(0)$ cube complex from a given space with walls. There are several variations and developments of this construction (\cite{roller1998poc,nica2004cubulating,chatterji2005wall,hruska2014finiteness}). Here we follow the construction in \cite{roller1998poc}, see Sageev's notes \cite{sageevnotes}.

\begin{definition}[Definition 1.5 of \cite{sageevnotes}]
\label{2.3}
A \textit{pocset} is a partially ordered set with an involution $A\to A^{c}$ such that
\begin{enumerate}
\item $A\neq A^{c}$ and $A$ and $A^{c}$ are incomparable.
\item $A\le B$ implies $B^{c}\le A^{c}$.
\end{enumerate}
\end{definition}

Note that the collection of all closed halfspaces in a $CAT(0)$ cube complex forms a pocset. The partially order comes from inclusion of sets, and the involution is defined by map a halfspace $h$ to the unique halfspace which intersects $h$ along a hyperplane.

\begin{definition}[Definition 2.1 of \cite{sageevnotes}]
\label{2.4}
Let $P$ be a pocset. An \textit{ultrafilter} $U$ is a subset of $P$ such that
\begin{enumerate}
\item For all pairs $\{A,A^{c}\}$ in $P$, precisely one of them is in $U$.
\item If $A\in U$ and $A\le B$, then $B\in U$.
\end{enumerate}
\end{definition}

For example, pick a vertex $p$ in a $CAT(0)$ cube complex $X$, then the collection of closed halfspaces in $X$ that contains $p$ forms an ultrafilter. Note that if $U$ is an ultrafilter and $A$ is minimal in $U$ with respect to the partial order on $P$, then $(U\setminus\{A\})\cup \{A^{c}\}$ is also an ultrafilter.

\begin{thm}[\cite{roller1998poc}]
\label{2.5}
If $P$ is a finite pocset, then there is a $CAT(0)$ cube complex $X$ such that its vertices are in 1-1 correspondence with ultrafilters of $P$ and two vertices $U$ and $U'$ are connected by an edge if and only if $$U'=(U\setminus\{A\})\cup \{A^{c}\}$$ for some $A$ minimal in $U$. Moreover, there is a natural pocset isomorphism from $P$ to the pocset of halfspaces in $X$.
\end{thm}

If $P$ is infinite, then similar conclusions hold under the additional assumptions that $P$ is discrete and of finite width, see \cite{roller1998poc,sageevnotes}. However, in this paper we only need the case when $P$ is finite.

\subsection{Basics about RAAGs}
\label{basics about raag}
Pick a finite simplicial graph $\Gamma$, and let $G(\Gamma)$ be the right-angled Artin group (RAAG) with defining graph $\Gamma$. Let $S(\Gamma)$ be the Salvetti complex (\cite[Section 2.6]{charney2007introduction}) of $G(\Gamma)$, which is a non-positively curved cube complex whose 2-skeleton is the presentation complex of $G(\Ga)$. Denote the universal cover of $S(\Gamma)$ by $X(\Gamma)$. Pick a standard generating set $S$ for $G(\Gamma)$, and we label the 1-cells in $S(\Gamma)$ by elements in $S$ and choose an orientation for each 1-cell in $S(\Gamma)$. This lifts to orientation and labelling of edges in $X(\Gamma)$ which are invariant under the action $G(\Ga)\acts X(\Ga)$. As $S(\Gamma)$ only has one vertex, we can identify $G(\Ga)$ as the 0-skeleton of $X(\Ga)$.

Let $\Gamma'\subset\Gamma$ be a full subgraph. Then the images of the embeddings $G(\Gamma')\to G(\Gamma)$ and $S(\Gamma')\to S(\Gamma)$ are called \textit{standard subgroup} (of type $\Gamma'$) and \textit{standard subcomplex} (of type $\Gamma'$) respectively. \textit{Standard subcomplexes} of $X(\Gamma)$ are lifts of standard subcomplexes of $S(\Gamma)$. When $\Gamma'$ is a complete subgraph, $G(\Gamma')$ is called a \textit{standard abelian subgroup}, $S(\Gamma')$ is called a \textit{standard torus}, and lifts of $S(\Gamma')$ are called \textit{standard flats}. One dimensional standard flats are also called \textit{standard geodesics}. As we are identifying $G(\Gamma)$ as the 0-skeleton of $X(\Gamma)$, sometime we will slightly abuse the notation by referring to a left coset of a standard abelian subgroup (resp. standard $\mathbb Z$ subgroup) of $G(\Gamma)$ as a standard flat (resp. standard geodesic).

\begin{definition}
\label{notation}
For every edge $e\in X(\Gamma)$, there is a vertex in $\Gamma$ which shares the same label as $e$. We denote this vertex by $V_{e}$. If $K\subset X(\Gamma)$ is a subcomplex ($K$ does not need to be a standard subcomplex), we define $V_{K}$ to be $\{V_{e}\mid e$ is an edge in $K\}$ and $\Gamma_{K}$ to be the full subgraph spanned by $V_{K}$. The subgraph $\Gamma_K$ is called the \textit{support} of $K$. Pick a vertex $v\in X(\Gamma)$ and a full subgraph $\Gamma'\subset\Gamma$, we denote the unique standard subcomplex with support $\Gamma'$ that contains $v$ by $K(v,\Gamma')$. 
\end{definition}

The following two results are strengthened versions of Lemma \ref{2.6} and Lemma \ref{2.9} in the case of coarse intersection of two standard subcomplexes:

\begin{lem}
\label{3.1}
$($\cite[Lemma 3.1]{raagqi1}$)$ Let $\Gamma$ be a finite simplicial graph and let $K_{1}$, $K_{2}$ be two standard subcomplexes of $X(\Gamma)$. If $(Y_{1},Y_{2})=\inc (K_{1},K_{2})$, then $Y_{1}$ and $Y_{2}$ are also standard subcomplexes.
\end{lem}

We can compute the supports of $Y_1$ and $Y_2$ as follows.
\begin{lem}
\label{3.2}
$($\cite[Corollary 3.2]{raagqi1}$)$ Let $K_{1},K_{2},Y_{1}$ and $Y_{2}$ be as above. 
\begin{enumerate}
\item Let $h$ be a hyperplane separating $K_{1}$ and $K_{2}$ and let $e$ be an edge dual to $h$. Then $V_{e}\in V^{\perp}_{Y_{1}}=V^{\perp}_{Y_{2}}$ (see Definition \ref{notation} for relevant notations).
\item A vertex $v\in\Gamma$ satisfies $v\in V_{Y_{1}}$ if and only if
\begin{enumerate}
\item $v\in V_{K_{1}}\cap V_{K_{2}}$.
\item For any hyperplane $h'$ separating $K_{1}$ from $K_{2}$ and any edge $e'$ dual to $h'$,  we have $d(v,V_{e'})=1$.
\end{enumerate}
\end{enumerate}
\end{lem}

The proof roughly goes as follows. Pick vertex $x\in K_1$ and let $y$ be the vertex in $K_2$ closest to $x$. Let $\ell$ be a combinatorial geodesic joining $x$ and $y$. Then there is a combinatorial embedding $Y_1\times \ell\hookrightarrow X$. Note that parallel edges has the same label. Two edges span a square if and only if their label are adjacent in $\Ga$. Thus the label of each edge in $\ell$ is adjacent to the label of every edge in $Y_1$. If $h$ separates $K_1$ and $K_2$, then $h$ must intersect $\ell$. Then (1) follows. Suppose $v\in V_{Y_1}$. Since $Y_1$ and $Y_2$ are parallel, then $V_{Y_1}=V_{Y_2}$, thus (2a) follows. (2b) is a consequence of (1).

\begin{lem}
	\label{coarse contain of standard subcomplexes}
Let $K_1$ and $K_2$ be two standard subcomplexes of $X(\Gamma)$. Then
\begin{enumerate}
	\item $K_1$ is coarsely contained in $K_2$ (cf. Section \ref{subsec_notation}) if and only if there is a standard subcomplex $K'_2\subset K_2$ such that $K_1$ and $K'_2$ are parallel.
	\item $K_1$ is coarsely equivalent to $K_2$ if and only if $K_1$ and $K_2$ are parallel.
\end{enumerate}
\end{lem}

\begin{proof}
We only prove (1). (2) is similar. It suffices to prove the only if direction of (1). Let $\inc(K_1,K_2)=(J_1,J_2)$. It suffices to prove $J_1=K_1$. By Lemma \ref{2.6} (4), $J_1$ and $K_1$ have finite Hausdorff distance, moreover, Lemma \ref{3.1} implies $J_1$ is a standard subcomplex of $K_1$. If $J_1\subsetneq K_1$, then their supports satisfy $\Gamma_{J_1}\subsetneq \Gamma_{K_1}$. Pick vertex $\bar{v}\in \Gamma_{K_1}\setminus \Gamma_{J_1}$ and let $\ell\subset K_1$ be a standard geodesic line with its support $=\{\bar{v}\}$. Then Lemma \ref{3.2} (2a) implies $\inc(J_1,\ell)$ is a pair of points, which implies $\ell$ is not contained in a bounded neighborhood of $K_1$ by Lemma \ref{2.6} (4). This contradicts that $J_1$ and $K_1$ have finite Hausdorff distance. So we must have $J_1=K_1$.
\end{proof}

Now we recall extension graph and extension complex defined by Kim and Koberda \cite{kim2013embedability}, which will be the key quasi-isometry invariants used in this paper. The combinatorial structure of extension complexes is instrumental for studying quasi-isometries between RAAGs. It is worth mentioning that it was known before that the extension graph is a commensurability invariant in certain classes of RAAGs \cite{kim2014geometry}. There is also a related graph called contact graph \cite{hagen2014weak} which is quasi-isometric to the extension graph as proved in \cite{kim2014geometry}, but with a quite different combinatorial structure.

Let $\mathcal{P}(\Gamma)$ be the \emph{extension complex} of $\Gamma$, which is the flag complex of the extension graph introduced in \cite{kim2013embedability}. We give an alternative but equivalent definition here. The vertices of $\mathcal{P}(\Gamma)$ are in 1-1 correspondence with the parallel classes of standard geodesics in $X(\Gamma)$ (two standard flats are in the same parallel class if they are parallel). Two distinct vertices $v_{1},v_{2}\in\mathcal{P}(\Gamma)$ are connected by an edge if for $i=1,2$, there is a standard geodesic $\ell_{i}$ in the parallel class associated with $v_{i}$ such that $\ell_{1}$ and $\ell_{2}$ span a standard 2-flat. This definition is equivalent to the definition in \cite{kim2013embedability} as explained in \cite[Lemma 4.2]{raagqi1}.

Note that edges in the same standard geodesics of $X(\Gamma)$ have the same label, and edges in parallel standard geodesics also have the same label. This induces a well-defined labeling of vertices in $\mathcal{P}(\Gamma)$ by vertices of $\Gamma$.  There is a label-preserving simplicial map $\pi:\mathcal{P}(\Gamma)\to F(\Gamma)$, where $F(\Gamma)$ is the flag complex of $\Gamma$. Moreover, since $G(\Gamma)\curvearrowright X(\Gamma)$ by label-preserving cubical isomorphisms, we obtain an induced action $G(\Gamma)\curvearrowright\mathcal{P}(\Gamma)$ by label-preserving simplicial isomorphisms.

Note that each complete subgraph in the 1-skeleton of $\mathcal{P}(\Gamma)$ gives rise to a collection of mutually orthogonal standard geodesics lines. Thus there is a 1-1 correspondence between the $(k-1)$-simplexes in $\mathcal{P}(\Gamma)$ and parallel classes of standard $k$-flats in $X(\Gamma)$ (\cite[Section 4.1]{raagqi1}). For standard flat $F\subset X(\Gamma)$, we denote the simplex in $\mathcal{P}(\Gamma)$ associated with the parallel class containing $F$ by $\Delta(F)$. For a standard subcomplex $K\subset X(\Gamma)$, define $\Delta(K):=\cup_{\lambda\in\Lambda}\Delta(F_{\lambda})$, here $\{F_{\lambda}\}_{\lambda\in\Lambda}$ is the collection of standard flats in $K$.

\begin{lem}
	\label{boundary of std subclex}
Let $K_1$ and $K_2$ be two standard subcomplexes. Then 
\begin{enumerate}
	\item $\Delta(K_1)=\Delta(K_2)$ if and only if $K_1$ and $K_2$ are parallel. 
	\item $\Delta(K_1)\subset\Delta(K_2)$ if and only if $K_1$ is coarsely contained in $K_2$.
\end{enumerate}
\end{lem}

\begin{proof}
It suffices to (1). As (2) follows from (1) and Lemma \ref{coarse contain of standard subcomplexes} (1). Suppose $K_1$ and $K_2$ are parallel. Let $F\subset K_1$ be a standard flat. Then $F_1$ is coarsely contained in $K_2$, hence by Lemma \ref{coarse contain of standard subcomplexes}, there exists a standard flat $F_2\subset K_2$ which is parallel to $F_1$. Thus $$\Delta(F_1)=\Delta(F_2)\subset \Delta(K_2)$$ and we deduce that $\Delta(K_1)\subset\Delta(K_2)$. Similarly we know the inclusion on the other direction, hence $\Delta(K_1)=\Delta(K_2)$.

Now we assume $\Delta(K_1)=\Delta(K_2)$. Then each standard flat of $K_1$ is coarsely contained in $K_2$. Let $\inc(K_1,K_2)=(J_1,J_2)$. Then each standard flat of $K_1$ is coarsely contained in $J_1$ by Lemma \ref{2.6} (4). By the same proof as Lemma \ref{coarse contain of standard subcomplexes}, we know that if $J_1\subsetneq K_1$, then there is a standard geodesic line in $K_1$ which is not coarsely contained in $J_1$. Thus $K_1=J_1$. Similarly we can prove $K_2=J_2$. Hence $K_1$ and $K_2$ are parallel.
\end{proof}

Given arbitrary vertex $p\in X(\Gamma)$, one can obtain a simplicial embedding $$i_{p}:F(\Gamma)\to\mathcal{P}(\Gamma)$$ by considering the collection of standard flats passing through $p$ (here $F(\Gamma)$ denotes the flag complex of $\Gamma$). We will denote the image of $i_{p}$ by $(F(\Gamma))_{p}$. Note that $\pi\circ i_{p}$ is the identity map, which implies the following lemma.
\begin{lem}
\label{isometric embedding}
The map $i_p$ is an isometric embedding with respect to the combinatorial distance between vertices of $F(\Ga)$.
\end{lem}

Now we look at the outer automorphism group $\out(G(\Gamma))$ of $G(\Ga)$. By \cite{servatius1989automorphisms,laurence1995generating}, $\out(G(\Gamma))$ is generated by the following four types of elements (we identify the vertex set of $\Gamma$ with a standard generating set of $G(\Gamma)$): 
\begin{enumerate}
\item Given vertex $v\in\Gamma$, sending $v\to v^{-1}$ and fixing all other vertices.
\item Graph automorphisms of $\Gamma$.
\item If $lk(w)\subset St(v)$ for vertices $w,v\in\Gamma$, sending $w\to wv$ and fixing all other vertices induces a group automorphism. It is called a \textit{transvection}. When $d(v,w)=1$, it is an \textit{adjacent transvection}, otherwise it is a \textit{non-adjacent transvection}.
\item Suppose $\Gamma\setminus St(v)$ is disconnected. Then one obtains a group automorphism by picking a connected component $C$ and sending $w\to vwv^{-1}$ for each vertex $w\in C$. It is called a \textit{partial conjugation}.
\end{enumerate}

\subsection{Special subgroups of RAAGs}
\label{special subgroup}
We first consider the special case $G(\Ga)\cong \Z^{n}$. Pick a finite rectangle $K\subset\Z^{n}$. Then the finite index subgroup $H=\oplus_{i=1}^{n}n_i\Z$, where $n_i$'s are the number of vertices along the edges of the rectangle, has $K$ as its fundamental domain. This can be naturally generalized to all RAAGs in the following way. Let $K\subset X(\Ga)$ be a compact convex subcomplex. Let $\{\ell_{i}\}_{i=1}^{s}$ be a maximal collection of standard geodesics such that $\ell_{i}\cap K\neq\emptyset$ for all $i$ and $\Delta(\ell_{i})\neq \Delta(\ell_{j})$ for any $i\neq j$. We consider the left action $G(\Ga)\acts X(\Ga)$. For each $i$, let $g_i\in G(\Ga)$ be the unique element that translates $\ell_i$ towards the positive direction with translation length $=1$ (recall that we have oriented edges of $X(\Ga)$ in a $G(\Ga)$-invariant way). Let $n_{i}=|v(K\cap \ell_{i})|$. 

\begin{thm}(\cite[Section 6.1]{raagqi1})
\label{special subgroup theorem}
Let $G\le G(\Ga)$ be the subgroup generated by $\{g^{n_{i}}_{i}\}_{i=1}^{s}$. Then the following holds.
\begin{enumerate}
\item The subcomplex $K$ is a \textquotedblleft fundamental domain\textquotedblright\ for $G$ in the sense that for $g_1,g_2\in G$, $g_1 K\cap g_2 K\neq \emptyset$ if and only if $g_1=g_2$. Moreover, the $G$-orbit of $K$ cover the $0$-skeleton of $X(\Ga)$. Thus $|G(\Ga):G|=$ the number of vertices in $K$.
\item Let $\Ga'$ be the 1-skeleton of the full subcomplex of $\P(\Ga)$ spanned by $\{\Delta(\ell_{i})\}_{i=1}^{s}$. Then $G$ is isomorphic to the RAAG $G(\Ga')$.
\end{enumerate}
\end{thm}

Such a group $G$ is called a \textit{special subgroup of $G(\Ga)$ (associated with $K$)}. Note that the definition of special subgroups implicitly depend on the choice of standard generators of $G(\Ga)$ (we can think $G(\Ga)$ as a fixed set, and different choices of standard generators give different ways of building $X(\Ga)$ from $G(\Ga)$). A subgroup is \textit{$S$-special} if it is special with respect to the standard generating set $S$. In most part of the paper, we fix a standard generating set, so there will be no confusion.

Alternatively, $G$ can be characterized as the fundamental group of the canonical completion (\cite{MR2377497}) of the local isometry $K\hookrightarrow X(\Ga)\to S(\Ga)$. However, we will not need this fact.

Let $G(\Ga)\cong F_2$, the free group with two generators. We take $K$ to be an edge in $X(\Ga)$. Then the associated special subgroup $G$ is isomorphic to $F_3$. In this case, if we collapse all the $G$-translations of $K$ in $X(\Ga)$, then the resulting space is isomorphic to the Cayley graph for $F_3$. Note that $F_3$ is a special subgroup of $F_2$ in the sense defined above. This specific example can be generalized to all RAAGs in the following way.

Recall that a cellular map between cube complexes is \textit{cubical} (see \cite{caprace2011rank}) if its restriction $\sigma\to\tau$ between cubes factors as $\sigma\to\eta\to\tau$, where the first map $\sigma\to\eta$ is a natural projection onto a face of $\sigma$ and the second map $\eta\to\tau$ is an isometry.

\begin{thm}(\cite[Lemma 6.18]{raagqi1})
\label{retraction map}
Let $G$, $\Ga$ and $\Ga'$ be as in Theorem \ref{special subgroup theorem}. There is a surjective cubical map $q:X(\Ga)\to X(\Ga')$ such that
\begin{enumerate}
\item the map $q$ sends standard flats onto standard flats. Moreover, each standard flat in $X(\Ga')$ is the $q$-image of some standard flat in $X(\Ga)$;
\item pick a vertex $x'\in X(\Ga')$, then $q^{-1}(x')=g\cdot K$ for some element $g\in G$ (here $K$ is the compact convex subcomplex of $X(\Gamma)$ defined above).
\item the map $q$ is $G$-equivariant. In particular, $q$ is a quasi-isometry.
\end{enumerate}
\end{thm}

If we identify each left coset of standard $\mathbb Z$ subgroup of $G(\Ga)$ and $G(\Ga')$ with a copy of $\mathbb Z$ (the identification is well-defined up to a translation of $\mathbb Z$), then Theorem~\ref{retraction map} implies that the restriction of $q$ to a left coset of standard $\mathbb Z$ subgroup takes form $$q(x)=\lfloor x/d\rfloor+r$$ for some integers $d$ and $r$ which depends on the $\mathbb Z$-coset. To see, denote the $\mathbb Z$-coset by $\ell$ and let $\ell'=q(\ell)$. Theorem~\ref{retraction map} (3) implies $q|_{\ell}$ is $H$-equivariant where $H$ is the $G$-stabilizer of $\ell'$. As $H$ acts transitively on $\ell'$ and $(q|_{\ell})^{-1}(y)$ is an interval for any $y\in \ell'$, the formula $$q|_{\ell}(x)=\lfloor x/d\rfloor+r$$ follows.

We claim the map $q$ in Theorem~\ref{retraction map} maps parallel standard geodesics to parallel standard geodesics. Indeed, this follows from that $q$ is cubical if two parallel geodesics are contained in the same standard flat. In general, we can interpolate any two parallel standard geodesics in $X(\Gamma)$ with a chain of parallel standard geodesics in $X(\Gamma)$ such that adjacent members in the chain are contained in the same standard flats. Then the claim follows from Theorem~\ref{retraction map} (1). 

Thus $q$ induces a map $$q_{\ast}:(\P(\Ga))^{(0)}\to (\P(\Ga'))^{(0)}.$$ Moreover, as $q$ is quasi-isometry, it cannot map a pair of standard geodesics without finite Hausdorff distance to a pair of standard geodesics with finite Hausdorff distance. Thus $q_{\ast}$ is injective. From Theorem~\ref{retraction map} (1) and the definition of $\P(\Ga)$, we also know that $q_{\ast}$ send adjacent vertices to adjacent vertices. Now the moreover part of Theorem~\ref{retraction map} (1) implies that $q_*$ is actually a bijection and extends to a simplicial isomorphism $q_{\ast}:\P(\Ga)\to \P(\Ga')$. 

The following lemma can be viewed as a form of converse to Theorem~\ref{retraction map}.
\begin{lem}
	\label{lem:recognizing special subgroup}
Let $G(\Gamma)$ and $G(\Gamma')$ be two RAAGs with a homomorphism $i:G(\Gamma')\to G(\Gamma)$. Suppose there exists a surjective cubical map $q:X(\Gamma)\to X(\Gamma')$ such that
\begin{enumerate}
	\item $q$ is a $G(\Gamma')$-equivariant, where the action $G(\Gamma')\curvearrowright X(\Gamma)$ is induced by $i$ and the left action $G(\Gamma)\curvearrowright X(\Gamma)$;
	\item the $q$-image of a standard flat in $X(\Gamma)$ is a standard flat in $X(\Gamma')$ and the restriction of $q$ to a left coset of standard $\mathbb Z$ subgroup takes form $$q(x)=\lfloor x/d\rfloor+r$$ for some integers $d\ge 1$ and $r$ which depends on the $\mathbb Z$-coset;
	\item $q$ induces a simplicial isomorphism $q_{\ast}:\P(\Ga)\to \P(\Ga')$;
	\item $q^{-1}(x')$ is bounded for some vertex $x'\in X(\Gamma')$.
\end{enumerate}
Then $i$ is injective and $i(G(\Gamma'))$ is a special subgroup of $G(\Gamma)$.
\end{lem}

\begin{proof}
Take a vertex $x'\in X(\Gamma')$ representing the identity element of $G(\Gamma')$. It follows from the proof of \cite[Theorem 5.12 (2)]{raagqi1} that $K=q^{-1}(x')$ is a convex subcomplex (assumption (2) and (3) as above is used here). As $G(\Gamma')$ acts transitively on the vertex set of $X(\Gamma')$, by (1) and (4), $q^{-1}(y)$ is bounded for any vertex $y\in X(\Gamma')$. Thus $K$ is bounded, hence compact. 
Let $\{\ell_{i}\}_{i=1}^{s}$, $g_i\in G(\Ga)$ and $n_{i}=|v(K\cap \ell_{i})|$ be as in the beginning of Section~\ref{special subgroup}.
Let $\ell'_i=q(\ell_i)$ and let $v_i$ be the generator of $G(\Gamma')$ which acts on $\ell'_i$ by translation. Note that $v_i\neq v_j$ if $i\neq j$ (as $\Delta(\ell_{i})\neq \Delta(\ell_{j})$ for any $i\neq j$ and $q_{\ast}$ is an isomorphism). By assumption (1), $i(v_i)=g^{n_i}_i$. Let $H$ be the standard subgroup of $G(\Gamma')$ generated by $\{v_i\}_{i=1}^s$. By \cite[Lemma 6.3 and 6.4]{raagqi1}, $i|_{H}$ is injective with finite index image. The lemma would follow if we know $H=G(\Gamma')$. If this is not true, then $H$ is of infinite index in $G(\Gamma')$. Take point $x'\in X(\Gamma')$ and $x\in X(\Gamma)$ with $q(x)=x'$. On one hand, as $i(H)$ is finite index in $G(\Gamma)$, the orbit $i(H)\cdot x$ is cobounded in $X(\Gamma)$ (i.e. $X(\Gamma)$ is contained in a finite neighborhood of $i(H)\cdot x$), then we deduce from that $q$ is equivariant and Lipschitz that $H\cdot x'$ is cobounded in $X(\Gamma')$. On the other hand, this is impossible as $H$ is of infinite index in $G(\Gamma)$. Thus we must have $H=G(\Gamma')$ and the lemma is proved.
\end{proof}

\subsection{Coarse invariants for RAAGs}
\label{subsec_qi invariant}
Here we summarize and generalize some results from \cite{raagqi1}. 

Note that every join decomposition $\Ga=\Ga_1\circ\Ga_2$ induces a direct sum decomposition $G(\Ga)=G(\Ga_1)\oplus G(\Ga_2)$. $G(\Ga)$ or $\Ga$ is called \textit{irreducible} if $\Ga$ does not allow a non-trivial join decomposition. There is a well-defined De Rham decomposition of $X(\Gamma)$ induced by the join decomposition of $\Gamma$, which is stable under quasi-isometries.

\begin{thm}
\label{2.13}
$($\cite[Theorem 2.9]{raagqi1}$)$ Given $X=X(\Gamma)$ and $X'=X(\Gamma')$, let $$X=\Bbb R^{n}\times\prod_{i=1}^{k}X(\Gamma_{i})$$ and $$X'=\Bbb R^{n'}\times\prod_{j=1}^{k'}X(\Gamma'_{j})$$ be the corresponding De Rahm decomposition. If $\phi:X\to X'$ is an $(L,A)$ quasi-isometry, then $n=n'$, $k=k'$ and there exist constants $$L'=L'(L,A),A'=A'(L,A),D=D(L,A)$$ such that after re-indexing the factors in $X'$, we have $(L',A')$ quasi-isometry $$\phi_{i}: X(\Gamma_{i})\to X(\Gamma'_{i})$$ so that $$d(p'\circ\phi, \prod_{i=1}^{k}\phi_{i}\circ p)<D,$$ where $$p: X\to \prod_{i=1}^{k}X(\Gamma_{i})$$ and $$p':X'\to \prod_{i=1}^{k}X(\Gamma'_{i})$$ are the projections.
\end{thm}

Note that when there is no Euclidean De Rham factor, the above theorem basically says that any quasi-isometry between $X$ and $X'$ is a product of quasi-isometries between their factors.

We are particularly interested in those standard subcomplexes that are stable under quasi-isometries.
\begin{definition}
\label{3.38}
A subgraph $\Gamma_{1}\subset\Gamma$ is \textit{stable in $\Gamma$} if $\Gamma_{1}$ is a full subgraph and for any standard subcomplex $K\subset X(\Gamma)$ with $\Gamma_{K}=\Gamma_{1}$ and $(L,A)$-quasi-isometry $q:X(\Gamma)\to X(\Gamma')$, there exists $D=D(L,A,\Gamma_{1},\Gamma)>0$ and standard subcomplex $K'\subset X(\Gamma')$ such that the Hausdorff distance $$d_{H}(q(K),K')<D.$$ A standard subcomplex $K\subset X(\Gamma)$ is \textit{stable} if its support is a stable subgraph of $\Gamma$.
\end{definition}

\begin{remark}
We caution the reader that the above definition of stable standard subcomplex is different from ``stable subgroups'' defined in	\cite{durham2015convex}. In particular, a stable subcomplex in the sense of Definition~\ref{3.38} does not have to be Gromov hyperbolic, and quasi-geodesics connecting points inside a stable standard subcomplex $Y$ with uniform quasi-isometric constants do not have to stay in a uniform neighborhood of $Y$.
\end{remark}

It is clear that the intersection of two stable subgraphs is still a stable subgraph. See \cite[Section 3.2]{raagqi1} for more properties about stable subgraphs. In this paper, we will use the following two properties repeatedly:

\begin{lem}\cite[Lemma 3.24]{raagqi1}
\label{3.50}
Let $\Gamma$ be a finite simplicial graph and pick stable subgraphs $\Gamma_{1},\Gamma_{2}$ of $\Gamma$. Let $\bar{\Gamma}$ be the full subgraph spanned by $V$ and $V^{\perp}$ where $V$ is the vertex set of $\Gamma_1$. If $\Gamma_{2}\subset\bar{\Gamma}$, then the full subgraph spanned by the vertices in $\Gamma_{1}\cup\Gamma_{2}$ is stable in $\Gamma$. 
\end{lem}

\begin{lem}
\label{4.7}
Suppose there is no non-adjacent transvection in $\out(G(\Gamma))$. Then every maximal clique subgraph of $\Gamma$ is stable.
\end{lem}

A slightly weaker reformulation of the above lemma is the following. If each automorphism of $G(\Ga)$ preserves maximal standard flats up to finite Hausdorff distance, then so does each quasi-isometry of $G(\Ga)$.

\begin{proof}
Let $\Gamma_{1}\subset\Gamma$ be a maximal clique. By \cite[Theorem 3.35]{raagqi1}, it suffices to prove for any vertices $v\in \Gamma_{1}$ and $w\in \Gamma$, $v^{\perp}\in St(w)$ implies $w\in\Gamma_{1}$. Note that $v^{\perp}\in St(w)$ implies $w\in v^{\perp}$ since there is no non-adjacent transvection. Then $w$ and vertices of $\Gamma_{1}$ span a clique in $\Gamma$, thus $w\in\Gamma_{1}$ by the maximality of $\Gamma_{1}$.
\end{proof}

Let $\Delta$ be the map sending standard subcomplexes of $X(\Gamma_{i})$ to subcomplexes of $\mathcal{P}(\Gamma_{i})$ defined in Section \ref{basics about raag}. Let $\S(\P(\Gamma))$ be the collection of subcomplexes of $\P(\Gamma)$ which are $\Delta$-images of stable standard subcomplexes in $X(\Gamma)$ (we assume the empty set is also in $\S(\P(\Gamma))$). 

\begin{lem}
	\label{iso}
Any quasi-isometry $q:X(\Gamma_1)\to X(\Gamma_2)$ induces a well-defined bijection $\tilde{q}_{\ast}:\S(\mathcal{P}(\Gamma_{1}))\to\S(\mathcal{P}(\Gamma_{2}))$.
\end{lem}

\begin{proof}
Pick element $M\in \S(\mathcal{P}(\Gamma_{1}))$ and let $K_1\subset X(\Ga_1)$ be a standard subcomplex such that $\Delta(K_1)=M$. Then we define $\tilde{q}_{\ast}$ to be $\Delta(K_2)$, where $K_2\subset X(\Gamma_2)$ is a standard subcomplex which is at finite Hausdorff distance from $q(K_1)$. Note that $K_2$ is also stable, so $$\Delta(K_2)\in \S(P(\Ga_2)).$$ If we choose another standard subcomplex $K'_1\subset X(\Ga_1)$ such that $\Delta(K'_1)=M$ and choose another standard subcomplex $K'_2\subset X(\Gamma)$ which is at finite Hausdorff distance from $q(K'_1)$, then Lemma \ref{boundary of std subclex} implies $K_1$ and $K'_1$ are parallel. Hence $K_2$ and $K'_2$ are coarsely equivalent. Then Lemma \ref{coarse contain of standard subcomplexes} implies $K_2$ and $K'_2$ are parallel and Lemma \ref{boundary of std subclex} implies $\Delta(K_2)=\Delta(K'_2)$. Thus $\tilde{q}_{\ast}$ is well-defined. By considering the quasi-isometry inverse of $q$, we know $\tilde{q}_{\ast}$ is a bijection.
\end{proof}

\begin{lem}
	\label{intersection}
The set $\S(\P(\Ga_1))$ is closed under intersection. Moreover, for $M,M'\in \S(\P(\Ga_1))$, we have
\begin{enumerate}
	\item $\tilde{q}_{\ast}(M)\cap \tilde{q}_{\ast}(M')=\tilde{q}_{\ast}(M\cap M')$.
	\item $M\subset M'$ if and only if $\tilde{q}_{\ast}(M)\subset \tilde{q}_{\ast}(M')$.
\end{enumerate}
\end{lem}

\begin{proof}
First we show $M\cap M'\in \S(P(\Ga_1))$. For $i=1,2$, let $K_1$ and $K'_1$ be stable standard subcomplexes of $X(\Ga_1)$ such that $\Delta(K_1)=M$ and $\Delta(K'_1)=M'$. Let $\inc(K_1,K'_1)=(J_1,J'_1)$. Since we already know $J_1$ is a standard subcomplex by Lemma \ref{3.1}, it remains to show $\Delta(J_1)=M\cap M'$ and $J_1$ is stable. 

Since $J_1$ is coarsely contained in $K'_1$ and $K_1$, it follows from Lemma \ref{boundary of std subclex} (2) that $\Delta(J_1)\subset M\cap M'$. Now pick simplex $s\subset M\cap M'$ and let $F_s$ be a standard flat in $X(\Ga_1)$ with $\Delta(F_s)=s$. Then $F_s$ is coarsely contained in $K_1$ and $K'_1$. By Lemma \ref{2.6} (4), for $R$ large enough, we have
\begin{equation}
\label{e:1}
d_{H}(J_1,N_R(K_1)\cap N_R(K'_1))<\infty,
\end{equation}
where $d_H$ is the Hausdorff distance. By taking $R$ large enough, we conclude that $F_s$ is coarsely contained in $J_1$, thus $s\subset \Delta(J_1)$ by Lemma \ref{boundary of std subclex} (2) and $$M\cap M'\subset \Delta(J_1).$$ Now we show $J_1$ is stable. Let $K'_2$ and $K_2$ be standard subcomplexes in $X(\Ga_2)$ which are Hausdorff close to $q(K'_1)$ and $q(K_1)$ respectively. Let $$\inc(K_2,K'_2)=(J_2,J'_2).$$ Then $$d_{H}(J_2,N_R(K_2)\cap N_R(K'_2))<\infty$$ by Lemma \ref{2.6} (4). This, together with (\ref{e:1}) imply $$d_H(q(J_1),J_2)<\infty.$$ Since $J_2$ is a standard subcomplex by Lemma \ref{3.1}, we know $J_1$ is a stable standard subcomplex. Moreover, $$\tilde{q}_{\ast}(M)\cap \tilde{q}_{\ast}(M')=\Delta(K_2)\cap\Delta(K'_2)=\Delta(J_2)=\tilde{q}_{\ast}(\Delta(J_1))=\tilde{q}_{\ast}(M\cap M').$$

It remains to prove (2). The only if direction follows from (1) and the if direction follows by considering the quasi-isometry inverse of $q$.
\end{proof}

A subcomplex of $\P(\Ga_1)$ (or $\P(\Ga_2)$) is \emph{stable} if it is a member of $\S(\P(\Ga_1))$ (or $\S(\P(\Ga_2))$). Then the map $\tilde{q}_{\ast}$ defined in Lemma \ref{iso} induces a 1-1 correspondence between stable $k$-simplexes in $\P(\Ga_1)$ and stable $k$-simplexes in $\P(\Ga_2)$.

The following result is the starting point of this paper.

\begin{thm}
\label{4.8}
Let $q:X(\Gamma_{1})\to X(\Gamma_{2})$ be a quasi-isometry. Suppose $\out(G(\Gamma_{1}))$ does not contain any non-adjacent transvection. Then there exists a simplicial embedding $q_{\ast}:\mathcal{P}(\Gamma_{1})\to\mathcal{P}(\Gamma_{2})$ such that for any stable simplex $s\subset \P(\Ga)$, we have
\begin{equation}
\label{coincidence}
q_{\ast}(s)=\tilde{q}_{\ast}(s).
\end{equation}
If we also assume $\out(G(\Gamma_{2}))$ does not contain any non-adjacent transvection, then $q_{\ast}$ is a simplicial isomorphism.
\end{thm}

\begin{proof}
For $i=1,2$, let $\mathcal{V}^{k}_{i}$ be the collection of vertices of $\P(\Ga_i)$ which are inside some stable $m$-simplex of $\P(\Ga)$ for $0\le m\le k$. By Lemma \ref{4.7}, $\mathcal{V}^{n-1}_{1}$ is exactly the 0-skeleton of $\mathcal{P}(\Gamma_{1})$, where $n=\dim(X(\Gamma_{1}))=\dim(X(\Gamma_{2}))$.

We first construct $q_{\ast}$ from the 0-skeleton of $\mathcal{P}(\Gamma_{1})$ to the 0-skeleton of $\mathcal{P}(\Gamma_{2})$ inductively as follows: define $q_{\ast}(v)=\tilde{q}_{\ast}(v)$ for $v\in\mathcal{V}^{0}_{1}$ and suppose we have already defined $q_{\ast}$ on $\mathcal{V}^{k}_{1}$ such that 
\begin{center}
($\ast$) For any stable simplex $s\subset \P(\Ga_1)$, $q_{\ast}$ is a bijection from $\mathcal{V}^{k}_{1}\cap s$ to $\mathcal{V}^{k}_{2}\cap \tilde{q}_{\ast}(s)$.
\end{center}
By Lemma~\ref{iso} and Lemma~\ref{intersection}, $q_{\ast}$ restricted on $\mathcal{V}^{0}_{1}$ satisfies $(\ast)$ for $k=0$. 

Now we define $q_{\ast}$ on $\mathcal{V}^{k+1}_{1}$. Pick a stable $(k+1)$-simplex $s^{k+1}\subset\P(\Ga_1)$. If all vertices of $s^{k+1}$ belong to $\mathcal{V}^{k}_{1}$, then we do not need to do anything. Otherwise we pick vertex $v\in s^{k+1}\setminus\mathcal{V}^{k}_{1}$. Note that $s^{k+1}$ is the only stable $(k+1)$-simplex of $\P(\Ga_1)$ that contains $v$ (if there is a distinct stable $(k+1)$-simplex $s^{k+1}_{1}\subset\P(\Ga_1)$ with $v\in s^{k+1}_{1}$, then $v\in s^{k+1}_{1}\cap s^{k+1}$, which is a stable simplex of dimension $\le k$ by Lemma \ref{intersection}. This implies $v\in \mathcal{V}^{k}_{1}$, which is a contradiction). By inductive assumption, vertices in $s^{k+1}\setminus\mathcal{V}^{k}_{1}$ and vertices in $\tilde{q}_{\ast}(s^{k+1})\setminus\mathcal{V}^{k}_{2}$ have the same cardinality, so we can choose an arbitrary bijection between them. 

Now we have $q_{\ast}$ defined on $\mathcal{V}^{k+1}_{1}$ and it remains to verify $(\ast)$. Given a stable simplex $s\subset\P(\Gamma)$, let $\{s_{i}\}_{i=1}^{d}$ be the collection of stable $(k+1)$-simplexes of $\P(\Ga_1)$ such that $s_{i}\subset s$. By Lemma \ref{intersection} (2), $\{\tilde{q}_{\ast}(s_{i})\}_{i=1}^{d}$ is exactly the collection of stable $(k+1)$-simplexes of $\P(\Ga_2)$ contained in $\tilde{q}_{\ast}(s)$. Then $\mathcal{V}^{k+1}_{1}\cap s$ is the vertex set of $$(\mathcal{V}^{k}_{1}\cap s)\cup (\cup_{i=1}^{d} s_{i}),$$ and $$\mathcal{V}^{k+1}_{2}\cap \tilde{q}_{\ast}(s)$$ is the vertex set of $$(\mathcal{V}^{k}_{1}\cap \tilde{q}_{\ast}(s))\cup(\cup_{i=1}^{d}\tilde{q}_{\ast}(s_{i})).$$ By our construction of $q_{\ast}$, it maps vertices in $s_i$ bijectively to vertices in $\tilde{q}_{\ast}(s_i)$. Thus $q_{\ast}$ maps $\mathcal{V}^{k+1}_{1}\cap s$ surjectively to $\mathcal{V}^{k+1}_{2}\cap \tilde{q}_{\ast}(s)$. It remains to check injectivity. Pick two points $v,v'\in \mathcal{V}^{k+1}_{1}\cap s$. The case $v,v'\in \mathcal{V}^{k}_{1}$ follows from induction. Now we consider the case $v,v'\notin \mathcal{V}^{k}_{1}$ and they are contained in different $s_i$'s (for simplicity we assume $v\in s_1$ and $v'\in s_2$). By construction of $q_{\ast}$, we have $$q_{\ast}(v)\in \tilde{q}_{\ast}(s_1)\setminus\mathcal{V}^{k}_{2}$$ and $$q_{\ast}(v')\in \tilde{q}_{\ast}(s_2)\setminus\mathcal{V}^{k}_{2}.$$ If $q_{\ast}(v)=q_{\ast}(v')$, then $\tilde{q}_{\ast}(s_1)\cap\tilde{q}_{\ast}(s_2)$ contains a point outside $\mathcal{V}^{k}_{2}$. On the other hand, by Lemma \ref{intersection} (1), $$\tilde{q}_{\ast}(s_1)\cap\tilde{q}_{\ast}(s_2)=\tilde{q}_{\ast}(s_1\cap s_2),$$ which is a stable simplex of dimension $\le k$. Thus every vertex of $$\tilde{q}_{\ast}(s_1)\cap\tilde{q}_{\ast}(s_2)$$ is in $\mathcal{V}^{k}_{2}$, which is a contradiction. Thus we must have $q_{\ast}(v)\neq q_{\ast}(v')$ in this case. The other cases are actually simpler and can be handled similarly.

Up to now, we have defined $q_{\ast}$ on the $0$-skeleton such that for each stable simplex $s\subset\P(\Ga)$, $q_{\ast}$ maps vertices in $s$ bijectively to vertices in $\tilde{q}_{\ast}(s)$. Next we show such $q_{\ast}$ is injective. Pick distinct vertices $v_{1},v_{2}$ in $\mathcal{P}(\Gamma_{1})$, if $d(v_{1},v_{2})=1$, then by applying $(\ast)$ to the maximal simplex containing $v_{1}$ and $v_{2}$ (recall that each maximal simplex is stable by our assumption), we have $d(q_{\ast}(v_{1}),q_{\ast}(v_{2}))=1$. If $d(v_{1},v_{2})\ge 2$, let $s_{i}$ be a maximal simplex containing $v_{i}$ for $i=1,2$. Then $$\tilde{q}_{\ast}(s_{1})\cap\tilde{q}_{\ast}(s_{2})=\tilde{q}_{\ast}(s_1\cap s_2)$$ by Lemma \ref{intersection} (1). By $(\ast)$, $q_{\ast}$ maps vertices in $s_1\cap s_2$ bijectively to vertices of $\tilde{q}_{\ast}(s_1\cap s_2)$. Since $v_1\notin s_1\cap s_2$, we apply $(\ast)$ to $s_1$ to deduce that $$q_{\ast}(v_1)\in \tilde{q}_{\ast}(s_{1})\setminus \tilde{q}_{\ast}(s_1\cap s_2)=\tilde{q}_{\ast}(s_{1})\setminus (\tilde{q}_{\ast}(s_{1})\cap\tilde{q}_{\ast}(s_{2}))=\tilde{q}_{\ast}(s_{1})\setminus \tilde{q}_{\ast}(s_{2}).$$ On the other hand, applying $(\ast)$ to $s_2$ implies $q_{\ast}(v_2)\in \tilde{q}_{\ast}(s_{2})$. Thus $$q_{\ast}(v_{1})\neq q_{\ast}(v_{2}).$$

We have already seen if $d(v_1,v_2)=1$, then $d(q_{\ast}(v_1),q_{\ast}(v_2))=1$. Thus $q_{\ast}$ naturally extends to an injective map on the 1-skeleton. Since $\mathcal{P}(\Gamma_{1})$ and $\mathcal{P}(\Gamma_{2})$ are flag complexes, we can further extend $q_{\ast}$ to obtain the required simplicial embedding.

Now we assume $\out(G(\Ga_2))$ does not contain non-adjacent transvection. Then $\V^{n-1}_2$ is the $0$-skeleton of $\P(\Ga_2)$. Thus $q_{\ast}$ is surjective on $0$-skeleton. We claim $q_{\ast}$ is an isomorphism between the 1-skeleton of $\P(\Ga_1)$ and the 1-skeleton of $\P(\Ga_2)$. It suffices to show if $d(q_{\ast}(v_1),q_{\ast}(v_2))=1$, then $d(v_1,v_2)=1$. However, this follows by considering a maximal simplex in $\P(\Ga_2)$ (which is also stable) containing $q_{\ast}(v_1)$ and $q_{\ast}(v_2)$ and applying $(\ast)$. Now we know $q_{\ast}$ is a simplicial isomorphism on the whole complex since $\mathcal{P}(\Gamma_{1})$ and $\mathcal{P}(\Gamma_{2})$ are flag complexes.
\end{proof}

\begin{cor}
\label{4.9}
Let $q:X(\Gamma_{1})\to X(\Gamma_{2})$ be a quasi-isometry. Suppose $\out(G(\Gamma_{1}))$ does not contain any non-adjacent transvection and let $q_{\ast}$ be the map defined in Theorem \ref{4.8}. Then for any subcomplex $M\in\mathcal{S}(\mathcal{P}(\Gamma_{1}))$.
\begin{equation}
\label{4.11}
q_{\ast}(M)\subset\tilde{q}_{\ast}(M)\ \textmd{and}\  M=q_{\ast}^{-1}(\tilde{q}_{\ast}(M)).
\end{equation} 
If we also assume $\out(G(\Gamma_2))$ does not contain any non-adjacent transvection, then
\begin{equation}
\label{4.10}
q_{\ast}(M)=\tilde{q}_{\ast}(M)
\end{equation}
for any subcomplex $M\in\mathcal{S}(\mathcal{P}(\Gamma_{1}))$.
\end{cor}

\begin{proof}
Since each maximal simplex in $\P(\Ga_1)$ is stable, the intersection of such simplex with $M$ is also stable by Lemma \ref{intersection}. Since $\P(\Ga_1)$ is a union of its maximal simplexes, $M$ is a union of stable simplexes. Now the first inclusion of (\ref{4.11}) follows from Lemma \ref{intersection} (2) and (\ref{coincidence}).

Now we prove the second equality of (\ref{4.11}). Suppose there exists vertex $v\notin M$ such that $q_{\ast}(v)\in\tilde{q}_{\ast}(M)$. Let $s_{v}$ be the minimal stable simplex of $\mathcal{P}(\Gamma_{1})$ such that $v\in s_{v}$ (recall that the collection of stable simplexes are closed under intersection by Lemma \ref{intersection}, and any maximal simplex which contains $v$ is stable, so $s_v$ is well-defined). By (\ref{coincidence}), $q_{\ast}(v)\in \tilde{q}_{\ast}(s_{v})$. By Lemma \ref{intersection} (2) and (\ref{coincidence}), $\tilde{q}_{\ast}(s_{v})$ is the minimal stable simplex in $\P(\Ga_2)$ that contains $q_{\ast}(v)$. Since $\tilde{q}_{\ast}(s_{v})\cap\tilde{q}_{\ast}(M)$ is also a stable simplex containing $q_{\ast}(v)$, we have $$\tilde{q}_{\ast}(s_{v})\subset\tilde{q}_{\ast}(M)$$ Thus $s_{v}\subset M$ by Lemma \ref{intersection} (2), which is contradictory to $v\in M$.

Now we assume $\out(G(\Ga_2))$ does not contain non-adjacent transvection. Then $\tilde{q}_{\ast}(M)$ is a union of stable simplexes by the same argument as before. By Lemma \ref{intersection} (2), there is a 1-1 correspondence between stable simplexes in $M$ and stable simplexes in $\tilde{q}_{\ast}(M)$. Thus (\ref{4.10}) follows from (\ref{coincidence}).
\end{proof}

\subsection{Visible isomorphisms between extension complexes}
\label{sec:visible}
It is natural to ask to what extent is the converse of Theorem \ref{4.8} true, namely, suppose $\alpha:\mathcal{P}(\Gamma)\to\mathcal{P}(\Gamma')$ is a simplicial isomorphism, does $\alpha$ induce a map from $G(\Gamma)\to G(\Gamma')$? Here is a natural construction. We identify $G(\Ga)$ and $G(\Ga')$ with the 0-skeleton of $X(\Ga)$ and $X(\Ga')$ respectively. Pick vertex $p\in G(\Gamma)$, let $\{F_{i}\}_{i=1}^{n}$ be the collection of maximal standard flats containing $p$. For each $i$, let $F'_{i}\subset X(\Gamma')$ be the unique maximal standard flat such that $\Delta(F'_{i})=\alpha(\Delta(F_{i}))$. One may wish to map $p$ to $\cap_{i=1}^{n}F'_{i}$, which motivates the following definition:
\begin{definition}
\label{visible}
The simplicial isomorphism $\alpha$ is \textit{visible} if $\cap_{i=1}^{n}F'_{i}\neq\emptyset$ for any $p\in G(\Gamma)$.
\end{definition}

If $G(\Gamma)$ has trivial center and $\alpha$ is visible, then it is easy to see $\alpha$ induces a unique map $\alpha_{\ast}:G(\Gamma)\to G(\Gamma')$. To see this, recall that by \cite{servatius1989automorphisms}, $G(\Ga)$ has trivial center if and only if $\Ga$ is not contained in the closed star of one of its vertex. Let $\{F_i\}_{i=1}^{n}$ and $\{F'_i\}_{i=1}^{n}$ be as in the above discussion. Since the intersection of all maximal cliques in $\Ga$ is empty, we have $\cap_{i=1}^{n}\Delta(F_i)=\emptyset$, hence $\cap_{i=1}^{n}\Delta(F'_i)=\emptyset$. Thus $p$ is the only point in $\cap_{i=1}^{n}F_{i}$ and $\cap_{i=1}^{n}F'_{i}$ contains at most one point. However, the visibility implies $\cap_{i=1}^{n}F'_{i}$ is exactly one point. We define this point to be $\alpha_{\ast}(p)$.

If $G(\Gamma)$ has non-trivial center, then $\cap_{i=1}^{n}F_{i}$ and $\cap_{i=1}^{n}F'_{i}$ corresponds to cosets of centralizers of $G(\Ga)$ and $G(\Ga')$ respectively. The map $\alpha$ only tells us which coset go to which coset. In order to define $\alpha_{\ast}:G(\Gamma)\to G(\Gamma')$, we need to choose a map for each coset. Thus $\alpha_{\ast}$ is not uniquely defined.

A sufficient condition for $\alpha$ to be visible has been given previously in \cite[Lemma 4.10]{raagqi1}. Here we will find a necessary and sufficient condition for the visibility of $\alpha$.

\section{The structure of extension complex}
\label{sec:extension complex structure}

In this section, we introduce the classes of RAAGs we want to investigate, namely weak type II and type II (Definition~\ref{7.2}), and weak type I (Definition~\ref{7.24}). The bulk of this section is on the structure of extension complexes for RAAGs of (weak) type II, and quasi-isometries invariance of (weak) type II. The application of this to quasi-isometric classification of RAAGs of weak type I will be discussed at the end of the section. In particular, Theorem~\ref{1.2} is proved in Section~\ref{sec:3.4}.

Throughout this section, we identify $\Gamma$ with the 1-skeleton of $F(\Gamma)$, and we will implicitly use Lemma \ref{connect} in various places.

\subsection{Tiers and branches of the extension complex}
\label{subsec:tiers}
Let $\pi:\mathcal{P}(\Gamma)\to F(\Gamma)$ be the label-preserving simplicial map in Section \ref{basics about raag}.

Pick a standard geodesic $\ell\subset X(\Gamma)$ and let $\pi_{\ell}: X(\Gamma)\to \ell$ be the $CAT(0)$ projection onto $\ell$. Suppose $\ell_{1}\subset X(\Gamma)$ is a standard geodesic such that $d(\Delta(\ell_{1}),\Delta(\ell))\ge 2$. Then $\pi_{\ell}(\ell_{1})$ is a vertex in $\ell$ by Lemma \ref{3.1} and Lemma \ref{3.2}. Moreover, if $\ell_{2}$ is a standard geodesic parallel to $\ell_{1}$, then $\pi_{\ell}(\ell_{1})=\pi_{\ell}(\ell_{2})$  (see \cite[Lemma 6.2]{raagqi1}). Thus $\pi_{\ell}$ induces a well-defined map $\pi_{\Delta(\ell)}$ from the $v(\mathcal{P}(\Gamma)\setminus St(\Delta(\ell)))$, the set of vertices in $\mathcal{P}(\Gamma)\setminus St(\Delta(\ell))$, to $v(\ell)$. 

\begin{lem}
\label{5.17}
\cite[Lemma 6.2]{raagqi1} If $v_{1}$ and $v_{2}$ are in the same connected component of $\mathcal{P}(\Gamma)\setminus St(\Delta(\ell))$, then $\pi_{\Delta(\ell)}(v_{1})=\pi_{\Delta(\ell)}(v_{2})$.
\end{lem}

The follow definition plays a central role in our understanding of the extension complex.
\begin{definition}
	\label{def:proj}
Pick $v\in\mathcal{P}(\Gamma)$, and let $\ell\subset X(\Gamma)$ be a standard geodesic such that $\Delta(\ell)=v$. Let $$\pi_{\Delta(\ell)}:v(\mathcal{P}(\Gamma)\setminus St(v))\to v(\ell)$$ be the map in Lemma \ref{5.17}. A \textit{$v$-tier} is the full subcomplex spanned by $\pi_{\Delta(\ell)}^{-1}(x)$, where $x$ is a vertex in $\ell$ and $x$ is called the \textit{height} of the $v$-tier. A \textit{$v$-branch} is the full subcomplex spanned by vertices in one connected component of $\mathcal{P}(\Gamma)\setminus St(v)$. 
	
\end{definition}

By Lemma \ref{5.17}, a $v$-branch has non-empty intersection with a $v$-tier if and only if it belongs to the $v$-tier, thus a $v$-tier is consists of disjoint union of $v$-branches. Also note that a simplicial isomorphism $\alpha:\mathcal{P}(\Gamma_{1})\to\mathcal{P}(\Gamma_{2})$ will map branches to branches, but it may not map tiers to tiers.

\begin{lem}
\label{7.1}
If the $\alpha$-image of any $v$-tier of $\mathcal{P}(\Gamma_{1})$ is inside a single $\alpha(v)$-tier of $\mathcal{P}(\Gamma_{2})$, then $\alpha$ is visible.
\end{lem}

\begin{proof}
Let $p$, $\{F_{i}\}_{i=1}^{n}$ and $\{F'_{i}\}_{i=1}^{n}$ be as in Definition \ref{visible}. By Lemma \ref{2.1}, it suffices to show $F'_{i}\cap F'_{j}\neq\emptyset$ for any $i\neq j$. Suppose $\alpha$ is not visible. Then there exists a hyperplane $h$ separating $F'_{i}$ and $F'_{j}$. Let $\ell'$ be a standard geodesic dual to $h$ and let $v'=\Delta(\ell')$. Note that $F'_i$ does not contain any line which is parallel to $\ell'$, otherwise $h$ would have non-trivial intersection with $F'_i$. Let $\pi_{\ell'}(F'_i)$ be the $CAT(0)$ projection of $F'_i$ to $\ell'$. Then $\pi_{\ell'}(F'_i)$ is a point. Similarly $\pi_{\ell'}(F'_j)$ is a point. Since $h$ separates $F'_{i}$ and $F'_{j}$, we have  $$\pi_{\ell'}(F'_i)\neq\pi_{\ell'}(F'_j).$$ The maximality of $F'_{i}$ and $F'_{j}$ implies there exist vertices $v'_{1}\in\Delta(F'_{i})$ and $v'_2\in\Delta(F'_{j})$ such that $v'_i\notin St(v')$ for $i=1,2$. Since $\pi_{\ell'}(F'_i)\neq\pi_{\ell'}(F'_j)$, $v'_{1}$ and $v'_{2}$ are in different $v'$-tiers. On the other hand, we claim $\alpha^{-1}(v'_{1})$ and $\alpha^{-1}(v'_{2})$ are in the same $\alpha^{-1}(v')$-tier, which would give a contraction. To see the claim, note that $\alpha^{-1}(v'_{i})$ can be represented by a standard geodesic $\ell_i\subset F_i$ for $i=1,2$. Since both $F_i$ and $F_j$ contain $p$, we can assume $p\in \ell_1\cap \ell_2$. Then the $CAT(0)$-projection of $\ell_1$ and $\ell_2$ to any standard geodesic representing $\alpha^{-1}(v')$ is the same. Thus $\alpha^{-1}(v'_{1})$ and $\alpha^{-1}(v'_{2})$ are in the same $\alpha^{-1}(v')$-tier.
\end{proof}

The main goal of this subsection is Corollary~\ref{7.14}, where
we characterize $v$-branches in a $v$-tier for a certain class of $\Gamma$ defined as follows. 

\begin{definition}
\label{7.2}
A graph $\Gamma$ is of \textit{type II} if $\Gamma$ is connected and for every pair of distinct vertices $v,w\in\Gamma$, $\Gamma\setminus (lk(v)\cap lk(w))$ is connected. $\Gamma$ is said to have \emph{weak type II} if $\Gamma$ is connected and for vertices $v,w\in\Gamma$ such that $d(v,w)=2$, $\Gamma\setminus (lk(v)\cap lk(w))$ is connected. 

We say $\mathcal{P}(\Gamma)$ is of \textit{type II} if $\mathcal{P}(\Gamma)$ is connected and for every pair of distinct vertices $v,w\in\mathcal{P}(\Gamma)$, we know $(\mathcal{P}(\Gamma))^{(1)}\setminus (lk(v)\cap lk(w))$ is connected. We define $\mathcal{P}(\Gamma)$ being \textit{weak type II} in a similar way.

We say $G(\Gamma)$ or $F(\Gamma)$ is of (weak) type II if $\Gamma$ is of (weak) type II.
\end{definition}

Note that $\Gamma$ is connected if it is of (weak) type II. 

\begin{example}
	\label{ex:type II}
A pentagon is a graph of type II. A slightly more complicated example of graph of type II is a 5-cycle and a 6-cycle identified along a closed star. However, a path of length 3 (or more generally a tree of diameter $\ge 3$) is not of weak type II. In the following discussions and proofs, it would be helpful to have these basic examples in mind and compare them.
\end{example}

%Let $K$ be a simplicial complex and let $K_{1},K_{2}$ be two subcomplexes. $K_{1}$ and $K_{2}$ \textit{contact} if there exist vertices $v_{i}\in K_{i}$ for $i=1,2$ such that $v_{1}$ and $v_{2}$ are adjacent.

Let $v\in\mathcal{P}(\Gamma)$ be a vertex, and let $\ell\subset X(\Gamma)$ be a standard geodesic such that $\Delta(\ell)=v$. Define $P_{v}$ to be the parallel set $P_{\ell}$ of $\ell$. Note that $P_{v}$ does not depend on the choice of the standard geodesic $\ell$ with $\Delta(\ell)=v$. A subset $K\subset P_{v}$ is \textit{horizontal} if $\pi_{\ell}(K)$ is a point (where $\pi_{\ell}:X(\Gamma)\to \ell$ is the $CAT(0)$ projection) and $\pi_{\ell}(K)$ is called the \textit{height} of $K$.

Let $\bar{v}\in\Gamma$ be the label of $v\in\P(\Ga)$. Then $P_v$ is a standard subcomplex whose support (Definition \ref{notation}) is $St(\bar{v})$. 

\begin{lem}
\label{7.4}
Suppose $\Gamma$ is of weak type II. Pick vertices $v,w\in\mathcal{P}(\Gamma)$ such that $d(v,w)=2$. Let $u\in v^{\perp}\cap w^{\perp}$. Then there exists a vertex $w'$ such that
\begin{enumerate}
\item $d(v,w')=2$ and $d(u,w')=1$.
\item $w'$ and $w$ are in the same $v$-branch.
\item $P_{v}\cap P_{w'}\neq\emptyset$.
\end{enumerate}
In particular, every $v$-branch contains a vertex $w'$ such that $P_{w'}\cap P_{v}\neq\emptyset$.
\end{lem}

\begin{proof}
Let $B$ be a $v$-branch containing $w$. Pick vertex $x\in P_{w}\cap P_{u}$ and let $y\in P_{v}$ be the nearest vertex to $x$ with respect to the $\ell^1$-metric. The existence and uniqueness of such vertex follows from \cite[Lemma 13.8]{haglund2008special}. Let us assume $x\neq y$, otherwise we are done by putting $w'=w$. Let $\omega$ be a combinatorial geodesic connecting $x$ and $y$. We claim $\omega\subset P_{u}$. By Lemma \ref{combinatorial convex}, it suffices to show $y\in P_{u}$. However, this follows by applying \cite[Lemma 13.8]{haglund2008special} with $C'=P_v$ and $C=P_u$.

Let $\{x_{i}\}_{i=0}^{n}$ be vertices in $\omega$ such that for $0\le i\le n-1$, $[x_{i},x_{i+1}]$ is a maximal sub-segment of $\omega$ that is contained in a standard geodesic ($x_{0}=x$ and $x_{n}=y$). Denote the standard geodesic containing $[x_i,x_{i+1}]$ by $\ell_{i}$ and let $v_{i}=\Delta(\ell_{i})$ for $0\le i\le n-1$. Since $\omega$ is the shortest combinatorial geodesic connecting $x$ and some vertex in $P_v$, every dual hyperplane of some edge in $\omega$ must separate $x$ and $P_{v}$. Thus for each $i$, there exists a hyperplane dual to $\ell_{i}$ which does not intersect $P_{v}$. It follows that
\begin{equation}
\label{7.5}
d(v_{i},v)\ge 2
\end{equation}
for all $i$. Since $\ell_{i}\subset P_{u}$, we also have 
\begin{equation}
\label{7.6}
d(v_{i},u)=1.
\end{equation}
Let $\pi:\mathcal{P}(\Gamma)\to F(\Gamma)$ be the projection mentioned at the beginning of this section. Since $\ell_{n-1}\cap P_{v}\neq\emptyset$, it follows from (\ref{7.5}) and (\ref{7.6}) that 
\begin{equation}
\label{7.7}
d(\pi(v_{n-1}),\pi(v))= 2.
\end{equation}

We claim $v_{0}\in B$. Let $K_{x_{0}}=(F(\Gamma))_{x_{0}}$ (i.e. $K_{x_0}$ is the subcomplex of $\P(\Ga)$ made of simplexes which come from standard flats passing $x_0$, see the paragraph before Lemma \ref{isometric embedding}). First we show $K_{x_{0}}\cap St(v)$ is contained in the intersection of the links of two vertices. Pick vertex $s\in K_{x_{0}}\cap St(v)$, and let $\ell_{s}$ be the standard geodesic such that $x_{0}\in \ell_s$ and $\Delta(\ell_s)=s$. Let $h$ be a hyperplane dual to $\ell_{n-1}$ such that it separates $x_{0}$ from $P_{v}$. Then $h\cap \ell_{s}=\emptyset$ by (\ref{7.5}) (note that if $h\cap \ell_s\neq\emptyset$, then $s=v_{n-1}$), hence $h$ separates $\ell_{s}$ from $P_{v}$. It follows from Lemma \ref{3.1} and Lemma \ref{3.2} that $$\pi(s)\in (\pi(v_{n-1}))^{\perp}\cap(\pi(v))^{\perp}.$$ Let $K$ be the full subgraph of $\Ga$ spanned by $(\pi(v_{n-1}))^{\perp}\cap(\pi(v))^{\perp}$. Then $t\notin St(v)$ for any vertex $t\in K_{x_{0}}$ such that $\pi(t)\notin K$. 

By (\ref{7.7}), $K$ does not separate $\Gamma$, so if $\pi(w)\notin K$ and $\pi(v_{0})\notin K$, then they can be connected by an edge path outside $K$, which lifts to a path in $K_{x_{0}}$ connecting $w$ and $v_{0}$ outside $St(v)$, thus $v_{0}\in B$. If $\pi(w)\notin K$ and $\pi(v_{0})\in K$, then we connect $\pi(w)$ and $\pi(v_{n-1})$ by an edge path outside $K$, then connect $\pi(v_{n-1})$ and $\pi(v_{0})$ by an edge, this path also lifts to a path in $K_{x_{0}}$ connecting $w$ and $v_{0}$ outside $St(v)$. The other cases can be dealt with in a similar way. We can repeat this process and argue inductively that actually $v_{i}\in B$ for $0\le i\le n-1$, then the lemma follows by taking $w'=v_{n-1}$.
\end{proof}

\begin{lem}
	\label{7.8}
Suppose $\Gamma$ is of weak type II. Take vertices $v\in \mathcal{P}(\Gamma)$ and $x\in X(\Gamma)$. If there exists two vertices $v_{1},v_{2}\in(F(\Gamma))_{x}$ such that they are in different $v$-branches, then $v\in (F(\Gamma))_{x}$.
\end{lem}
\begin{proof}
	If this is not true, then $x\notin P_{v}$. We take a combinatorial geodesic of shortest length from $x$ to a vertex in $P_v$ and repeat the above argument to see that $(F(\Gamma))_{x}\setminus St(v)$ is connected, which contradicts that $v_{1}$ and $v_{2}$ are in different $v$-branches.
\end{proof}

\begin{lem}
\label{7.9}
Suppose $\Gamma$ is an arbitrary finite simplicial graph. Let $v\in\P(\Ga)$ be a vertex and let $v_{1},v_{2}\in\mathcal{P}(\Gamma)\setminus St(v)$ be two vertices such that $P_{v_{i}}\cap P_{v}\neq\emptyset$ for $i=1,2$. Suppose $\bar{v}=\pi(v)$. If $\pi(v_{1})$ and $\pi(v_{2})$ are in different connected components of $F(\Gamma)\setminus St(\bar{v})$, then $v_{1}$ and $v_{2}$ are in different $v$-branches. 
\end{lem}

\begin{proof}
For $i=1,2$, let $\ell_{i}$ be a standard geodesic such that $\Delta(\ell_{i})=v_{i}$ and $\ell_{i}\cap P_{v}\neq \emptyset$. We prove the lemma in 2 steps. In the first step, we will show that is we assume $v_1$ and $v_2$ are in the same $v$-branch, then there exist vertices $x\in \ell_{1}\setminus P_{v}$ and $y\in \ell_{2}\setminus P_{v}$ such that they can be connected by an edge path outside $P_{v}$. In the second step, we will show $x$ and $y$ cannot be connected by an edge path outside $P_v$. Such contradiction will finish the proof.

\emph{Step 1:} Note that $\ell_{1}\cap P_{v}$ and $\ell_{2}\cap P_{v}$ are of the same height, otherwise $v_1$ and $v_2$ are in different $v$-tiers, hence are in different $v$-branches. Let $\{w_{i}\}_{i=1}^{n}\subset\mathcal{P}(\Gamma)\setminus St(v)$ be vertices such that $d(w_{i},w_{i+1})=1$, $w_{1}=v_{1}$ and $w_{n}=v_{2}$. Pick $x=x_{1}$ to be any vertex in $\ell_{1}\setminus P_{v}$. For $1\le i\le n-1$, let $r'_{i}$ be a standard geodesic with $\Delta(r'_{i})=w_{i}$ and $r'_{i}\cap P_{w_{i+1}}\neq \emptyset$ (set $r'_{n}=\ell_{2}$). Let $\omega_{1}$ be horizontal edge path in $P_{w_{1}}$ connecting $x_{1}$ and a vertex $x_{2}\in r'_{1}$. Note that $\omega_{1}\cap P_{v}=\emptyset$ since $P_{v}\cap P_{w_{1}}$ is either empty or horizontal in $P_{w_{1}}$. Let $r_{2}$ be the standard geodesic such that $x_{2}\in r_{2}$ and $\Delta(r_{2})=w_{2}$. If $P_{w_{2}}\cap P_{v}=\emptyset$ or $P_{w_{2}}\cap P_{v}$ and $x_{2}$ have different height in $P_{w_2}$, let $\omega_{2}$ be a horizontal edge path joining $x_{2}$ and a vertex $x_{3}\in r'_{2}$. If $P_{w_{2}}\cap P_{v}$ and $x_{2}$ have the same height in $P_{w_2}$, let $\omega'_{2}$ be an edge in $r_{2}$ joining $x_{2}$ and another vertex $x'_{2}$ and let $\omega''_{2}$ be a horizontal edge path joining $x'_{2}$ and a vertex $x_{3}\in r'_{2}$. Set $\omega_{2}=\omega_{2}'\cup\omega''_{2}$, it clear that $\omega_{2}\cap P_{v}=\emptyset$ in both cases. We can define $\omega_{i}$ and $x_{i+1}$ for $3\le i\le n$ in the same way. Let $y=x_{n+1}$ and the first step follows.

\emph{Step 2:} Let $C_{1}$ be the component of $F(\Gamma)\setminus St(\bar{v})$ that contains $\pi(v_{1})$ and $C_{2}$ be the union of all other components. For $i=1,2$, let $\Gamma_{i}$ be the full subgraph spanned by vertices in $C_{i}\cup St(\bar{v})$. Then $St(\bar{v})=\Gamma_{1}\cap\Gamma_{2}$. Let $S(St(\bar{v}))$ and $S(lk(\bar{v}))$ be the Salvetti complexes with defining graphs $St(\bar{v})$ and $lk(\bar{v})$ respectively. Note that $$S(St(\bar{v}))\cong S(lk(\bar{v}))\times \Bbb S^{1}.$$ Note that $S(St(\bar{v}))$ sits naturally in $S(\Ga_1)$ and $S(\Ga_2)$, we can obtain $S(\Ga)$ by gluing $S(\Ga_1)$ and $S(\Ga_2)$ along $S(St(\bar{v}))$.

Now we glue $S(\Ga_1)$ and $S(\Ga_2)$ in a different to obtain a new space $\bar{S}(\Ga)$ as follows. For reason which will be clear shortly, we assume the $\S^1$ factor in $S(St(\bar{v}))$ has length $=4\pi$. Let $h$ be an isometry of $S(St(\bar{v}))$ which is identity on the $S(lk(\bar{v}))$ factor and is a rotation of degree $=2\pi$ on the $\Bbb S^{1}$ factor. Now we gluing $S(\Ga_1)$ and $S(\Ga_2)$ using the isometry $h$ to obtain $\bar{S}(\Ga)$. Note that there is a homotopy equivalence $g:\bar{S}(\Gamma)\to S(\Gamma)$ induced by collapsing the interval $[e^{i0},e^{i2\pi}]$ in the $\Bbb S^{1}$ factor of $S(St(\bar{v}))$ to one point (see the following picture, where the black part is collapsed). It lifts to a cubical map $\tilde{g}:\bar{X}(\Gamma)\to X(\Gamma)$. 
\begin{center}
\includegraphics[scale=0.3]{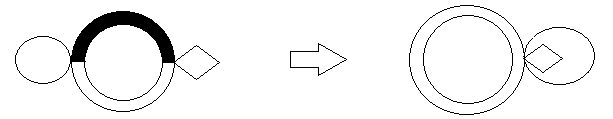}
\end{center}

Let $M\subset P_{v}$ be the standard subcomplex such that $\ell_{1}\cap P_{v}\subset M$ and the support of $M$ satisfies $\Gamma_{M}=lk(\bar{v})$. Then there exists a unique hyperplane $\bar{h}\subset\bar{X}(\Gamma)$ such that $\tilde{g}(\bar{h})=M$. For $i=1,2$, let $\bar{\ell}_{i}\subset \bar{X}(\Gamma)$ be the unique geodesic such that $\tilde{g}(\bar{\ell_{i}})=\ell_{i}$. Then $\bar{\ell}_1$ and $\bar{\ell}_2$ have non-empty intersection with $\tilde{g}^{-1}(P_v)$. Since $\tilde{g}^{-1}(P_v)$ is a lift of $S(St(\bar{v}))$ in $\bar{X}(\Gamma)$, and $\pi(\Delta(\ell_i))\in C_i$ for $i=1,2$, we know that $\bar{\ell}_{1}$ and $\bar{\ell}_{2}$ are separated by $\bar{h}$. Let $\omega=\cup_{i=1}^{n}\omega_{i}$ be the edge path connecting $x$ and $y$ in the previous step. Note that the inverse image of each edge in $X(\Ga)$ under $\tilde{g}$ is either an edge, or a square; the inverse image of each vertex is either a vertex, or an edge. Then $\tilde{g}^{-1}(\omega)$ is a compact connected subcomplex of $X(\bar{\Gamma})$. Since $\omega\cap P_v=\emptyset$, $\tilde{g}^{-1}(\omega)\cap \tilde{g}^{-1}(P_v)=\emptyset$. Hence $\tilde{g}^{-1}(\omega)\cap\bar{h}=\emptyset$. Moreover, $\tilde{g}^{-1}(\omega)\cap\bar{\ell}_{i}\neq\emptyset$ for $i=1,2$, which contradicts the separation property of $\bar{h}$.
\end{proof}

The following observation follows from step 2 of the proof of Lemma \ref{7.9}.
\begin{lem}
\label{7.10}
Let $\Gamma$ be arbitrary. Let $\omega\subset X(\Gamma)$ be an edge path joining vertices $x_{1},x_{2}\in P_{v}$, and suppose $\omega\setminus\{x_{1},x_{2}\}$ stays inside one component of $X(\Gamma)\setminus P_{v}$. Then
\begin{enumerate}
\item $x_{1}$ and $x_{2}$ are of the same height in $P_v$.
\item For $i=1,2$, let $e_{i}\subset\omega$ be the edge containing $x_{i}$, and let $\bar{v}_{i}\in\Gamma$ be the label of $e_{i}$. Then $\bar{v}_{1}$ and $\bar{v}_{2}$ are in the same component of $\Gamma\setminus St(\bar{v})$.
\end{enumerate}
\end{lem}

Let $v\in\P(\Ga)$ be a vertex and let $\bar{v}=\pi(v)\in\Ga$. Let $C$ be a component of $\Gamma\setminus St(\bar{v})$. We define $\partial C$ to be the full subgraph spanned by vertices in $\bar{C}\setminus C$, here $\bar{C}$ is the closure of $C$. Equivalently, $\partial C$ is the full subgraph spanned by vertices in $\{u\in\Gamma\setminus C\mid$ there exists vertex $w\in C$ such that $d(w,u)=1 \}$. Similarly, for every $v$-branch $B\subset\mathcal{P}(\Gamma)$, we define the \textit{boundary} of $B$, denoted by $\partial B$, to be the full subcomplex spanned by vertices in $\{u\in\mathcal{P}(\Gamma)\setminus B\mid$ there exists vertex $w\in B$ such that $d(w,u)=1 \}$ = $\{u\in St(v)\mid$ there exists vertex $w\in B$ such that $d(w,u)=1 \}$. Such $\partial B$ is called a \textit{$v$-peripheral subcomplex} of $\mathcal{P}(\Gamma)$. We caution the reader that $B\cup\partial B$ may not equal to the closure of $B$.

A subcomplex $K\subset P_v$ is called a \textit{$v$-peripheral subcomplex (of type $\partial C$)} if $K$ is a standard subcomplex and $\Gamma_{K}=\partial C$ for some component $C$ of $\Gamma\setminus St(\bar{v})$. If the vertex set of $\partial C$ is properly contained in $\bar{v}^{\perp}$, then there are infinitely many $v$-peripheral subcomplexes of type $\partial C$ which are of the same height. 

\begin{example}
We give an example of $v$-peripheral subcomplex. Let $\Ga$ be a pentagon and $\bar{v}\in\Ga$ be a vertex. Pick a lift $v\in\P(\Ga)$ of $\bar{v}$. Then $P_v$ is isometric to $\R\times T_4$ where $T_4$ is the $4$-valence tree. Note that $\Ga\setminus St(\bar{v})$ only have one component $C$ and $\partial C=lk(\bar{v})$. So any standard subcomplex of $P_v$ whose support is $lk(\bar{v})$ is a $v$-peripheral subcomplex of type $\partial C$. In our case $v$-peripheral subcomplexes are those $T_4$-slices in $\R\times T_4$. For a given height, there is only one $v$-peripheral subcomplex of type $\partial C$.
\end{example}

\begin{lem}
\label{7.11}
Let $\Gamma$ be arbitrary. Let $x_{1},x_{2},\bar{v}_{1},\bar{v}_{2},\omega$ and $P_{v}$ be as in Lemma \ref{7.10} and let $C$ be the component of $\Gamma\setminus St(\bar{v})$ containing $\bar{v}_{1}$ and $\bar{v}_{2}$. Then $x_{1}$ and $x_{2}$ are in the same $v$-peripheral subcomplex of type $\partial C$.
\end{lem}

\begin{proof}
For $i=1,2$, let $K_{i}$ be the $v$-peripheral subcomplex of type $\partial C$ such that $x_{i}\in K_{i}$. Note that $K_{1}$ and $K_{2}$ are horizontal subcomplexes of $\mathcal{P}(\Gamma)$ of the same height. We argue by contradiction and suppose $K_{1}\neq K_{2}$. Then $K_{1}\cap K_{2}=\emptyset$. We claim there exists an edge $e\in P_{v}$ such that its label $v_{e}$ does not belong to $\partial C$ and the hyperplane dual to $e$ separates $K_{1}$ from $K_{2}$. To see this, pick vertices $x\in K_{1}$ and $y\in K_{2}$ such that $$d(x,y)=d(K_{1},K_{2}).$$ Let $\omega_{1}$ be a combinatorial geodesic joining $x$ and $y$. Then $\omega_{1}\subset P_{v}$ by Lemma \ref{combinatorial convex}. Moreover, every hyperplane dual to some edge in $\omega_{1}$ separates $K_{1}$ and $K_{2}$. Thus there exists an edge $e\in\omega_{1}$ such that $v_{e}\notin\partial C$, otherwise we would have $\omega_{1}\subset K_{1}$. 

Let $h_{e}$ be the hyperplane dual to $e$ and let $N_{h_{e}}$ be the carrier of $h_{e}$. Then $h_{e}$ separates $x_{1},x_{2}$ and there exists an edge $e'\subset\omega$ parallel to $e$ ($\omega$ is the path in Lemma \ref{7.10}). Pick endpoint $y\in e'$ and let $\omega_{2}\subset N_{h_{e}}$ be an edge path of shortest combinatorial length connecting $y$ and $P_{v}\cap N_{h_{e}}$. Let $x_{3}$ be the other endpoint of $\omega_{2}$ and let $e''\subset \omega_{2}$ be the edge containing $x_{3}$. Then $d(v_{e''},v_{e})=1$ ($v_{e''}$ is the label of $e''$), and it follows from $v_{e}\notin \partial C$ that
\begin{equation}
\label{7.12}
v_{e''}\notin C.
\end{equation}
Let $\omega_{3}$ be an edge path connecting $x_{1}$ and $x_{3}$ obtained by first following $\omega$ from $x_{1}$ to $y$, then following $\omega_{2}$ until $x_{3}$. Then applying Lemma \ref{7.10} to $\omega_{3}$ yields a contradiction to (\ref{7.12}).
\end{proof}

\begin{cor}
\label{7.13}
Suppose $\Gamma$ is connected. Let $v$ and $\bar{v}$ be as in Lemma \ref{7.9} and pick a component $C$ of $\Gamma\setminus St(\bar{v})$. Suppose $K_{1}$ and $K_{2}$ are two distinct $v$-peripheral subcomplexes of type $\partial C$ and they have the same height. Let $w_{1},w_{2}\in\mathcal{P}(\Gamma)\setminus St(v)$ be vertices such that $\pi(w_{i})\in C$ for $i=1,2$. Suppose $P_{w_{i}}\cap K_{i}\neq\emptyset$ for $i=1,2$. Then $w_{1}$ and $w_{2}$ are in different $v$-branches.
\end{cor}

\begin{proof}
For $i=1,2$, let $\ell_{i}$ be a standard geodesic such that $\ell_{i}\cap K_i\neq\emptyset$ and $\Delta(\ell_i)=w_i$. If $w_1$ and $w_2$ are in the same $v$-branch, then the argument in the second paragraph of the proof of Lemma \ref{7.9} implies there exists an edge path $\omega\subset X(\Gamma)\setminus P_v$ connecting a vertex in $\ell_1\setminus P_v$ to a vertex in $\ell_2\setminus P_v$. Then it follows from Lemma \ref{7.11} that $K_1=K_2$, which is a contradiction.
\end{proof}

\begin{cor}
\label{7.14}
Suppose $\Gamma$ is of weak type II. Let $v$ and $\bar{v}$ be as before, and let $\ell\subset X(\Gamma)$ be a standard geodesic such that $\Delta(\ell)=v$. Pick vertex $x\in \ell$, then 
\begin{enumerate}
\item There is a 1-1 correspondence between $v$-branches in the $v$-tier of height $x$ and pairs $(C,K)$, where $C$ is a component in $\Gamma\setminus St(\bar{v})$ and $K$ is a $v$-peripheral subcomplexes in $X(\Gamma)$ of height $x$ such that $\Gamma_{K}=\partial C$. Moreover, let $B$ be the $v$-branch corresponding to $(C,K)$. Then $\partial B=\Delta(K)$.
\end{enumerate}
Now we assume $\Gamma$ is of type II, then the following holds.
\begin{enumerate}[resume]
\item For every $v$-peripheral subcomplex $A\subset\mathcal{P}(\Gamma)$, there exists a unique $v$-peripheral subcomplex $K\subset X(\Gamma)$ of height $x$ such that $\Delta(K)=A$.
\item Let $A$ be as in (2). Then there are only finitely many $v$-branches with boundary equal to $A$ in a $v$-tier.
\item Let $v_{1},v_{2}\in\mathcal{P}(\Gamma)$ be two different vertices and $B_{i}\subset\mathcal{P}(\Gamma)$ be a $v_{i}$-branch for $i=1,2$. Then $B_{1}\neq B_{2}$.
\item Let $v_{1},v_{2}$ be as above. Then $\mathcal{P}(\Gamma)\setminus (lk(v_{1})\cap lk(v_{2}))$ is connected.
\end{enumerate}
\end{cor}

\begin{proof}
For $i=1,2$, pick pairs $(C_{i},K_{i})$ as above, let $w_{i}\in\mathcal{P}(\Gamma)$ be a vertex such that $\pi(w_{i})\in C_{i}$ and $P_{w_{i}}\cap K_{i}\neq\emptyset$, we claim $w_{1}$ and $w_{2}$ are in the same $v$-branch if and only if $C_{1}=C_{2}$ and $K_{1}=K_{2}$. Assuming the claim, then the first part of (1) follows from Lemma \ref{7.4}. The only if direction follows from Lemma \ref{7.9} and Corollary \ref{7.13}. For the other direction, pick vertex $x_{i}\in P_{w_{i}}\cap K_{i}$, it suffices to consider the case when $x_{1}$ and $x_{2}$ are joined by an edge $e\subset K_{1}$. Let $\ell_{e}$ be the standard geodesic containing $e$ and let $v_{e}=\Delta(\ell_{e})$. Then $\pi(v_{e})\in \partial C_{1}$ and there exists $\bar{u}\in C_{1}$ such that 
\begin{equation}
\label{7.15}
d(\bar{u},\pi(v_{e}))=1.
\end{equation}
For $i=1,2$, let $\bar{\omega}_{i}\subset C_{1}$ be the edge path connecting $\bar{u}$ and $\pi(w_{i})$. Then we lift $\bar{\omega}_{i}$ to an edge path $\omega_{i}\subset (F(\Gamma))_{x_{i}}$. (\ref{7.15}) implies we can concatenate $\omega_{1}$ and $\omega_{2}$ to obtain a path connecting $w_{1}$ and $w_{2}$ outside $St(v)$.

Now we prove the second statement of (1). Pick pair $(C,K)$ as above and $B$ be the associated $v$-branch. Since $\Gamma_K=\partial C$, for each standard geodesic $\ell\subset K$, there exists a standard geodesic $\ell'$ such that 
\begin{enumerate}
	\item $\pi(\Delta(\ell'))\in C$;
	\item $\ell'$ and $\ell$ span a 2-flat.
\end{enumerate}
 Thus $\Delta(\ell')\in B$ and $\Delta(\ell)\in \partial B$. Hence $\Delta(K)\subset \partial B$. Now we prove the other direction. Pick $u\in \partial B$. By Lemma \ref{7.4}, we can assume there exists $w'\in B$ such that $d(w',u)=1$ and $P_{w'}\cap P_v\neq\emptyset$. Then $\pi(w')\in C$ by Lemma \ref{7.9}, hence $\pi(u)\in\partial C$. Note that $P_{w'}\cap P_u\neq\emptyset$ and $P_{v}\cap P_u\neq\emptyset$. Thus $P_v\cap P_u\cap P_{w'}\neq\emptyset$ by Lemma \ref{2.1}. Pick vertex $z$ in this triple intersection. Then $z\in K$ by Lemme \ref{7.13}. Let $\ell_z$ be a standard geodesic such that $z\in \ell_z$ and $\Delta(\ell_z)=u$. Since $\pi(u)\in\partial C=\Gamma_K$, we have $\ell_z\subset K$. Thus $u\in\Delta(K)$.

The existence in (2) follows from Lemma \ref{7.4} and the above discussion. Let $K_{1}$ and $K_{2}$ be two $v$-peripheral subcomplexes of the same height such that $\Delta(K_{1})=\Delta(K_{2})=A$. Then the Hausdorff distance $d_{H}(K_{1},K_{2})<\infty$ by Lemma \ref{3.1}. If $K_{1}\cap K_{2}\neq\emptyset$, then $K_{1}=K_{2}$ since $\Gamma_{K_{1}}=\Gamma_{K_{2}}$. Otherwise there exists a horizontal edge $e\subset P_{v}$ such that the hyperplane dual to $e$ separates $K_{1}$ from $K_{2}$ (note that $K_{1}$ and $K_{2}$ are horizontal). Suppose $\bar{w}\in \Gamma$ is the label of $e$. Then $d(\bar{w},\bar{v})=1$ and $\Gamma_{K_{1}}\subset St(\bar{w})\setminus\{\bar{w}\}$ by Lemma \ref{3.1}. It follows that $lk(\bar{w})\cap lk(\bar{v})$ contains $\Gamma_{K_{1}}$, thus $lk(\bar{w})\cap lk(\bar{v})$ separates $\bar v$ from a vertex in a component of $\Gamma\setminus St(\bar v)$. This contradicts that $\Gamma$ has type II. (3) follows from (1) and (2).

To see (4), suppose $B_{1}=B_{2}$. By (1), there exist standard subcomplexes $K_i\subset P_{v_i}$ for $i=1,2$ such that $\Delta(K_i)=\partial B_{i}$. Then $K_1$ and $K_2$ are parallel, hence $\Gamma_{K_1}=\Gamma_{K_2}$. Let $\bar{v}_{i}=\pi(v_{i})$ for $i=1,2$. Then $$\Gamma_{K_1}\subset lk(\bar{v}_{1})\cap lk(\bar{v}_{2}).$$ By Lemma \ref{7.4}, there exists vertex $w\in B_{1}$ such that $\bar{w}=\pi(w)\in C$ where $C$ is a component of $\Gamma\setminus St(v_{1})$ with $\partial C=\Gamma_{K_1}$. Therefore, $\Gamma_{K_1}$ separates $\bar{v}_{1}$ from $\bar{w}$, so does $lk(\bar{v}_{1})\cap lk(\bar{v}_{2})$. This leads to a contradiction in the case $\bar{v}_{1}\neq\bar{v}_2$. Suppose $\bar{v}_1=\bar{v}_2$. Then $P_{v_1}$ and $P_{v_2}$ are standard complexes with the same support. Thus $P_{v_{1}}\cap P_{v_{2}}=\emptyset$, otherwise we would have $P_{v_{1}}=P_{v_{2}}$ and $v_{1}=v_{2}$. Let $h$ be a hyperplane separating $P_{v_1}$ and $P_{v_2}$ such that the carrier of $h$ intersects $P_{v_1}$. Then the label of edges dual to $h$, denoted by $\bar{v}_h$, satisfies $d(\bar{v}_h,\bar{v}_1)\ge 2$. It follows from Lemma \ref{3.2} that $\Gamma_{K_1}\subset lk(\bar{v}_h)$. Thus $lk(\bar{v}_{1})\cap lk(\bar{v}_h)$ separates $\Gamma$, which is a contradiction.

To see (5), first we assume $d(\bar{v}_{1},\bar{v}_{2})\neq 0$. Since $\Gamma$ is of type II, for any component $C$ of $\Gamma\setminus St(\bar{v}_{1})$,
\begin{equation}
\label{7.16}
\partial C\setminus (lk(\bar{v}_{1})\cap lk(\bar{v}_{2}))\neq\emptyset.
\end{equation}
Let $B$ be a $v_{1}$-branch. Then (1) and (2) imply there exist a standard subcomplex $K\subset P_{v_{1}}$ and a component $C'$ of $\Gamma\setminus St(\bar{v}_{1})$ such that $\partial B=\Delta(K)$ and $\Gamma_{K}=\partial C'$. Hence $\partial C'=\pi(\partial B)\cap \Gamma$ (recall that we have identified $\Gamma$ with the 1-skeleton of $F(\Gamma)$). It follows from (\ref{7.16}) that $$\partial B\setminus (lk(v_{1})\cap lk(v_{2}))\neq\emptyset,$$ otherwise we would have $$\partial C'\subset\pi(\partial B)\subset\pi (lk(v_{1})\cap lk(v_{2}))\subset lk(\bar{v}_{1})\cap lk(\bar{v}_{2}).$$ Then every vertex in $B$ can be connected to $v_{1}$ outside $lk(v_{1})\cap lk(v_{2})$ and (5) follows.

If $\bar{v}_{1}=\bar{v}_{2}$, then $P_{v_{1}}\cap P_{v_{2}}=\emptyset$. Hence $d(v_{1},v_{2})\ge 2$ and $$lk(v_{1})\cap lk(v_{2})=St(v_{1})\cap St(v_{2}).$$ Let $h$ and $\bar{h}_v$ be as in the proof of (4). Let $\ell_{h}$ be a standard geodesic dual to $\ell$ and let $v_{h}=\Delta(\ell_{h})$. Note that $d(v_h,v_1)\ge 2$ and $\pi(v_h)=\bar{v}_h\neq \bar{v}_1$. It suffices to prove $$St(v_{1})\cap St(v_{2})\subset St(v_{1})\cap St(v_{h}),$$ since this reduces the current case to the previous case. Let $k$ be a vertex in $St(v_{1})\cap St(v_{2})$. Then $P_{k}$ has nontrivial intersection with both $P_{v_1}$ and $P_{v_2}$. Since $h$ separates $P_{v_1}$ and $P_{v_2}$, we have $P_k\cap h\neq\emptyset$. Hence $\ell_h\subset P_k$ and $d(k,v_{h})\le 1$. 
\end{proof}

\subsection{Quasi-isometry invariance of type II and weak type II} The main goal of this subsection is Corollary~\ref{7.22}, where it is shown that weak type II and type II are quasi-isometry invariants for right-angled Artin groups.
\begin{lem}
\label{7.17}
If $\Gamma$ is of weak type II, then
\begin{enumerate}
\item There is no non-adjacent transvection in $\out(G(\Gamma))$.
\item $\mathcal{P}(\Gamma)$ is of weak type II.
\end{enumerate}
\end{lem}

\begin{proof}
(1) follows directly from the definition. To see (2), pick distinct vertices $v_1,v_2\in\mathcal{P}(\Gamma)$ such that $d(v_{1},v_{2})=2$. For $i=1,2$, let $\bar{v}_i=\pi(v_i)$. The case $d(\bar{v}_1,\bar{v}_2)=2$ and $\bar{v}_1=\bar{v}_2$ has been dealt with in Corollary \ref{7.14} (5). Now we assume $d(\bar{v}_1,\bar{v}_2)=1$. If $P_{v_{1}}\cap P_{v_{2}}\neq\emptyset$, then it is a standard complex whose support is the intersection of the supports of $P_{v_1}$ and $P_{v_2}$, which is $St(\bar{v}_1)\cap St(\bar{v}_{2})$. Thus $d(v_{1},v_{2})=1$, which yields a contradiction. So $P_{v_{1}}\cap P_{v_{2}}=\emptyset$ and we have reduced to second case of Corollary \ref{7.14} (5).
\end{proof}

\begin{lem}
\label{5.1}
\cite[Lemma 5.1]{raagqi1}
Let $\Gamma$ be a finite simplicial graph. Pick a vertex $\bar{w}\in\Gamma$ and let $\Gamma_{\bar{w}}$ be the minimal stable subgraph containing $\bar{w}$. Denote $\Gamma_{1}=lk(\bar{w})$ and $\Gamma_{2}=lk(\Gamma_{1})$ $($see Section 2.1 for definition of links$)$, then either of the following is true:
\begin{enumerate}
\item $\Gamma_{\bar{w}}$ is a clique. In this case $St(\bar{w})$ is a stable subgraph. %Moreover, $St(\bar{w})$ is the full subgraph spanned by $\Gamma_{\bar{w}}$ and $\Gamma^\perp_{\bar{w}}$.
\item Both $\Gamma_{1}$ and the join $\Gamma_{1}\circ\Gamma_{2}$ of $\Gamma_1$ and $\Ga_2$ are stable subgraphs of $\Gamma$. Moreover, $\Gamma_{2}$ is disconnected.
\end{enumerate}
\end{lem}
%Note that the moreover part in Assertion (1) is not stated in the content of \cite[Lemma 5.1]{raagqi1}, however, it is stated explicitly in the proof of \cite[Lemma 5.1]{raagqi1} (1).

\begin{lem}
\label{7.18}
Suppose $\mathcal{P}(\Gamma)$ is of weak type II. Let $q:G(\Gamma)\to G(\Gamma')$ be a quasi-isometry. Then $q$ induces a simplicial isomorphism $q_{\ast}:\mathcal{P}(\Gamma)\to\mathcal{P}(\Gamma')$.
\end{lem}

The proof is a variation of \cite[Theorem 5.3]{raagqi1}.

\begin{proof}
Let $\Gamma=\Gamma_{1}\circ\Gamma_{2}$ be any join decomposition. Then $\Gamma$ is of weak type II if and only if each $\Gamma_{i}$ is of weak type II. So in the light of Theorem \ref{2.13}, we only need to focus on the case when $\Gamma$ is irreducible and is not a clique. In this case $\Gamma'$ is also irreducible and is not a clique, hence $diam(\mathcal{P}(\Gamma))=\infty$ and $diam(\mathcal{P}(\Gamma'))=\infty$ \cite[Lemma 26 (5)]{kim2013embedability}.

Let $q_{\ast}:\mathcal{P}(\Gamma)\to\mathcal{P}(\Gamma')$ be the simplicial embedding in Theorem \ref{4.8}. Suppose $q_{\ast}$ is not surjective. Then there exists a vertex $w'\in\mathcal{P}(\Gamma')$ which is not in the image of $q_{\ast}$. Let $\bar{w}'=\pi(w')$ and let $\ell'$ be a standard geodesic $\Delta(\ell')=w'$. We apply Lemma \ref{5.1} to $\bar{w}'\in\Gamma'$. If case (1) is true, let $F'$ be the standard flat in $X(\Gamma')$ such that $\ell'\subset F'$ and $\Gamma_{F'}=\Gamma_{w'}$. Since $\Gamma_{w'}$ is stable, $$w'\in\Delta(F')\subset q_{\ast}(\mathcal{P}(\Gamma')),$$ which is a contradiction.

If case (2) is true, let $\Gamma'_{1}=lk(\bar{w}')$ and $\Gamma'_{2}=lk(\Gamma'_{1})$. Take $K'_{1}$ and $K'$ to be the standard subcomplexes in $X(\Gamma')$ such that (1) $\Gamma_{K'_{1}}=\Gamma'_{1}$ and $\Gamma_{K'}=\Gamma'_{1}\circ\Gamma'_{2}$; (2) $\ell'\subset K'$ and $K'_{1}\subset K'$. Set $M'_{1}=\Delta(K'_{1})$ and $M'=\Delta(K')$. Let $K'_{2}$ be an orthogonal complement of $K'_{1}$ in $K'$, i.e. $K'_{2}$ is a standard subcomplex such that $\Gamma_{K'_{2}}=\Gamma'_{2}$ and $K'=K'_{1}\times K'_{2}$. It follows that $M'$ has a join decomposition $M'=M'_{1}\ast M'_{2}$ for $M'_{2}=\Delta(K'_{2})$. By construction, $w'\in M'$ and $lk(w')=M'_{1}$.

Since $K'$ and $K'_{1}$ are stable, there exist stable standard subcomplexes $K$ and $K_{1}$ in $X(\Gamma)$ such that the Hausdorff distances satisfy $d_{H}(q(K),K')<\infty$ and $d_{H}(q(K_1),K'_1)<\infty$. Moreover, by applying Theorem \ref{2.13} to the quasi-isometry between $K$ and $K'$, there exists standard subcomplex $K_{2}\in K$ such that $K=K_{1}\times K_{2}$ and $K_{2}$ is quasi-isometric to $K'_{2}$, thus $\Gamma_{K_{2}}$ is also disconnected. Let $M_{i}=\Delta(K_{i})\subset\mathcal{P}(\Gamma)$ for $i=1,2$ and $M=M_{1}\ast M_{2}=\Delta(K)$. Then $q_{\ast}^{-1}(M'_{1})=M_{1}$ by (\ref{4.10}) and (\ref{4.11}). Since $lk(w')=M'_{1}$ and $w'\notin q_{\ast}(\mathcal{P}(\Gamma))$,
\begin{equation}
\label{7.185}
q^{-1}_{\ast}(St(w'))=M_{1}.
\end{equation}

Let $I=q_{\ast}(\mathcal{P}(\Gamma))$. Then $I$ is $D$-dense in $\mathcal{P}(\Gamma')$ for some constant $D>0$, i.e. each vertex of $\mathcal{P}(\Gamma)$ is at combinatorial distance $\le D$ from a vertex in $I$. To see this, it suffices to show $\Gamma'$ contains a stable clique, but this follows from the existence of stable clique in $\Gamma$. 

We claim every $w'$-tier contains vertices arbitrarily far from $w'$. To see this, let $\ell'$ be a standard geodesic such that $\Delta(\ell')=w'$. We consider the action $G(\Ga)\acts X(\Ga)$ by deck transformations and the induced action $G(\Ga)\acts\P(\Ga)$. Then the stabilizer of $\ell'$ is isomorphic to $\Z$. Moreover, this copy of $\Z$ acts transitively on the collection of $w'$-tiers. Now the claim follows from $diam (\mathcal{P}(\Gamma'))=\infty$.

We pick vertices $w'_{1},w'_{2}\in \mathcal{P}(\Gamma')$ such that they are not in the same $w'$-tier and $d(w'_{i},w')>D+5$ for $i=1,2$. Let $u'_{i}\in I$ be a vertex such that $d(u'_{i},w'_{i})\le D$. Then $u'_{1}$ and $u'_{2}$ is separated by $St(w')$ and $$d(u'_{i},I\cap St(w'))>4.$$ Define $u_{i}=q^{-1}_{\ast}(u'_{i})$, then $u_{1}$ and $u_{2}$ are in different components of $$\mathcal{P}(\Gamma)\setminus q^{-1}_{\ast}(St(w'))=\mathcal{P}(\Gamma)\setminus M_{1},$$ and 
\begin{equation}
\label{7.19}
d(u_{i},M_{1})>4.
\end{equation}

Since $\Gamma_{K_{2}}$ is disconnected, there exists vertex $v_{1},v_{2}\in M_{2}$ such that $d(v_{1},v_{2})=2$. Recall that $M=M_{1}\ast M_{2}\subset\mathcal{P}(\Gamma)$, so $$M_{1}\subset lk(v_{1})\cap lk(v_{2}).$$ Moreover, $$d(u_{i},lk(v_{1})\cap lk(v_{2}))>0$$ by (\ref{7.19}), so $u_{1}$ and $u_{2}$ are separated by $lk(v_{1})\cap lk(v_{2})$, which is contradictory to our assumption on $\mathcal{P}(\Gamma)$. So $q_{\ast}$ must be surjective.
\end{proof}

\begin{lem}
\label{7.20}
If $\mathcal{P}(\Gamma)$ is of weak type II, then $\Gamma$ is of weak type II.
\end{lem}

\begin{proof}
Suppose $\Gamma$ is not of weak type II. Then there exist vertices $\bar{v}_{1},\bar{v}_{2}\in\Gamma$ such that $d(\bar{v}_{1},\bar{v}_{2})=2$ and $$\Gamma\setminus lk(\bar{v}_{1})\cap lk(\bar{v}_{2})$$ is disconnected. Then we can find component $C$ of $\Gamma\setminus St(\bar{v}_{1})$ such that $$\partial C\subset lk(\bar{v}_{1})\cap lk(\bar{v}_{2}).$$ Pick vertex $x_{0}\in P_{v_1}$ and let $K\subset P_{v_{1}}$ be the standard subcomplex with support $=\partial C$ that contains $x_{0}$. Pick vertex $\bar{v}_{3}\in C$ (it is possible that $\bar{v}_3=\bar{v}_2$). For $i=1,2,3$, let $v_{i}\in (F(\Gamma))_{x_{0}}$ be the lift of $\bar{v}_{i}$. Then $$\Delta(K)\subset St(v_{1})\cap St(v_{2})=lk(v_{1})\cap lk(v_{2}).$$ Let $B$ be the $v_{1}$-branch that contains $v_{3}$. We claim $\partial B=\Delta (K)$, which then implies $v_{1}$ and $B$ are in different components of $\mathcal{P}(\Gamma)\setminus (lk(v_{1})\cap lk(v_{2}))$.

Note that $\Delta(K)\subset \partial B$ follows from the argument in (1) of Corollary \ref{7.14}. To see the other direction, pick $w_{1}\in \partial B$ and $w_{2}\in B$ such that $d(w_{1},w_{2})=1$. Let $\ell_{3}$ be the geodesic such that $x_{0}\in \ell_{3}$ and $\Delta(\ell_{3})=v_{3}$. If $P_{w_{2}}\cap P_{v_{1}}=\emptyset$, then by the argument in Lemma \ref{7.9}, we can find an edge path $\omega\in X(\Gamma)\setminus P_{v_{1}}$ connecting $x\in \ell_{3}\setminus\{x_{0}\}$ and $y\in P_{w_{2}}\cap P_{w_{1}}$. Let $\omega_{1}\subset P_{w_{1}}$ be a horizontal edge path connecting $y$ and some vertex $z\in P_{w_{1}}\cap P_{v_{1}}$. Such path exists since $P_{w_{1}}\cap P_{v_{1}}$ contains a standard geodesic $\ell$ with $\Delta(\ell)=w_1$. We can also assume $\omega_1\cap P_{v_1}=\{z\}$. Let $e\subset\omega_{1}$ be the edge containing $z$ and $\bar{v}_{e}$ be the label of $e$. Since $\omega_1$ is horizontal in $P_{w_1}$,
\begin{equation}
\label{7.21}
d(\bar{v}_{e},\pi(w_{1}))=1.
\end{equation}
Let $\omega'$ be the edge path obtained by (1) going from $x_{0}$ to $x$ along $\ell_{3}$; (2) going from $x$ to $y$ along $\omega$; (3) going form $y$ to $z$ along $\omega_{1}$. By applying Lemma \ref{7.10} and Lemma \ref{7.11} to $\omega'$, we have $\bar{v}_{e}\in C$ and $z\in K$. Hence $\pi(w_{1})\in\partial C$ by (\ref{7.21}). This, together with $z\in P_{w_1}$ and $z\in K$ imply $w_{1}\in\Delta(K)$. If $P_{w_{2}}\cap P_{v_{1}}\neq\emptyset$, we still have $w_{1}\in\Delta(K)$ by the proof of the second statement of Corollary \ref{7.14} (1). Thus $\partial B\subset\Delta(K)$.
\end{proof}

Actually the above argument also shows that if $\mathcal{P}(\Gamma)$ is of type II, then $\Gamma$ is of type II. The following corollary follows from (5) of Corollary \ref{7.14}, Lemma \ref{7.17}, Lemma \ref{7.18} and Lemma \ref{7.20}.

\begin{cor}
\label{7.22}
$\Gamma$ is of (weak) type II if and only if $\mathcal{P}(\Gamma)$ is of (weak) type II. If $G(\Gamma')$ is quasi-isometric to $G(\Gamma)$, then $\Gamma'$ is also of (weak) type II.
\end{cor}

\subsection{RAAGs of weak type I}
\label{sec:3.4}
In this subsection we introduce the notion of RAAGs of weka type I and prove quasi-isometric classification results for them.
\begin{definition}
\label{7.24}
A finite simplicial graph $\Gamma$ is of \textit{weak type I} if:
\begin{enumerate}
\item $\Gamma$ is of weak type II.
\item $\Gamma$ does not contain any separating closed star.
\end{enumerate}
$G(\Ga)$ is of \textit{weak type I} if $\Ga$ is of weak type I.
\end{definition}

It is immediate from the definition that if $\Gamma=\Gamma_{1}\circ\Gamma_{2}$, then $\Gamma$ is of weak type I if and only if $\Gamma_{1}$ and $\Gamma_{2}$ are of weak type I.

\begin{lem}
\label{7.25}
$G(\Gamma)$ is of weak type I if and only if
\begin{enumerate}
\item $\Gamma$ does not contain any separating closed star.
\item There do not exist vertices $\bar{v},\bar{w}\in\Gamma$ such that $d(\bar{v},\bar{w})=2$ and $\Gamma=St(\bar{v})\cup St(\bar{w})$.
\end{enumerate}
\end{lem}

Thus Definition \ref{1.1} and Definition \ref{7.24} are consistent.
\begin{proof}
For the only if direction, note that if $\Gamma=St(\bar{v})\cup St(\bar{w})$ with $d(\bar{v},\bar{w})=2$, then $lk(\bar{v})\cap lk(\bar{w})$ separates $\Gamma$. For the if direction, we follow the argument in \cite[Theorem 5.3]{raagqi1}. Suppose there exist vertices $\bar{v}_{1}$ and $\bar{v}_{2}$ such that $lk(\bar{v}_{1})\cap lk(\bar{v}_{2})$ separates $\Gamma$. Let $\{C_{j}\}_{j=1}^{d}$ be the connected components of $F(\Gamma)\setminus lk(\bar{v}_{1})\cap lk(\bar{v}_{2})$. Then at most one of $C_{j}$ is contained in $St(\bar{v}_{1})$. If $d\ge 3$, $St(\bar{v}_{1})$ would separate $F(\Gamma)$, contradiction. Suppose $d=2$. At least one of $C_1$ and $C_2$ is inside $St(\bar{v}_1)$, otherwise $St(\bar{v}_{1})$ will separate $F(\Gamma)$. We assume $C_1\subset St(\bar{v}_1)$. Thus $\bar{v}_2\in C_2$. Similarly, at least one of $C_1$ and $C_2$ is inside $St(\bar{v}_2)$. So we must have $C_2\subset St(\bar{v}_2)$. Hence $F(\Gamma)=St(\bar{v}_{1})\cup St(\bar{v}_{2})$, contradiction again. 
\end{proof}

\begin{thm}
\label{weak type I visible}
Let $\Gamma_{1}$ be of weak type I. Then any simplicial isomorphism $s:\mathcal{P}(\Gamma_{1})\to\mathcal{P}(\Gamma_{2})$ is visible. In particular, if $q:G(\Gamma_{1})\to G(\Gamma_{2})$ is a quasi-isometry, then $q$ will induce a visible map $q_{\ast}:\mathcal{P}(\Gamma_{1})\to\mathcal{P}(\Gamma_{2})$. In this case, $\Gamma_{2}$ is of weak type II, hence $\out(\Gamma_{2})$ does not contain non-adjacent transvections.
\end{thm}

\begin{proof}
Let $p$, $\{F_{i}\}_{i=1}^{n}$ and $\{F'_{i}\}_{i=1}^{n}$ be as in Definition \ref{visible}. Suppose $s$ is not visible. Then there exist $i\neq j$ and a hyperplane $h$ separating $F'_{i}$ and $F'_{j}$. Let $\ell'$ be a standard geodesic dual to $h$ and let $v'=\Delta(\ell')$. Then there exists $v'_{1}\in\Delta(F'_{i})$ and $v'_{2}\in\Delta(F'_{j})$ such that they are in different $v'$-tiers. Thus $St(v')$ separates $v'_{1}$ from $v'_{2}$ by Lemma \ref{5.17}. Let $v=s^{-1}(v')$. Then $$(F(\Gamma))_{p}\setminus St(v)$$ is disconnected, hence $v\in(F(\Gamma))_{p}$ by Lemma \ref{7.8}. This would imply $F(\Gamma)$ has a separating closed star, which is a contradiction. The second statement follows from Lemma \ref{7.18}.
\end{proof}

\begin{thm}
\label{7.26}
Suppose $G(\Gamma_{1})$ and $G(\Gamma_{2})$ are groups of weak type I. Then they are quasi-isometric if and only if they are isomorphic.
\end{thm}

\begin{proof}
Let $q:G(\Ga_1)\to G(\Ga_2)$ be a quasi-isometry. By Theorem \ref{weak type I visible}, $q$ induces a visible simplicial isomorphism $q_{\ast}:\P(\Ga)\to\P(\Ga)$. Pick vertex $x_1\in X(\Ga_1)$. Then the visibility implies $$q_{\ast}((F(\Ga_1))_{x_1})\subset (F(\Ga_2))_{x_2}$$ for some vertex $x_2\in X(\Ga_2)$. This induces a graph embedding $\Ga_1\to \Ga_2$. By consider the quasi-isometry inverse of $q$, we obtain another graph embedding $\Ga_2\to \Ga_1$. Hence $\Ga_1\cong \Ga_2$ and $G(\Ga_1)\cong G(\Ga_2)$.
\end{proof}

Though the definition of weak type I looks technical, it is actually a natural condition to consider for the following reason.

\begin{thm}
\label{7.28}
The following are equivalent:
\begin{enumerate}
\item $G(\Gamma)$ is of weak type I.
\item There do not exist vertex $x\in X(\Gamma)$ and vertex $v\in\mathcal{P}(\Gamma)$ such that $St(v)$ separates $(F(\Gamma))_{x}$.
\item Every element in $Aut(\mathcal{P}(\Gamma))$ is visible.
\end{enumerate}
\end{thm}

We will not need $(3)\Rightarrow (2)$ in the rest of the paper.
\begin{proof}
$(1)\Leftrightarrow(2)$ follows from Lemma \ref{7.29} below and $(2)\Rightarrow(3)$ follows from the proof of Theorem \ref{weak type I visible}. It suffices to prove $(3)\Rightarrow (2)$. We argue by contradiction and suppose $v_{1},v_{2}$ are vertices in different components of $(F(\Gamma))_{x}\setminus St(v)$. Then Lemma \ref{7.9} implies $v_{1}$ and $v_{2}$ are in different $v$-branches. For $i=1,2$, let $B_{i}$ be the $v$-branch that contains $v_{i}$. Let $\ell$ be a standard geodesic such that $\Delta(\ell)=v$ and $x\in \ell$ and pick $\alpha\in G(\Gamma)$ to be an non-trivial element such that $\alpha(\ell)=\alpha$. Let $\alpha_{\ast}:\P(\Ga)\to\P(\Ga)$ be the induced map. Then $\alpha_{\ast}$ fixes every point in $St(v)$. Thus there exists $f\in Aut(\mathcal{P}(\Gamma))$ such that 
\begin{enumerate}
	\item $f$ fixes every vertex in $\mathcal{P}(\Gamma)\setminus(B_{1}\cup\alpha_{\ast}(B_{1}))$;
	\item $f|_{B_{1}}=\alpha_{\ast}|_{B_{1}}$ and $f|_{\alpha_{\ast}(B_{1})}=\alpha^{-1}_{\ast}|_{\alpha_{\ast}(B_{1})}$. 
\end{enumerate}

We claim $f$ is not visible. To see this, for $i=1,2$, pick maximal standard flat $F_{i}$ such that $x\in F_{i}$ and $v_{i}\in\Delta(F_{i})$. Then $$f(\Delta(F_{1}))=\alpha_{\ast}(\Delta(F_{1}))$$ and $f(\Delta(F_{2}))=\Delta(F_{2})$, thus the maximal standard flats corresponding to $f(\Delta(F_{1}))$ and $f(\Delta(F_{2}))$ are separated by a hyperplane dual to $\ell$, hence have empty intersection.
\end{proof}

\begin{lem}
\label{7.29}
$\Gamma$ is of weak type II if and only if there do not exist vertex $x\in X(\Gamma)$ and vertex $v\in \mathcal{P}(\Gamma)\setminus(F(\Gamma))_{x}$ such that $(F(\Gamma))_{x}\setminus St(v)$ is disconnected.
\end{lem}

\begin{proof}
By Lemma \ref{7.8}, it suffices to prove the if direction. Suppose $\Gamma$ is not of weak type II. Let $\{\bar{v}_{i}\}_{i=1}^{3}$, $\{v_{i}\}_{i=1}^{3}$ and $x_{0}\in X(\Gamma)$ be as in Lemma \ref{7.20}. For $i=1,2$, let $\ell_{i}$ be the standard geodesic such that $x_{0}\in \ell$ and $\Delta(\ell_{i})=v_{i}$. Pick vertex $x'_{0}\neq x_{0}$ in $\ell_{1}$ and let $\ell'_{2}$ be the standard geodesic such that $x'_{0}\in \ell'_{2}$ and $\pi(\Delta(\ell'_{2}))=\bar{v}_{2}$. Then $d(v'_{2},v_{1})=2$ where $v'_{2}=\Delta(\ell'_{2})$, in particular $x_{0}\notin P_{v'_{2}}$, hence $v'_{2}\notin (F(\Gamma))_{x_{0}}$. 

Since $P_{v'_{2}}$ and $x_{0}$ are separated by some hyperplane dual to $\ell_{1}$, thus by Lemma \ref{3.2}, $$St(v'_{2})\cap (F(\Gamma))_{x_{0}}\subset St(v_{1}).$$ Recall that $d(\bar{v}_{3},\bar{v}_{1})\ge 2$, then $v_{3}\in (F(\Gamma))_{x_{0}}\setminus St(v_{1})$. It follows that $v_{3}\notin St(v'_{2})$.

We claim that $v_{3}$ and $v_{1}$ are in different components of $\mathcal{P}(\Gamma)\setminus St(v'_{2})$, which then implies $(F(\Ga))_{x_0}\setminus St(v'_2)$ is disconnected. Lemma \ref{7.20} already implies $v_{1}$ and $v_{3}$ are separated by $lk(v_{1})\cap lk(v_{2})$. Let $\alpha\in G(\Gamma)$ be the left translation such that $\alpha(x_{0})=x'_{0}$. Then $\alpha(\ell_2)=\ell'_2$. Now we pass to the induced action $G(\Ga)\acts\P(\Ga)$, then $\alpha(v_2)=v'_2$. Since $\alpha$ fixes $St(v_{1})$, we have $$lk(v_{1})\cap lk(v_{2})=\alpha(lk(v_{1})\cap lk(v_{2}))=lk(v_{1})\cap lk(v'_{2}).$$ So $lk(v_{1})\cap lk(v'_{2})$ separates $v_{1}$ from $v_{3}$ and the claim follows.
\end{proof}

\section{Quasi-isometric invariance of $v$-branches and peripheral subcomplexes}

In this section we collection several observations on quasi-isometric invariance of $v$-branches and $v$-peripheral subcomplexes.
While the content of this section is closely related to Section~\ref{sec:extension complex structure}, it will not be used until Section~\ref{sec:shuffle}.

\subsection{Quasi-isometric invariance of $v$-branches}
Let $G(\Gamma)$ be a RAAG of weak type II, and let $q:G(\Gamma)\to G(\Gamma')$ be a quasi-isometry. Then $G(\Gamma')$ is of weak type II by Corollary~\ref{7.22}. Then both $\out(G(\Gamma))$ and $\out(G(\Gamma'))$ do not admit non-adjacent transvections (Lemma~\ref{7.17}). Thus Theorem~\ref{4.8} implies that $q$ induces a simplicial isomorphism $q_*:\mathcal{P}(\Gamma)\to\mathcal{P}(\Gamma')$. Take a vertex $v$ of $\mathcal{P}(\Gamma)$. Then $q_*$ induces a bijection between $v$-branches and $q_*(v)$-branches. Later in Section~\ref{sec:shuffle} we want to use this bijection as a quasi-isometry invariant of $q$. However, a priori this is problematic as the simplicial isomorphism $q_*$ is not canonically determined by $q$ (as in the construction of the proof of Theorem~\ref{4.8}), so strictly speaking, $q_*(v)$ is not even well-defined. The goal of this subsection is to clarify this point. 
\begin{lem}
	\label{lem:unique}
	Suppose $\Gamma$ is of weak type II. Let $q:G(\Gamma)\to G(\Gamma')$ be a quasi-isometry and let $\tilde{q}_*$ be as in Lemma~\ref{iso}. For $i=1,2$, let $(q_i)_*:\mathcal{P}(\Gamma)\to \mathcal{P}(\Gamma')$ be a simplicial isomorphism such that $(q_i)_*(s)=\tilde q_*(s)$ for any stable simplex $s\subset\mathcal{P}(\Gamma)$.
	\begin{enumerate}
		\item For any vertex $v\in\mathcal{P}(\Gamma)$, we have $(q_1)_*(St(v))=(q_2)_*(St(v))$.
		\item For $i=1,2$, we know $d_{H}(q(P_{v}),P_{(q_i)_{\ast}(v)})<\infty$.
		\item If $\Gamma$ is of type II and $G(\Gamma)$ is centerless, then $(q_1)_*=(q_2)_*$.
	\end{enumerate} 
\end{lem}

\begin{proof}
	Let $\bar v\in \Gamma$ be the label of $v\in\mathcal{P}(\Gamma)$. Then it follows from Lemma \ref{7.17}, Lemma \ref{4.7} and Lemma \ref{5.1} that $St(\bar{v})$ is a stable subgraph. Thus $St(v)$ is a stable subcomplex of $\mathcal{P}(\Gamma)$. By Corollary~\ref{7.22} and Lemma~\ref{7.17}, both $\out(G(\Gamma))$ and $\out(G(\Gamma'))$ do not admit non-adjacent transvection. Thus Corollary~\ref{4.9} implies $$(q_1)_*(St(v))=(q_2)_*(St(v))=\tilde q_*(St(v)).$$ Then Assertion (1) follows. This also implies Assertion (2) as $\Delta(P_v)=St(v)$ and $\Delta(P_{(q_i)_{\ast}(v)})=St((q_i)_{\ast}(v))$.

	Now we prove Assertion (3). By Corollary~\ref{7.22}, $\Gamma'$ is of type II.
	As $$(q_i)_*(St(v))=St((q_i)_*(v)),$$ we know $$St(w_1)=St(w_2)$$ where $w_i$ is defined as $(q_i)_*(v)$. Then $$St(\bar w_1)=St(\bar w_2).$$ We must have $\bar w_1=\bar w_2$, otherwise $lk(\bar w_1)\cap lk(\bar w_2)$ separates $\bar w_1$ from any vertex outside $St(\bar w_1)$, which contradicts that $\Gamma'$ is of type II (note that $G(\Gamma)$ is centerless implies that $G(\Gamma')$ is centerless, so $G(\Gamma')$ contains at least one vertex outside $St(\bar w_1)$). Thus $w_1=w_2$.
\end{proof}

\begin{example}
We give an example of Lemma~\ref{lem:unique} (1).
Let $\Gamma$ be the graph obtained by gluing a pentagon and the 1-skeleton of a 3-simplex along an edge. Let $a$ and $b$ be the two vertices of $\Gamma$ which is outside the pentagon. Take a quasi-isometry $q:G(\Gamma)\to G(\Gamma)$. Let $v\in\mathcal{P}(\Gamma)$ be a vertex labeled by $a$, and let $s$ be the maximal simplex of $\mathcal{P}(\Gamma)$ containing $v$. Then $s$ is a stable subcomplex. We know $q_*(a)\in \tilde q_*(s)$. However, $q_*(a)$ could be either the vertex in $\tilde q_*(s)$ labeled by $a$ or the vertex in $\tilde q_*(s)$ labeled by $b$. Note that the star of these two possible values of $q_*(a)$ are equal.

We now give an example of Lemma~\ref{lem:unique} (3). Let $\Gamma$ be the graph obtained by gluing a pentagon and a triangle along an edge. Let $a$ the vertex of $\Gamma$ outside the pentagon. Take a quasi-isometry $q:G(\Gamma)\to G(\Gamma)$. Let $v\in\mathcal{P}(\Gamma)$ and $s\subset \mathcal{P}(\Gamma)$ be as before. Then $q_*(a)\in \tilde q_*(s)$. Then $q_*(a)$ can only be the vertex in $\tilde q_*(s)$ labeled by $a$. Note that if $v$ is not labeled by $a$, then $\{v\}$ is a stable subcomplex of $\mathcal{P}(\Gamma)$. Then for any standard line $\ell$ with $\Delta(\ell)=v$, $q(\ell)$ is at finite Hausdorff distance to a standard line $\ell'$ with $\Delta(\ell')=q_*(v)$. However, if $v$ is labeled by $a$, in general $q(\ell)$ might not be at finite Hausdorff distance from any standard line, due to the transvection at $a$, yet $q_*(v)$ is canonically defined.
\end{example}

Returning to the discussion before Lemma~\ref{lem:unique}, even though $q_*(v)$ depends on the choices in the proof of Theorem~\ref{4.8}, Lemma~\ref{lem:unique} implies that the set $St(q_*(v))$ does not depend on these choices, so is the set of $q_*(v)$-branches.

\begin{cor}
	\label{cor:unique}
	Suppose $\Gamma$ is of type II and $G(\Gamma)$ is centerless. Let $q:G(\Gamma)\to G(\Gamma')$ be a quasi-isometry and let $\tilde{q}_*$ be as in Lemma~\ref{iso}. Then there is a unique simplicial isomorphism $q_*:\mathcal{P}(\Gamma)\to \mathcal{P}(\Gamma')$  such that $(q_i)_*(s)=\tilde q_*(s)$ for any stable simplex $s\subset\mathcal{P}(\Gamma)$. In particular, for any quasi-isometric inverse $q^{-1}:G(\Gamma')\to G(\Gamma)$, we have $(q^{-1})_*=(q_*)^{-1}$.
\end{cor}

\begin{proof}
The first statement of the corollary is an immediate consequence of Lemma~\ref{lem:unique} (3). For the in particular part, note that $\widetilde{q^{-1}}_*\circ \tilde q_*(s)=s$. Then $q^{-1}_*\circ  q_*$ must be the identity map by Lemma~\ref{lem:unique} (3) again. Similarly, $q_*\circ q^{-1}_*$ is identity.
\end{proof}

\subsection{Correspondence between $v$-branches and subcomplexes of $X(\Gamma)$} 
\label{sec:correspondence}
The main goal of this subsection is Proposition \ref{7.23}, where we establish 1-1 correspondence between $v$-branches and components of $X(\Gamma)\setminus P_v$, and prove a quasi-isometric invariance result.

Let $\Gamma$ be an arbitrary simplicial graph (not necessarily of weak type II). Take vertex $v\in\mathcal{P}(\Gamma)$.
For any $v$-branch $B$, we denote the full subcomplex of $\mathcal{P}(\Gamma)$ spanned by vertices in $B$ and $\partial B$ by $\bar{B}$. For any component $L$ of $X(\Gamma)\setminus P_{v}$, we use $\partial L$ to denote the full subcomplex of $X(\Gamma)$ spanned by vertices outside $L$ which are adjacent to some vertex in $L$, and use $\bar{L}$ to denote the full subcomplex spanned by vertices in $L$ and $\partial L$. Note that $\bar{B}$ may not be the closure of $B$ and $\bar{L}$ may not be the closure of $L$.

For any subcomplex $K\subset X(\Gamma)$, let $\{F_{\lambda}\}_{\lambda\in \Lambda}$ be the collection of standard flats in $K$ and define $\Delta(K)=\cup_{\lambda\in\Lambda} \Delta(F_{\lambda})$. 

\begin{lem}
	\label{full subcomplex}
	If $K$ is a convex subcomplex, then $\Delta(K)$ is a full subcomplex of $\P(\Ga)$.
\end{lem}

\begin{proof}
	Let $F\subset X(\Ga)$ be a standard flat such that vertices of $\Delta(F)$ are in $\Delta(K)$. Suppose $(F',K')=\inc (F,K)$. Lemma \ref{2.6} (4) implies that every standard geodesic of $F$ is contained in a $R$-neighbourhood of $F'$ for some $R>0$. However, $F'$ is a convex subcomplex of $F$, so actually $F=F'$. Hence $\Delta(F)\subset\Delta(K)$.
\end{proof} 

\begin{lem}
	\label{lem:closure}
Let $\Gamma$ be arbitrary. Pick vertex $v\in\mathcal{P}(\Gamma)$ and let $\bar{v}=\pi(v)$. Let $L$ be a component of $X(\Gamma)\setminus P_{v}$. Then
	\begin{enumerate}
		\item $\partial L$ is a $v$-peripheral subcomplex of $X(\Gamma)$. Moreover, the topological boundary $\partial^{Top} L$ of $L$ (i.e. $\partial^{Top} L$ is the closure of $L$ in $X(\Gamma)$ minus $L$) is contained in $\partial L$, and $\partial^{Top} L$ contains the 1-skeleton of $\partial L$.
		\item $\bar{L}=L\cup \partial L$, and $\bar L$ is a convex subcomplex of $X(\Gamma)$.
	\end{enumerate}
\end{lem}

\begin{proof}
		To see (1), note that Lemma \ref{7.10} implies there exists a component $C$ of $\Ga\setminus St(\bar{v})$ such that the label of any edge which connects a vertex in $L$ and a vertex outside $L$ is inside $C$. Pick a vertex in $X(\Ga)\setminus L$ which is adjacent to some vertex in $L$, and let $K$ be the $v$-peripheral subcomplex of type $\partial C$ that contains this vertex. Then Lemma \ref{7.11} implies vertex set of $\partial L$ is contained in $K$, hence $\partial L\subset K$. Note that $\partial^{Top} L$ is a subcomplex whose vertex set is the same as $\partial L$, hence $\partial^{Top} L\subset \partial L$. On the other hand, the proof of the first statement of Corollary \ref{7.14} (1) implies every edge of $K$ is contained in $\partial^{Top} L$. Thus (1) follows.
	
	To see (2), note that $L\cup \partial L$ is a subcomplex.  We claim for a vertex $x\in L\cup \partial L$, if a collection of mutual orthogonal edges emanating from $x$ are contained in $L\cup \partial L$, then the cube spanned by these edges is contained in $L\cup \partial L$. This is clear when $x\in L$. The case when $x\in \partial L$ follows from that the labels of these edges are either vertices in $C$ or vertices in $\partial C$. From the claim we know $L\cup \partial L$ is a locally convex subcomplex of $X(\Gamma)$, in particular it is a full subcomplex, hence (2) follows as locally convex subcomplexes of CAT(0) cube complexes are convex.
\end{proof}

\begin{prop}
	\label{7.23}
Suppose $\Gamma$ is of weak type II. Pick vertex $v\in\mathcal{P}(\Gamma)$ and let $\bar{v}=\pi(v)$. Let $L$ be a component of $X(\Gamma)\setminus P_{v}$.  Let $q: X(\Gamma)\to X(\Gamma')$ be a quasi-isometry. Then the following holds true.
	\begin{enumerate}
		\item There is a 1-1 correspondence between $v$-branches and components of $X(\Gamma)\setminus P_{v}$. In particular, there exists a unique $v$-branch $B$ such that $\Delta(\bar{L})=\bar{B}$ and $\Delta(\partial L)=\partial B$. 
		\item There is a component $L'$ of $X(\Gamma')\setminus P_{q_{\ast}(v)}$ such that $$d_{H}(q(L),L')<\infty.$$
		\item For any component $C$ of $\Gamma\setminus St(\bar{v})$, $\partial C$ is a stable subgraph of $\Gamma$.
	\end{enumerate}
\end{prop}

\begin{proof}
We first prove Assertion (1). For $i=1,2$, let $e_{i}\subset X(\Gamma)$ be an edge such that one of its endpoint $x_{i1}\in X(\Gamma)\setminus P_{v}$ and another endpoint $x_{i2}\in P_{v}$. Let $\bar{v}_{i}\in \Gamma$ be the label of $e_{i}$, and let $C_{i}$ be the component of $\Gamma\setminus St(\bar{v})$ that contains $\bar{v}_{i}$. We claim $x_{11}$ and $x_{21}$ are in the same component of $X(\Gamma)\setminus P_{v}$ if and only if $C_{1}=C_{2}$ and $x_{21}$ and $x_{22}$ belong to the same $v$-peripheral subcomplex of $\partial C_{1}$. Then we have a 1-1 correspondence between components of $X(\Gamma)\setminus P_{v}$ and the pair $(C,K)$ as in Corollary \ref{7.14} (1) and the first part of Assertion (1) follows. The only if part of the claim follows from Lemma \ref{7.10} and Lemma \ref{7.11}. Note that $C_{1}$ contains more than one point (otherwise $\Ga$ will not be of weak type II), so the if direction holds in the special case when $\bar{v}_{1}=\bar{v}_{2}$, $x_{21}=x_{22}$ and $x_{11}\neq x_{12}$. The general case follows from the argument in the proof of Corollary \ref{7.14} (1).
	
	Suppose $(C,K)$ is the pair as above corresponding to $L$. Then the above claim implies $\partial L=K$. Let $B$ be the $v$-branch corresponding to $(C,K)$. Then $\partial B=\Delta(K)=\Delta(\partial L)$. Now we prove $\Delta(\bar{L})\subset\bar{B}$. Let $\ell\subset\bar{L}$ be a standard geodesic. If $d(\Delta(\ell),v)\ge 2$, by Lemma \ref{7.4}, there exists standard geodesic $\ell_{1}$ such that $\Delta(\ell_{1})$ and $\Delta(\ell)$ are in the same $v$-branch and $\ell_{1}\cap P_{v}\neq \emptyset$. The argument in Lemma \ref{7.9} implies that there exists an edge path $\omega\subset X(\Gamma)\setminus P_{v}$ connecting a vertex in $\ell$ and a vertex in $\ell_{1}$, thus $$\ell_{1}\subset \bar{L}.$$ It follows that $\ell_{1}\cap K\neq\emptyset$ and $\pi(\Delta(\ell_{1}))\subset C$, so $$\Delta(\ell_{1})\in B$$ by Corollary \ref{7.14} (1). Hence $$\Delta(\ell)\in B.$$ If $d(\Delta(\ell),v)= 1$, since $\bar{L}\cap P_{v}=K$, we apply Lemma \ref{2.6} (4) with $C_1=\bar{L}$ and $C_2=P_v$ to deduce that $\ell$ stays in the $R$-neighbourhood of $K$ for some $R>0$, thus $$\Delta(\ell)\in \Delta(K)=\partial B.$$ Note that $\Delta(\ell)\in\bar{B}$ in both cases, so $\Delta(\bar{L})\subset\bar{B}$. Now we prove $\bar{B}\subset\Delta(\bar{L})$. Pick vertex $w\in\bar{B}$. If $w\in\partial B$, then we are done by $\partial B=\Delta(\partial L)\subset\Delta(\bar{L})$. Suppose $w\in B$. Pick an edge $e\subset X(\Ga)$ which connects a point in $L$ and a point outside $L$ and let $\ell_e$ be the standard geodesic containing $e$. Then $\ell_e\subset \bar{L}$ by the discussion in the previous paragraph. Then $$\Delta(\ell_e)\in \bar{L}\subset \bar{B}.$$ However, $\Delta(\ell_e)\notin \partial B$, hence $$\Delta(\ell_e)\in B.$$ The argument in Lemma \ref{7.9} implies that there exists an edge path outside $P_{v}$ connecting a vertex in $\ell_e$ and a vertex in $P_w$. Thus $w\in\Delta(\bar{L})$. In summary, each vertex of $\bar{B}$ is in $\Delta(\bar{L})$. Since $\bar{L}$ is convex, $\Delta(\bar{L})$ is a full subcomplex by Lemma \ref{full subcomplex}, then $\bar{B}\subset\Delta(\bar{L})$.
	
	To see (2), let $\{\Delta_{\lambda}\}_{\lambda\in\Lambda}$ be the collection of maximal simplexes in $\mathcal{P}(\Gamma)$ such that $\Delta_{\lambda}\cap B\neq\emptyset$ and let $\{F_{\lambda}\}_{\lambda\in\Lambda}$ be the collection of maximal standard flats such that $\Delta(F_{\lambda})=\Delta_{\lambda}$. We claim $$d_{H}(L,\cup_{\lambda\in\Lambda}F_{\lambda})<\infty$$ where $d_H$ denotes the Hausdorff distance. Note that $\Delta_{\lambda}\subset\bar{B}$, hence $F_{\lambda}\subset\bar{L}$ by Assertion (1) and the maximality of $F_{\lambda}$. Pick an arbitrary vertex $x\in L$ and let $\ell_{x}$ be a standard geodesic such that $d(\pi(\Delta(\ell_{x})),\bar{v})\ge 2$ and $x\in \ell_{x}$. Then $$d(\Delta(\ell_{x}),v)\ge 2.$$ Hence $\ell_x\cap P_v$ is at most one point. It follows from the proof of Assertion (1) that $\ell_{x}\subset \bar{L}$ and $\Delta(\ell_{x})\subset B$. Thus there exists $\lambda_{0}\in\Lambda$ such that $$x\in \ell_x\subset F_{\lambda_{0}}.$$ Thus $L$ is contained in some neighborhood of $\cup_{\lambda\in\Lambda}F_{\lambda}$. However, $$d_H(L,\bar{L})<\infty$$ by Lemma~\ref{lem:closure}, hence the claim follows. Let $B'=q_{\ast}(B)$ and $L'$ be the component of $X(\Gamma')\setminus P_{q_{\ast}(v)}$ corresponding to $B'$ (note that $\Gamma'$ is also of weak type II by Corollary \ref{7.22}). By Lemma \ref{4.7}, for each $\lambda\in\Lambda$, there exists a unique maximal standard flat $F'_{\lambda}\subset X(\Gamma')$ such that $$d_{H}(q(F_{\lambda}),F'_{\lambda})<C$$ ($C$ is independent of $\lambda$). Note that $\{\Delta(F'_\lambda)\}_{\lambda\in\Lambda}$ is the collection of maximal simplexes of $\mathcal{P}(\Gamma)$ which have non-empty intersection with $B'$. We argue as before to deduce  $$d_{H}(L',\cup_{\lambda\in\Lambda}F'_{\lambda})<\infty.$$ Then $$d_{H}(q(L),L')<\infty.$$

	Now we prove (3). By Lemma~\ref{lem:unique},  $d_{H}(P_{v},P_{q_{\ast}(v)})<\infty$. Let $$K'=P_{q_{\ast}(v)}\cap \bar{L'}.$$ Then $K'$ is a $q_{\ast}(v)$-peripheral subcomplex by Lemma~\ref{lem:closure}, hence is a standard subcomplex. Recall that $K=P_{v}\cap L$, so $d_{H}(q(K),K')<\infty$ by Lemma \ref{2.6} (4).
\end{proof}

\section{Rigidity and flexibility of RAAG of weak type I}
\subsection{Motivating discussion and overview}
The goal of this section is to understand RAAGs that are quasi-isometric to a given RAAG of weak type I. We will start with a discussion of motivating examples.

We start with the case when $\Gamma$ is a pentagon. Let $G(\Gamma')$ be a RAAG quasi-isometric to $G(\Gamma)$. This gives a simplicial isomorphism $\mathcal{P}(\Gamma)\to \mathcal{P}(\Gamma')$ (cf. Theorem~\ref{4.8}). We can use this to conjugate the natural action of $G(\Gamma')$ on $\mathcal{P}(\Gamma')$ to another action $G(\Gamma')\curvearrowright \mathcal{P}(\Gamma)$. However, each automorphism of $\mathcal{P}(\Gamma)$ is visible, as for each vertex $v\in \mathcal{P}(\Gamma)$, each $v$-tier only contains one $v$-branches. This follows from Corollary~\ref{7.14} and Lemma~\ref{7.1}. Thus each automorphism of $\mathcal{P}(\Gamma)$ gives a bijection of $G(\Gamma)$ that preserves maximal standard flats (recall that we identify $G(\Gamma)$ as the vertex set of $X(\Gamma)$, and we will refer maximal standard flats in $G(\Gamma)$ as the intersections of maximal standard flats in $X(\Gamma)$ with $G(\Gamma)$). As a special feature of the pentagon graph, we know each standard flats in $G(\Gamma)$ is an intersection of maximal standard flats. Thus each automorphism of $\mathcal{P}(\Gamma)$ gives a bijection of $G(\Gamma)$ that sends standard flats to standard flats. We refer to this kind of bijections as \emph{flat-preserving bijections}. The action $G(\Gamma')\curvearrowright \mathcal{P}(\Gamma)$ gives an action $G(\Gamma')\curvearrowright G(\Gamma)$ by flat-preserving bijections. If each flat-preserving bijection of $G(\Gamma)$ were left translations of $G(\Gamma)$, then we can conclude immediately that $G(\Gamma')$ is isomorphic to a finite index subgroup of $G(\Gamma)$. However, this is not the case in general. 

Recall that each standard line in $G(\Gamma)$ is a left coset of a standard $\mathbb Z$-subgroup. Thus each standard line is labeled by a generator of $G(\Gamma)$, and inherits an order from $\mathbb Z$. A flat-preserving bijection of $G(\Gamma)$ is a left translation if and only if it respects the order on each standard line, and respects the labels of standard lines. Thus it is too much to hope that the action $G(\Gamma')\curvearrowright G(\Gamma)$ is by left translations of $G(\Gamma)$. However, if we are able to find a different labeling and ordering of the standard lines of $G(\Gamma)$ such that both of them are invariant $G(\Gamma')\curvearrowright G(\Gamma)$, then the action of $G(\Gamma')\curvearrowright G(\Gamma)$ is conjugate to an action by left translations, which will imply $G(\Gamma')$ is a finite index subgroup of $G(\Gamma)$. The new labeling and ordering need to satisfy some natural consistency conditions for this to work, and
the conjugation is via a flat-preserving bijection which connects the new labeling and ordering of standard lines to the old ones. 

The way to produce $G(\Gamma')$-invariant labeling and ordering of standard lines of $G(\Gamma)$ is as follows. As the map $\mathcal{P}(\Gamma)\to \mathcal{P}(\Gamma')$ coming from the quasi-isometry is visible, it gives a map $f:G(\Gamma) \to G(\Gamma')$ which is $G(\Gamma')$-invariant, where the action $G(\Gamma')\curvearrowright G(\Gamma')$ is by left translation. Very roughly speaking, we want to pull back the usual labeling and ordering of standard lines in $G(\Gamma')$ to obtain a $G(\Gamma')$-invariant labeling and ordering of standard lines of $G(\Gamma)$. However, this needs some work as $f$ is not an injective map, so pulling back does not make sense immediately. This gives a rough summary of the strategy in \cite{raagqi1} for handling the pentagon case.

Now we consider the general case of RAAGs of weak type I. The example to have in mind is $\Gamma$ being a pentagon and a triangle glued along an edge. Any quasi-isometry $q:G(\Gamma)\to G(\Gamma')$ still induces a visible simplicial isomorphism $q_*:\mathcal{P}(\Gamma)\to\mathcal{P}(\Gamma')$. Thus as before we have an action $\rho:G(\Gamma')\curvearrowright G(\Gamma)$. The key difference from the previous case is that the action is no longer by flat-preserving bijections. For an element $g'\in G(\Gamma')$, we still know $\rho(g')$ send maximal standard flats to maximal standard flats. But standard flats which are not maximal might not be preserved, as they are not necessarily intersections of maximal standard flats. For example, let $\bar v$ be the vertex in $\Gamma$ that is not insider the pentagon. Then $\rho(g')$ could send a standard line labeled by $\bar v$ to something which is not a standard line. More precisely, suppose the vertices of the triangle $\Gamma_1$ in $\Gamma$ are $\{\bar v,\bar v_1,\bar v_2\}$. Then $\rho(g')$ sends a standard flat $F$ of type $\Gamma_1$ to another 3-dimensional standard flat (as this standard flat is maximal). Moreover, as standard lines of type $\bar v_1$ and $\bar v_2$ are intersections of maximal standard flats, they are sent to standard lines by $\rho(g')$. As a consequence, those 2-dimensional standard flats of type $\{\bar v_1,\bar v_2\}$ are also sent to standard flats by $\rho(g')$. We want to think $F$ as being foliated by these 2-dimensional standard flats. The map $\rho(g')$ sends the leaves to parallel standard flats in another 3-dimensional standard flats, however, the $\rho(g')$-images of standard lines of type $\bar v$ in $F$ (which are transverse to all the leaves) could be rather arbitrary. So we can think $\rho(g')|_{F}$ as a leave-preserving shearing map.

For more general $\Gamma$ of weak type I, we will typically see that $\rho(g')$ preserves some standard flats, but there could also be different types of ``shearing'' of a family of lower dimensional standard flats inside a bigger standard flats. Such shearing could happen whenever an adjacent transvection is possible (e.g. in the above example, there is an adjacent transvection sending $\bar v$ to $\bar v_1$ or $\bar v_2$), though the behavior of $\rho(g')$ is in generally more complicated than adjacent transvections. Moreover, in general RAAGs of weak type I could allow many adjacent transvections.

The strategy in the case of pentagon no longer works, due to the failure of preservation of standard lines. In the case of weak type I, we will use an atlas system which encodes how various shearing is happening.  More precisely, an atlas on $G(\Gamma)$ is a collection of bijections between stable (cf. Definition~\ref{3.38}) standard flats in $G(\Gamma)$ and $\mathbb Z^n$ with suitable $n$, such that these bijections satisfies some natural consistency condition (see Definition~\ref{L-atlas}). This generalizes the pentagon case, as each standard line is stable in this case and we will have bijections between standard lines and $\mathbb Z$s, which correspond to the order on standard lines as discussed before. 

Now the main point is to construct an atlas on $G(\Gamma)$ which is invariant under the action of $G(\Gamma')$ (cf. Proposition~\ref{prop:atlas}). Once this is established, we can conclude that the action of $G(\Gamma')$ on $G(\Gamma)$ is conjugated to an action by left translation, which implies that $G(\Gamma')$ is a finite index subgroup of $G(\Gamma)$.

\subsection{An atlas for RAAG}

Let $G(\Gamma)$ be a RAAG of weak type I with trivial center. We identify $G(\Ga)$ with the 1-skeleton of $X(\Ga)$ and define a \textit{standard flat} in $G(\Ga)$ to be the vertex set of some standard flat in $X(\Ga)$.

Theorem \ref{weak type I visible} implies there is a homomorphism $s:Aut(P(\Gamma))\to Perm(G(\Gamma))$, where $Perm(G(\Gamma))$ is the permutation group of elements in $G(\Ga)$. Note that images of $s$ preserve maximal standard flats. However, this may not be true for all standard flats, since adjacent transvections are allowed in $\out(G(\Ga))$.

Let $\mathcal{P}(\Ga)$ be the extension complex and let $\pi:\P(\Ga)\to F(\Ga)$ be label-preserving simplicial map defined in Section \ref{basics about raag}. Note that for any vertex $x\in X(\Ga)$, $\pi$ induces an isomorphism $(F(\Ga))_x\to F(\Ga)$. This motivates the following definition.

\begin{definition}[Coherent labelling]
	\label{def:coherent labeling}
A \textit{coherent labelling} of $G(\Ga)$ is a simplicial map $L:\mathcal{P}(\Ga)\to F(\Ga)$ such that $L$ induces an isomorphism $(F(\Ga))_x\to F(\Ga)$ for every vertex $x\in X(\Gamma)$.
\end{definition} 

Assume $n=\dim(X(\Gamma))$. Let $\mathcal{F}(\Gamma)$ be the collection of stable standard flat in $X(\Gamma)$ and let $\mathcal{F}_{k}(\Gamma)$ be the collection of $k$-flats in $\mathcal{F}(\Gamma)$. Define $\mathcal{F}_{< k}(\Gamma):=\cup_{i=1}^{k-1}\mathcal{F}_{i}(\Gamma)$. Here we are considering the set itself, not the coarse equivalent classes of the sets (compared to Theorem \ref{4.8}). Recall that we use $v(K)$ to denote the set of vertices in a subset $K$ of some polyhedral complex. 

\begin{definition}[$L$-atlas]
	\label{L-atlas}
	An \textit{$L$-atlas} is a coherent labelling $L:\mathcal{P}(\Gamma)\to F(\Gamma)$ together with a collection of bijections 
	\begin{equation*}
	\{v(F)\to \Bbb Z^{v(L(\Delta(F)))}\}_{F\in\mathcal{F}_{k}(\Gamma),1\le k\le n}
	\end{equation*}
	with the following compatibility condition: pick $F_{1}\in\mathcal{F}_{m}(\Gamma)$ and $F_{2}\in\mathcal{F}_{\ell}(\Gamma)$ with $F_{1}\subset F_{2}$, let $f:v(F_{2})\to \Bbb Z^{v(L(\Delta(F_{2})))}$ and $g:v(F_{1})\to \Bbb Z^{v(L(\Delta(F_{1})))}$ be the associated bijections. Suppose $p:\Bbb Z^{v(L(\Delta(F_{2})))}\to \Bbb Z^{v(L(\Delta(F_{1})))}$ is the natural projection. Then
	\begin{enumerate}
		\item $f(v(F_{1}))$ is a coset of $\Bbb Z^{v(L(\Delta(F_{1})))}$ in $\Bbb Z^{v(L(\Delta(F_{2})))}$.
		\item The following diagram commutes up to translation:
		\begin{center}
			$\begin{CD}
			v(F_{1})                         @>g>>        \Bbb Z^{v(L(\Delta(F_{1})))}\\
			@VViV                                                              @AApA\\
			v(F_{2})                  @>f>>       \Bbb Z^{v(L(\Delta(F_{2})))}
			\end{CD}$
		\end{center}
		Here $i$ is the inclusion map.
	\end{enumerate}
\end{definition}

\begin{definition}[Equivalence and pull back]
We say $L$-atlas $\mathcal{A}_{L}$ and $L'$-atlas $\mathcal{A}_{L'}$ are \textit{equal up to translations} if $L=L'$ and the bijections in $\mathcal{A}_{L}$ and $\mathcal{A}_{L'}$ agree up to translation. We will write $\mathcal{A}_{L}\overset{e}{=}\mathcal{A}_{L'}$ in this case. Pick $\alpha\in Aut(\P(\Ga))$ and let $\alpha_{\ast}:G(\Ga)\to G(\Ga)$ be the bijection induced by $\alpha$ (cf. Section~\ref{sec:visible} and Theorem~\ref{7.28}). Recall that $\alpha_{\ast}$ preserves stable standard flats. The \textit{pull-back} of an $L$-atlas $\A_L$ under $\alpha$, denoted by $\alpha^{\ast}(\A_L)$, is defined to be the $(L\circ \alpha)$-atlas with its charts being the pull-backs of charts of $\A_L$ under $\alpha_{\ast}$. More precisely, charts of $\alpha^{\ast}(\A_L)$ are compositions:
\begin{equation*}
\{v(F)\stackrel{\al_*}{\to} \al_*(v(F)) \to \Bbb Z^{v(L(\Delta(\al_*(F))))}=\Bbb Z^{v(L\circ \alpha(\Delta(F)))}\}_{F\in\mathcal{F}_{k}(\Gamma),1\le k\le n}
\end{equation*}
Note that $L(\Delta(\al_*(F)))$ and $L\circ \alpha(\Delta(F))$ are the same subset of $F(\Gamma)$.
\end{definition}

\begin{remark}
Note that the construction of $\al_*$ from $\al$ in Section~\ref{sec:visible} only use the information of what are $\al$-images of maximal simplexes in $\mathcal{P}(\Gamma)$. Thus a priori it could happen that two different elements $\alpha_1$ and $\alpha_2$ in $Aut(\mathcal{P}(\Gamma))$ give the same $\al_*:G(\Gamma)\to G(\Gamma)$. It is natural ask how different is $\alpha^{\ast}_1(\A_L)$ from $\alpha^{\ast}_2(\A_L)$.

We now clarify this point via the following example. Suppose $\Gamma$ is obtained by gluing a pentagon and the 1-skeleton of a 3-simplex along an edge. Let $a$ and $b$ be the two vertices of $\Gamma$ which are outside the pentagon. Let $e$ be an edge in $\mathcal{P}(\Gamma)$ which maps to the edge $\overline{ab}\subset \Gamma$ under the map $\mathcal{P}(\Gamma)\to F(\Gamma)$. We take $\al_1\in Aut(\mathcal{P}(\Gamma))$ be the identity map. Take $\al_2\in Aut(\mathcal{P}(\Gamma))$ be the automorphism which exchanges the two end points of $e$ and fixes all other vertices of $\mathcal{P}(\Gamma)$ pointwise. Then $(\al_1)_*=(\al_2)_*$ is the identity map and $\al_1(\Delta(F))=\al_2(\Delta(F))=\Delta(F)$ for any $F\in \mathcal{F}(\Gamma)$. Thus all charts in $\alpha^{\ast}_1(\A_L)$ and $\alpha^{\ast}_2(\A_L)$ are exactly the same. However, $L\circ \al_1$ and $L\circ \al_2$ are different maps. So $\alpha^{\ast}_1(\A_L)\stackrel{e}{\neq}\alpha^{\ast}_2(\A_L)$.
\end{remark}

Recall that we label each circle in $S(\Ga)$ by a generator of $G(\Ga)$ and fix an orientation for each circle. This lifts to $G(\Ga)$-invariant labelling and orientation of edges in $X(\Ga)$. Moreover, we have induced action $G(\Ga)\acts\P(\Ga)$ and induced $G(\Ga)$-invariant labelling of vertices in $\P(\Ga)$. This leads to a naturally defined $L$-atlas as follows. 

\begin{definition}[Reference atlas]
	\label{def:reference atlas}
Let $L$ be the label preserving map $\pi:\mathcal{P}(\Gamma)\to F(\Gamma)$. For each vertex $u\in\P(\Ga)$, we pick a standard geodesic $\ell\subset X(\Ga)$ such that $\Delta(\ell)=u$, and identify vertices of $\ell$ with $\Z^{u}$ in an orientation-preserving way. Let $p_u:G(\Ga)\to \Z^u$ be the map induced by the $CAT(0)$ projection from $G(\Ga)$ to $\ell$ (recall that we have identified $G(\Ga)$ with vertices of $X(\Ga)$, and Lemma \ref{2.6} implies that the image of each vertex of $X(\Ga)$ under the $CAT(0)$ projection is a vertex in $\ell$). For each standard flat $F\subset X(\Ga)$, $p_u(v(F))$ is surjective if $u\in\Delta(F)$, otherwise $p_u(v(F))$ is a point. This induces a bijection $\prod_{u\in\Delta(F)}p_u:v(F)\to \Z^{v(\Delta(F))}$ and we define the chart for $F$ to be $\prod_{u\in\Delta(F)}p_u$ post-composed with $\Z^{v(\Delta(F))}\to \Z^{v(L(\Delta(F)))}$. One readily verifies this atlas $\A_L$ satisfies the above definition of $L$-atlas, moreover, the diagram in (2) commutes exactly, not up to translations. The following properties are immediate.
\begin{enumerate}
	\item $\A_L$ is $G(\Ga)$-invariant up to translations in the sense that $g^{\ast}(\A_L)\overset{e}{=}\A_L$ for all $g\in G(\Ga)$. Conversely, if $\alpha\in Aut(\P(\Ga))$ satisfies $\alpha^{\ast}(\A_L)\overset{e}{=}\A_L$, then the induces map $\alpha_{\ast}:G(\Ga)\to G(\Ga)$ is a left translation.
	\item $\A_L$ is unique up to translations. Since the only ambiguity in the definition of $\A_L$ is the orientation-preserving identification of $v(\ell)$ with $\Z^u$, which is unique up to translations.
\end{enumerate}
The atlas $\A_L$ is called the \textit{reference} atlas.
\end{definition}

\begin{lem}
\label{6.1}
Let $G(\Gamma)$ be of weak type I and pick $F\in \mathcal{F}(\Gamma)$. Then there exist standard flats $\{F_{i}\}_{i=1}^{k}$ in $F$ such that $F$ is the convex hull of these flats and each $F_{i}$ is the intersection of maximal standard flats.
\end{lem}

\begin{proof}
Pick vertex $w\in\Gamma$. Let $\Gamma_{w}$ be the minimal stable subgraph containing $w$ and Let $\Gamma'_{w}$ be the intersection of maximal cliques that contains $w$. It suffices to show $\Gamma_{w}=\Gamma'_{w}$. Since each maximal clique is stable (Lemma \ref{7.17} and Lemma \ref{4.7}), $\Gamma_{w}\subset\Gamma'_{w}$. Pick vertex $v\in\Gamma'_{w}$, then $w^{\perp}\subset St(v)$, thus $v\in\Gamma_{w}$ by \cite[Lemma 3.32]{raagqi1}. It follows that $\Gamma'_{w}\subset\Gamma_{w}$.
\end{proof}

Suppose $G(\Gamma)$ has weak type I and it has trivial center. Let $q:G(\Gamma)\to G(\Gamma')$ be a quasi-isometry and let $s:\mathcal{P}(\Gamma)\to\mathcal{P}(\Gamma')$ be an induced simplicial isomorphism (cf. Definition~\ref{7.24} and Lemma \ref{7.18}). By Theorem \ref{weak type I visible} and Section~\ref{sec:visible}, $s$ induces a map $\phi:G(\Gamma)\to G(\Gamma')$. 

\begin{lem}
If we assume in additional that $G(\Gamma)$ There exists $D_0>0$ such that $d(q(x),\phi(x))<D_0$ for any $x\in G(\Gamma)$. Thus $\phi$ is a quasi-isometry.
\end{lem}

\begin{proof}
By Lemma~\ref{4.7}, Definition~\ref{7.24} and Lemma~\ref{7.17}, each maximal clique subgraph of $\Gamma$ is stable. Thus there exists $D_1>0$ such that for any maximal standard flat $F\subset X(\Gamma)$, there exists a maximal standard flat $F'\subset X(\Gamma')$ such that $d_H(q(F),F')<D_1$. Let $\{F_i\}_{i=1}^n$ and $\{F'_i\}_{i=1}^n$ be as in Section~\ref{sec:visible}. Then $d(q(\cap_{i=1}^n F_i),\cap_{i=1}^n F'_i)$ is uniformly upper bounded, which implies the lemma.
\end{proof}

For every $g'\in G(\Gamma')$, there is left translation $\bar{\phi}_{g'}:G(\Gamma')\to G(\Gamma')$, which gives rise to a simplicial isomorphism $\bar{s}_{g'}:\mathcal{P}(\Gamma')\to\mathcal{P}(\Gamma')$. Let $$s_{g'}=s^{-1}\circ \bar{s}_{g'}\circ s.$$ Then $s_{g'}$ induces a unique bijection $\phi_{g'}:G(\Gamma)\to G(\Gamma)$ by Theorem \ref{weak type I visible}, moreover,
\begin{equation}
\label{5.8}
\bar{\phi}_{g'}\circ\phi=\phi\circ\phi_{g'}
\end{equation}
In summary, we have $G(\Gamma')$ acts on $G(\Gamma')$, $\mathcal{P}(\Gamma')$, $G(\Gamma)$ and $\mathcal{P}(\Gamma)$. 

\begin{lem}\
\label{6.2}
\begin{enumerate}
\item $\phi$ is surjective. For any $y,y'\in G(\Gamma')$, $|\phi^{-1}(y)|=|\phi^{-1}(y')|<\infty$.
\item For any $k$ and $F\in \mathcal{F}_{k}(\Gamma)$, there exists unique $F'\in \mathcal{F}_{k}(\Gamma')$ such that $\phi(v(F))=v(F')$. Moreover, let $\Stab(v(F'))$ and $\Stab(v(F))$ be the stabilizer of $v(F')$ and $v(F)$ with respect to the action $G(\Ga')\acts G(\Ga')$ and $G(\Ga')\acts G(\Ga)$ respectively. Then $\Stab(v(F'))=\Stab(v(F))$. In this case we will write $F'=\phi(F)$ for simplicity.
\item Let $F_{1},F_{2}\in \mathcal{F}_{k}(\Gamma)$ be parallel standard flats. Then for vertices $x_{1}\in F_{1}$ and $x_{2}\in F_{2}$, $$|\phi^{-1}(\phi(x_{1}))\cap F_{1}|=|\phi^{-1}(\phi(x_{2}))\cap F_{2}|.$$
\end{enumerate}
\end{lem}

\begin{proof}
Pick a reference point $q\in \textmd{Im}\ \phi$ and let $K_{q}=(F(\Gamma'))_{q}$. Denote the points in $\phi^{-1}(q)$ by $\{p_{\lambda}\}_{\lambda\in\Lambda}$ and let $K_{p_{\lambda}}=(F(\Gamma))_{p_{\lambda}}$. Since $\{\phi(K_{p_{\lambda}})\}_{\lambda\in\Lambda}$ are distinct subcomplexes of $K_{q}$, $\Lambda$ is a finite set. The other parts of (1) follows from (\ref{5.8}).

Now we prove (2). It is clear if $F$ is a maximal standard flat. Next we look at the case when $F=\cap_{i=1}^{h}F_{i}$ where each $F_{i}$ is a maximal standard flat. Let $F'_{i}$ be maximal standard flat in $X(\Gamma')$ such that $\Delta(F'_{i})=s(\Delta(F_{i}))$ for $1\le i\le h$ and let $F'=\cap_{i=1}^{h}F'_{i}$. Then $$\phi(v(F))\subset v(F').$$ Note that $\Stab(v(F'))$ (resp. $\Stab(v(F'_i))$) is a conjugate of a standard subgroup of type $\Gamma'_{F'}$ (resp. $\Gamma'_{F'_i}$). Thus $$\Stab(v(F'))\subset \Stab(v(F_i)).$$ Then
the stabilizer $\Stab(v(F'))$ fixes $\Delta(F'_{i})$ for all $i$, hence it fixes $\Delta_{i}$ for all $i$ and $\Stab(v(F'))\subset \Stab(v(F))$. Since $\Stab(v(F'))$ acts on $v(F')$ transitively, (\ref{5.8}) implies $$\phi(v(F))=v(F')$$ and $$\Stab(v(F))\subset \Stab(v(F')).$$ Thus $\Stab(v(F'))= \Stab(v(F))$.

In general, by Lemma \ref{6.1}, we can assume $F$ is the convex hull of $F_{1},F_{2}\in\mathcal{F}(\Gamma)$ such that (2) is true for flats in $F$ which are parallel to $F_{1}$ or $F_{2}$. Let $F'_{i}=\phi(F_{i})$ for $i=1,2$. Then $F'_{1}\cap F'_{2}\neq\emptyset$ and the convex hull of $F'_{1}$ and $F'_{2}$ is a flat $F'$ (since $F$ is contained in a maximal standard flat, whose image under $\phi$ is a maximal standard flat containing $F'_1$ and $F'_2$). It follows from Lemma \ref{3.50} that $F'\in\mathcal{F}(\Gamma')$. Note that any standard flat that is parallel to $F_{1}$ and intersects $F_{2}$ is mapped by $\phi$ to a standard flat that is parallel to $F'_{1}$ and intersects $F'_{2}$, thus $\phi(v(F))\subset v(F')$. Let $F_{3}\subset F$ be a standard flat parallel to $F_{1}$ and let $F'_{3}=\phi(F_{3})$. Since parallel standard flats in $X(\Gamma')$ have the same stabilizer, we have $$\Stab(v(F_{1}))=\Stab(v(F'_{1}))=\Stab(v(F'_{3}))=\Stab(v(F_{3})).$$ By considering all such $F_{3}$'s in $F$, we have $$\Stab(v(F_{1}))\subset \Stab(v(F)).$$ Similarly, $\Stab(v(F_{2}))\subset \Stab(v(F))$, thus 
\begin{align*}
\Stab(v(F'))&= \langle \Stab(v(F'_{1})),\Stab(v(F'_{2}))\rangle\\
&=\langle \Stab(v(F_{1})), \Stab(v(F_{2}))\rangle\subset \Stab(v(F))
\end{align*}
Now we can conclude $\phi(v(F))=v(F')$ as before. It also follows that $\Stab(v(F))\subset \Stab(v(F'))$, thus 
\begin{equation}
\label{6.3}
\Stab(v(F'))= \Stab(v(F)).
\end{equation}

Now we prove (3). Note that for a pair of parallel standard flats $F'_{1}$ and $F'_{2}$ in $X(\Gamma')$, there exists $g'\in G(\Gamma')$ such that $g'(v(F'_{1}))=v(F'_{2})$, so by (\ref{5.8}), it suffices to prove (3) in the case where $$\phi(v(F_{1}))=\phi(v(F_{2}))=v(F').$$ Note that $F_1$ and $F_2$ are parallel and there is an isometry $p:F_1\to F_2$ induced by the nearest point projection in the ambient CAT(0) cube complex. The map $p$ sends vertices to vertices, hence restrict to a bijection $p: v(F_{1})\to v(F_{2})$, which we call the \emph{parallelism map} between $v(F_1)$ and $v(F_2)$. 
Denote $$p_{1}=\phi|_{v(F_{1})}:v(F_{1})\to v(F')$$ and $$p_{2}=\phi|_{v(F_{2})}\circ p: v(F_{1})\to v(F').$$ Then there exist $L$ and $A$ such that $p_{1}$ and $p_{2}$ are $(L,A)$-quasi-isometries and
\begin{equation}
\label{6.4}
d(p_{1}(x),p_{2}(x))<A
\end{equation}
for any $x\in v(F_{1})$. Pick $y\in v(F')$ and let $r_{i}$ be the number of points $|p^{-1}_{i}(y)|$ in $p^{-1}_{i}(y)$ for $i=1,2$ ($r_{i}$ does not depend on $y$ by previous discussion). We argue by contradiction and assume $r_{1}<r_{2}$. Pick base point $x_{0}\in v(F_{1})$, let $m=\dim(F_{1})$, $$B_{R}=B(x_{0},R)$$ and $$K_{i,R}=p_{i}(B_{R})$$ for $i=1,2$. Then it follows from (\ref{6.4}) that  
\begin{align*}
&|K_{1,R}|\le |N_{A}(K_{2,R})|=|K_{2,R}|+|N_{A}(K_{2,R})\setminus K_{2,R}| \\
&\le|K_{2,R}|+ |p^{-1}_{2}(N_{A}(K_{2,R})\setminus K_{2,R})|\le |K_{2,R}|+ |B_{LA+A+R}\setminus B_{R}|\\
&\le |K_{2,R}|+ CR^{m-1}(LA+A),
\end{align*}
here $C$ is some constant independent of $R$. Thus by symmetry we have
\begin{equation}
\label{6.5}
||K_{1,R}|-|K_{2,R}||\le CR^{m-1}(LA+A).
\end{equation}
On the other hand, $B_{R}\subset p^{-1}_{i}({K_{i,R}})\subset B_{R+A}$ for $i=1,2$, thus
\begin{equation}
\label{6.6}
CR^{m}\le |p^{-1}_{i}({K_{i,R}})|=r_{i}|K_{i,R}|\le C(R+A)^{m}
\end{equation}
for $i=1,2$. Now (\ref{6.5}) and (\ref{6.6}) imply
\begin{align*}
&CR^{m}/r_{1}-C(R+A)^{m}/r_{2}\le |K_{1,R}|-|K_{2,R}|  \\
&\le||K_{1,R}|-|K_{2,R}||\le CR^{m-1}(LA+A).
\end{align*}
Since $r_{1}<r_{2}$, we will get a contradiction when $R\to\infty$.
\end{proof}

\begin{lem}
\label{6.7}
Suppose $G(\Gamma)$ has weak type I and trivial center.
Given $L$-atlas $\mathcal{A}_{L}$ and $L'$-atlas $\mathcal{A}_{L'}$, there exists $\alpha\in Aut(\mathcal{P}(\Gamma))$ such that $\alpha^{\ast}(\mathcal{A}_{L'})\overset{e}{=}\mathcal{A}_{L}$. 
\end{lem}

%Moreover, there is a 1-1 correspondence between elements in $Aut(\mathcal{P}(\Gamma))$ and the following set of information:
%\begin{enumerate}
%	\item A base point $p\in G(\Gamma)$.
%	\item A class of $L$-atlases which are equal up to translations.
%\end{enumerate}

The proof is variation of \cite[Lemma 5.7]{raagqi1}.
\begin{proof}
We prove the first part of the lemma. Pick $v\in G(\Gamma)$, set $\alpha'(e)=v$. For $u\in G(\Gamma)$, pick a word $w_{u}=a_{1}a_{2}\cdots a_{n}$ representing $u$, let $u_{i}=a_{1}a_{2}\cdots a_{i}$ for $1\le i\le n$ and $u_{0}=e$. We define $q_{i}=\alpha'(a_{1}a_{2}\cdots a_{i})\in G(\Gamma)$ inductively as follows: set $q_{0}=v$ and suppose $q_{i-1}=\alpha'(a_{1}a_{2}\cdots a_{i-1})$ is already defined. Let $F_{i}\in \mathcal{F}(\Gamma)$ be a standard flat containing $u_{i-1}$ and $u_{i}$ and let $F'_{i}$ be the unique standard flat such that $q_{i-1}\in F'_{i}$ and $$L'(\Delta(F'_{i}))=L(\Delta(F_{i})).$$ There is a natural identification of $g_i:F_{i}\to F'_{i}$ via the charts $$f:F_{i}\to \Bbb Z^{v(L(\Delta(F_{i})))}$$ and $$f':F'_{i}\to \Bbb Z^{v(L'(\Delta(F'_{i})))}=\Bbb Z^{v(L(\Delta(F_{i})))}$$ such that $g_i=(f')^{-1}\circ f$. Up to post-composing $f$ and $f'$ by translations, we can assume $f_{i}(u_{i-1})=q_{i-1}$. Then we define $q_{i}=f_{i}(u_{i})$. Note that the definition of  $q_{i}=\alpha'(a_{1}a_{2}\cdots a_{i})$ does not depend on the choice of $F_{i}$ by the compatibility condition (2).

We now claim for any other word $w'_{u}$ representing $u$, $\alpha'(w_{u})=\alpha'(w'_{u})$, hence we have a well-defined map $\alpha': G(\Gamma)\to G(\Gamma)$. To see this, recall that one can obtain $w_{u}$ from $w'_{u}$ by performing the following two basic moves:
\begin{enumerate}
\item $w_{1}aa^{-1}w_{2}\to w_{1}w_{2}$.
\item $w_{1}abw_{2}\to w_{1}baw_{2}$ when $a$ and $b$ commutes.
\end{enumerate}
It is clear that $\alpha'(w_{1}aa^{-1}w_{2})=\alpha'(w_{1}w_{2})$ and it suffices to show $\alpha'(ab)=\alpha'(ba)$ where $a$ and $b$ are mutually commuting generators. Let $F$ be a maximal standard flat that contains $e,a$ and $b$, we could always choose $F$ in the definition of $\alpha'(ab)$ or $\alpha'(ba)$, thus they are equal.

By switching the role of $\A_L$ and $\A_{L'}$, we can define $\alpha'':G(\Ga)\to G(\Ga)$ which maps $v$ to $e$ in a similar way.
It is not hard to check $\alpha'$ and $\alpha''$ are mutual inverse. Thus $\alpha'$ is bijective, moreover, $\alpha'$ preserves $\mathcal{F}(\Gamma)$. To define $\alpha$, pick vertex $w\in\mathcal{P}(\Gamma)$, let $\Delta$ be a maximal simplex containing $w$. Take $F\subset X(\Gamma)$ to be the flat such that $\Delta(F)=\Delta$ and take $F'$ to be the maximal standard flat such that $\alpha(v(F))=v(F')$, we set $\alpha(w)$ to be the unique point such that $\alpha(w)\in\Delta(F')$ and $L(w)=L'(\alpha(w))$. One readily verifies that $L=L'\circ\alpha$, $\alpha'$ is induced by $\alpha$ and $\alpha'$ pulls back the charts up to translations, so $\alpha^{\ast}(\mathcal{A}_{L'})\overset{e}{=}\mathcal{A}_{L}$. 
\end{proof}

\subsection{Shearing standard flats}
In this subsection we prove the following theorem.
\begin{thm}
	\label{6.12}
	Let $G(\Gamma)$ be a group of weak type I. Then the following are equivalent.
	\begin{enumerate}
		\item $G(\Gamma')$ is quasi-isometric to $G(\Gamma)$.
		\item $G(\Gamma')$ is isomorphic to a finite index subgroup of $G(\Gamma)$.
		\item $G(\Gamma')$ is isomorphic to a special subgroup (Section \ref{special subgroup}) of $G(\Gamma)$.
	\end{enumerate}
\end{thm}

\begin{remark}
From Theorem~\ref{6.12}, we know in particular that a finite index RAAG subgroup $H$ of $G(\Ga)$ is isomorphic to a special subgroup. However, $H$ might not be a special subgroup of $G(\Gamma)$. Interestingly, under the strong condition that $\out(G(\Ga))$ is finite, any finite index RAAG subgroup will automatically be a special subgroup \cite[Section 6]{raagqi1}.
\end{remark}

The following is a consequence of Theorem \ref{6.12} and \cite[Section 6.3]{raagqi1}.

\begin{cor}
	Let $G(\Gamma)$ be a group of weak type I. Then there is an algorithm to determine whether $G(\Gamma')$ and $G(\Gamma)$ are quasi-isometric.
\end{cor}

In the rest of this subsection, we prove Theorem~\ref{6.12}. Note that it suffices to prove the case when $G(\Gamma)$ has trivial center. Thus from now on, $G(\Gamma)$ is a RAAG of weak type I with trivial center.  Let $\mathcal{A}_{L'}$ be the reference atlas for $G(\Gamma')$. Let $q,s,s_{g'},\bar{s}_{g'},\phi$, $\phi_{g'}$ and $\bar{\phi}_{g'}$ be as in the discussion before Lemma \ref{6.2}. We will also be using actions of $G(\Ga')$ on $\P(\Ga'),G(\Ga'),\P(\Ga)$ and $G(\Ga)$ discussed over there.
A main ingredient of the proof is the following.

\begin{prop}
	\label{prop:atlas}
Under the above setting, there exists a coherent labelling $L:\mathcal {P}(\Gamma)\to F(\Gamma)$ for $G(\Gamma)$ which is invariant under the action $G(\Gamma')\curvearrowright \mathcal{P}(\Gamma)$ and an $L$-atlas $\mathcal{\bar{A}}_{L}$ for $G(\Gamma)$ such that
\begin{enumerate}
	\item $\mathcal{\bar{A}}_{L}\overset{e}{=}(\phi_{g'})^*(\mathcal{\bar{A}}_{L})$ for any $g'\in G(\Gamma')$. 
	\item (inverse images are boxes) Given $F\in \mathcal{F}(\Gamma)$, let $F'=\phi(F)$ and let $$s_{0}:v(L(\Delta(F)))\to v(L'(\Delta(F')))$$ be the bijection induced by $s$. Suppose $\bar{h}$ and $h'$ are charts for $F$ and $F'$ with respect to $\bar{\mathcal{A}}_L$ and $\mathcal{A}_{L'}$ respectively. Then $\varphi=h'\circ\phi\circ\bar{h}$ admits splitting $$\varphi=\prod_{w\in v(L(\Delta(F)))}\varphi_{w}$$ where $$\varphi_{w}:\Bbb Z^{w}\to \Bbb Z^{s_{0}(w)}$$ (here $\mathbb Z^w$ denotes the copy of $\mathbb Z$ associated with vertex $w$) is of form $$\varphi_{w}(a)=\lfloor a/d_{w}\rfloor+r_{w}$$ for some integers $r_{w}$ and $d_{w}$ ($d_{w}>0$).
\end{enumerate}
\end{prop}

%In the light of Lemma \ref{coherent labelling}, it suffices to show charts of $\bar{A}_{L}$ are $G(\Ga')$-invariant up to translations. The slightly more technical Proposition~\ref{prop:atlas} (2) roughly means under our choice of the charts, the map $\phi$ looks rather simple in the sense that when res in $\mathcal{\bar{A}}_{L}$ are ``pull-backs'' of the charts in $\mathcal{A}_{L'}$. 
Now we prove Theorem~\ref{6.12}, assuming Proposition~\ref{prop:atlas}.
\begin{proof}[Proof of Theorem~\ref{6.12}.]
	$(3)\Rightarrow (1)$ and $(2)\Rightarrow (1)$ are trivial. Now we look at $(1)\Rightarrow (2)$. By Theorem \ref{2.13}, we can assume $G(\Gamma)$ has trivial center. Let $\A_{ref}$ be the reference atlas for $G(\Ga)$ and let $\bar{\mathcal A}_L$ be as in Proposition~\ref{prop:atlas}. By Lemma \ref{6.7}, there exists simplicial isomorphism $r:\mathcal{P}(\Gamma)\to \mathcal{P}(\Gamma)$ such that $r^{\ast}(\bar{\mathcal{A}}_{L})\overset{e}{=}\mathcal{A}_{ref}$. The $G(\Gamma')$-invariance of $\bar{\mathcal{A}}_{L}$ implies $$(r^{-1}\circ s_{g'}\circ r)^{\ast}(\mathcal{A}_{ref})\overset{e}{=}\mathcal{A}_{ref}.$$ Let $\phi_{r}:G(\Gamma)\to G(\Gamma)$ be the map induced by $r$. Then by Definition~\ref{def:reference atlas} (1), $\phi_{r}^{-1}\circ\phi_{g'}\circ\phi_{r}$ is a left translation of $G(\Ga)$. Hence we have obtained a faithful action of $G(\Ga')$ on $G(\Ga)$ via left translations with finitely many orbits. Thus (2) follows.
	
	$(1)\Rightarrow (3)$. Since the atlas $\bar{A}_L$ satisfies Proposition~\ref{prop:atlas} (2), we deduce that for $F\in\mathcal{F}(\Gamma)$, there is $F'\in \mathcal{F}(\Gamma')$ such that the map $\phi\circ\phi_{r}|_{v(F)}:v(F)\to v(F')$ is surjective and it maps a collection of boxes in $v(F)$ to single points in $v(F')$. As $\Stab(F')$ acts transitively on $v(F')$, it follows from Proposition~\ref{prop:atlas} (1) that all these boxes have the same dimension. As each standard flat is contained a stable standard flat, we know from Proposition~\ref{prop:atlas} (2) that $\phi\circ\phi_{r}$ sends standard flats to standard flats. Moreover,
	\begin{itemize}
		\item if two elements of $G(\Gamma)$ are adjacent in $X(\Gamma)$, then their $\phi\circ\phi_{r}$-images are either the same or adjacent;
		\item given a pair of parallel edges $e_1,e_2\subset X(\Gamma)$ are parallel, if the $\phi\circ\phi_{r}( \partial e_1)$ is a single point ($\partial e_1$ denotes the collection of two end points of $e_1$), then the same holds for $\phi\circ\phi_{r}( \partial e_2)$; if the $\phi\circ\phi_{r}(\partial e_1)=\partial e'_1$ for an edge $e'_1\subset X(\Gamma')$, then  $\phi\circ\phi_{r}(\partial e_2)=\partial e'_2$ for an edge $e'_2\subset X(\Gamma')$ with $e'_2$ parallel to $e'_1$.
	\end{itemize}
Now it is clear that $\phi\circ\phi_{r}$ extends to a cubical map $X(\Ga)\to X(\Ga')$. The inverse image of a vertex under this cubical map is a compact subcomplex by Lemma~\ref{6.2} (1). Note that $\phi\circ\phi_{r}$ is $G(\Gamma')$-equivariant, where the $G(\Gamma')$ action on $X(\Gamma)$ is given by $g'\to \phi_{r}^{-1}\circ\phi_{g'}\circ\phi_{r}$. As $\phi$ gives a 1-1 correspondence between maximal standard flats in $X(\Gamma)$ and maximal standard flats in $X(\Gamma')$, and $\phi_r$ gives a bijection on the set of maximal standard flats in $X(\Gamma)$, we know $\phi\circ\phi_{r}$ induces an isomorphism of the associated extension complexes in the sense explained after Theorem~\ref{retraction map}.
Then Lemma~\ref{lem:recognizing special subgroup} implies $G(\Ga')$ is isomorphic to a special subgroup (note that Lemma~\ref{lem:recognizing special subgroup} (2) follows from Proposition~\ref{prop:atlas} (2)).
\end{proof}

In the rest of this subsection, we prove Proposition~\ref{prop:atlas}. We first arrange the coherent labelling $L$.

\begin{lem}
\label{coherent labelling}
There exists a coherence labelling $L:\P(\Ga)\to F(\Ga)$ for $G(\Ga)$ which is invariant under the action $G(\Ga')\acts \P(\Ga)$.
\end{lem}

This proof of this lemma is part of the proof of \cite[Lemma 5.9]{raagqi1}. We extract here for the convenience of the reader.

\begin{proof}
Recall that each vertex of $\mathcal{P}(\Gamma)$ is labeled by a vertex in $\Gamma$ (cf. Section~\ref{basics about raag}), which induces a simplicial map $L_0:\mathcal{P}(\Gamma)\to F(\Gamma)$. Similarly we define $L'_0:\mathcal{P}(\Gamma')\to F(\Gamma')$. Take $g'\in G(\Gamma)$ and let $i_{g'}:F(\Gamma')\to \mathcal{P}(\Gamma')$ be the embedding as in Lemma~\ref{isometric embedding}. Define
	\begin{equation*}
	L=L_0\circ s^{-1}\circ i_{g'}\circ L'_0\circ s,
	\end{equation*}
	which is a simplicial map from $\mathcal{P}(\Gamma)$ to $F(\Gamma)$. Pick arbitrary $g\in G(\Gamma)$. We need to show $L\circ i_{g}$ is a simplicial isomorphism. Let $K_{g}=i_{g}(F(\Gamma))$ and let $g_{1}'\in G(\Gamma')$ such that $g_{1}'\cdot \phi(g)=g'$. Then $i_{g'}\circ L'_0|_{s(K_{g})}=\bar{s}_{g_{1}'}|_{s(K_{g})}$. Thus 
	\begin{align*}
	L\circ i_{g}=L_0\circ s^{-1}\circ i_{g'}\circ L'_0\circ s\circ i_{g}=L_0\circ s^{-1}\circ \bar{s}_{g'_{1}}\circ s\circ i_{g}=L_0\circ s_{g'_{1}}\circ i_{g},
	\end{align*}
	which is a simplicial isomorphism by Lemma \ref{4.10}. It follows that $L$ is a coherent labeling, moreover,
	\begin{align*}
	(s_{g'})^{\ast}L &= (L_0\circ s^{-1}\circ i_q\circ L'_0\circ s)\circ(s^{-1}\circ \bar{s}_{g'}\circ s)=L_0\circ s^{-1}\circ i_q\circ L'_0\circ \bar{s}_{g'}\circ s \\
	&=L_0\circ s^{-1}\circ i_q\circ L'_0\circ s=L
	\end{align*}
	for any $g'\in G(\Gamma')$, where the third equality follows from $G(\Gamma')$-invariance of $L'_0$. So $L$ is the required coherent labeling.
\end{proof}

%$L$ is defined via the composition $\P(\Ga)\xrightarrow{s}\P(\Ga')\xrightarrow{L'}F(\Ga')\xrightarrow{p} F(\Ga)$. To define $p$, we pick a base point $x'\in X(\Ga')$, then $p$ is defined via the composition $F(\Ga')\to F(\Ga')_{x'}\hookrightarrow \P(\Ga')\xrightarrow{s^{-1}}\P(\Ga)\xrightarrow{\pi} F(\Ga)$. 

It remains to construct an $L$-atlas $\bar{\mathcal A}_L$ for $G(\Gamma)$ satisfying all the requirement. This is the main part of the proof of Proposition~\ref{prop:atlas}. Note that in the light of Lemma \ref{coherent labelling}, for verifying Proposition~\ref{prop:atlas} (1), it
suffices to show charts of $\bar{\mathcal A}_L$ are $G(\Gamma')$-invariant up to translations. We will construct $\bar{\mathcal A}_L$ by induction on the dimension of charts. 

By induction, we assume the charts are already defined for standard flats in $\mathcal{F}(\Gamma)$ of dimension $\le k-1$ such that the following inductive assumptions hold true:
\begin{enumerate}[leftmargin=0.7cm]
\item The charts are compatible and $G(\Gamma')$-invariant up to translations.
\item (inverse images are boxes) Given $F\in \mathcal{F}_{< k}(\Gamma)$, let $F'=\phi(F)$ and let $$s_{0}:v(L(\Delta(F)))\to v(L'(\Delta(F')))$$ be the bijection induced by $s$. Suppose $\bar{h}$ and $h'$ are charts for $F$ and $F'$ respectively. Then $\varphi=h'\circ\phi\circ\bar{h}$ admits splitting $$\varphi=\prod_{w\in v(L(\Delta(F)))}\varphi_{w}$$ where $$\varphi_{w}:\Bbb Z^{w}\to \Bbb Z^{s_{0}(w)}$$ is of form $$\varphi_{w}(a)=\lfloor a/d_{w}\rfloor+r_{w}\ \ \ (a\in \Bbb Z^{w})$$ for some integers $r_{w}$ and $d_{w}$ ($d_{w}>0$).
\item (extension condition) For $F_{1},F_{2}\in \mathcal{F}_{< k}(\Gamma)$ such that $\phi(v(F_{1}))=\phi(v(F_{2}))=v(F')$, there is a bijections $f:v(F_{1})\to v(F_{2})$ such that 
\begin{enumerate}
\item $\phi(x)=\phi\circ f(x)$ for any $x\in v(F_{1})$.
\item $\bar{h}_{2}\circ f\circ \bar{h}^{-1}_{1}$ is a translation (here $\bar{h}_{i}:v(F_{i})\to \Bbb Z^{v(L(\Delta(F_{i})))}$ are charts).
\item Let $F\in\mathcal{F}(\Gamma)$ such that $F_{1}\cap F\neq\emptyset$, $F_{2}\cap F\neq\emptyset$ and the convex hull of $F_{1}$ and $F$ is a flat. Then $f(v(F_{1}\cap F))=v(F_{2}\cap F)$.
\end{enumerate}
\end{enumerate}

\begin{remark}
Note that only requiring condition (1) is not enough since the compatibility of existing charts does not imply that we can add more charts in a compatible way to obtain an atlas, thus we need condition (3). 
\end{remark}

For $i=1,2$, let $\varphi_{i}=h'\circ\phi\circ\bar{h}^{-1}_{i}$. Then (2) and (3) imply that $\varphi^{-1}_{1}(y)$ and $\varphi^{-1}_{2}(y)$ are boxes of the same size for any $y\in \Bbb Z^{v(L(\Delta(F')))}$. Thus (3a) and (3b) uniquely determine the map $f$ and we call $f$ a \textit{charts-induced identification} (CII) between $v(F_{1})$ and $v(F_{2})$. 

In order to define charts for standard flats in $\mathcal{F}_{\le k}(\Gamma)$, we need a way to ``connect'' parallel copies of lower dimensional standard flats in a $k$-dimensional stable standard flat. For this purpose, we will define CII between parallel standard flats in $\mathcal{F}_{\le k-1}(\Ga)$ and prove some basic properties via a sequence of lemmas (Lemma~\ref{equivariant CII 0}, Lemma~\ref{equivariant CII}, Lemma~\ref{6.9}, Lemma~\ref{6.10} and Lemma~\ref{6.11}). We will define charts after these preparatory lemmas.

\begin{lem}
\label{equivariant CII 0}
The map $f$ is $\Stab(v(F'))$-equivariant.
\end{lem}

\begin{proof}
Recall that $\Stab(v(F'))=\Stab(v(F_1))=\Stab(v(F_2))$ by Lemma \ref{6.2}. By (1), the induced action of $\Stab(v(F'))$ on the range of $\bar{h}_1$ (or $\bar{h}_2$) is an action by translations, moreover, this action is completely determined by the size of the box $\varphi^{-1}_1(y)$ (or $\varphi^{-1}_{2}(y)$). It follows from (3a) and (3b) that $\bar{h}_{2}\circ f\circ \bar{h}^{-1}_{1}$ is $\Stab(v(F'))$-equivariant. Then the lemma follows.
\end{proof}

%\begin{remark}
%Recall that in \cite[Lemma 5.7]{raagqi1}, all CIIs between standard geodesics are induced by parallelism. This is relaxed to (3c), which says CII is induced by parallelism whenever it has to be. 
%\end{remark}

Let $F_{1},F_{2}\in \mathcal{F}_{<k}(\Gamma)$ be two parallel elements. If $F'_{1}=\phi(F_{1})$ and $F'_{2}=\phi(F_{1})$, then $F'_{1}$ and $F'_{2}$ are parallel. Let $p:v(F'_{1})\to v(F'_{2})$ be the map induced by parallelism, i.e. $p$ sends a vertex in $F'_1$ to the nearest point in $F'_2$ (which is a vertex) with respect to the metric on the ambient CAT(0) cube complex.
Let $g'$ be the unique element in $G(\Gamma')$ such that $\bar{\phi}_{g'}|_{v(F'_{1})}=p$. Suppose $F_{21}=\phi_{g'}(F_{1})$. Then $$\phi(F_{21})=\phi(F_{2})$$ by (\ref{5.8}). We define the \emph{chart induced identification} $$f:v(F_{1})\to v(F_{2})$$ between $v(F_1)$ and $v(F_2)$ by $$f=f_{1}\circ\phi_{g'}$$ where $f_{1}$ is the CII between $v(F_{21})$ and $v(F_{2})$. 

\begin{lem}
\label{equivariant CII}
The CII map $f$ is $\Stab(v(F_1))$-equivariant.
\end{lem}

\begin{proof}
Note that $\Stab(v(F_1))=\Stab(v(F_2))=\Stab(v(F'_1))=\Stab(v(F'_2))$. Since $g'$ commutes with any element in $\Stab(v(F_1))$, $\phi_{g'}$ is $\Stab(v(F_1))$-equivariant. By Lemma \ref{equivariant CII 0}, $f_1$ is $\Stab(v(F_1))$-equivariant. Thus the lemma follows.
\end{proof}

\begin{lem} 
	\label{6.9}
	The following properties of $f$ are true.
	\begin{enumerate}
	\item The map $f$ satisfies all the properties in inductive assumption (3) with (3a) replaced by $\bar{\phi}_{g'}\circ\phi=\phi\circ f$.	
	\item The map $f$ is uniquely characterized by the following two properties:
	\begin{itemize}
		\item $\bar{h}_{2}\circ f\circ \bar{h}^{-1}_{1}$ is a translation.
		\item $\phi\circ f(x)=p\circ \phi(x)$ for any $x\in v(F_{1})$, where $p:v(F'_{1})\to v(F'_{2})$ is the parallelism map.
	\end{itemize} 
	\end{enumerate}
\end{lem}

\begin{proof}
We prove the first assertion of the lemma. Condition (3a) and (3b) follows from inductive assumption (1). It suffices to check (3c). Let $F'=\phi(F)$. Then the convex hull of $F'$ and $F'_{1}$ (or $F'_{2}$) is also a flat, thus $\bar{\phi}_{g'}(v(F'))=v(F')$ and (\ref{6.3}) implies $\phi_{g'}(v(F))=v(F)$. It follows that $$\phi_{g'}(v(F\cap F_{1}))=v(F\cap F_{21}).$$ Note that $F_{21}$ and $F_{2}$ are in the convex hull of $F_{1}$ and $F$, thus $$f_{1}(v(F\cap F_{21}))=v(F\cap F_{2})$$ by (3c), which implies $$f(v(F\cap F_{1}))=v(F\cap F_{2}).$$ The second assertion of the lemma follows from the first assertion and inductive assumption (2).
\end{proof}

\begin{lem}
\label{6.10}
Let $\{F_{i}\}_{i=1}^{4}\subset \mathcal{F}_{<k}(\Gamma)$ such that $F_{1}$, $F_{2}$ and $F_{3}$ are parallel. Suppose $f_{ij}$ is the CII between $v(F_{i})$ and $v(F_{j})$ and $\bar{h}_{i}$ is the chart for $F_{i}$. Then 
\begin{enumerate}
\item $f_{13}=f_{23}\circ f_{12}$.
\item If $F_{4}\subset F_{1}$, then $f_{12}(v(F_{4}))$ is the vertex set of some standard flat and $f_{12}|_{F_{4}}$ is the CII between $v(F_{4})$ and $f_{12}(v(F_{4}))$.
\item If $F_{i}\subset F_{4}$ for $i=1,2$, then $f_{12}$ coincides with the map induced by parallelism between $\bar{h}_{4}(v(F_{1}))$ and $\bar{h}_{4}(v(F_{2}))$ in $\Bbb Z^{v(L(\Delta(F_{4})))}$ (note: for the definition of parallelism map, we treat $\Bbb Z^{v(L(\Delta(F_{4})))}$ as an integer lattice in a Euclidean space, and send a point in $\bar{h}_{4}(v(F_{1}))$ to the nearest point in $\bar{h}_{4}(v(F_{2}))$ with respect to the Euclidean metric).
\item CIIs are $G(\Gamma')$-invariant. Namely for any $g'\in G(\Gamma')$, the CII between $\phi_{g'}(F_{1})$ and $\phi_{g'}(F_{2})$ is given by $\phi_{g'}\circ f_{12}\circ\phi_{g'}^{-1}$.
%\item Induction assumption (3c) is true with $f$ replaced by $f_{12}$ $($we do not need $\phi(v(F_{1}))=\phi(v(F_{2})))$.
\end{enumerate}
\end{lem} 

\begin{proof}
The first assertion is a consequence of Lemma \ref{6.9} (2).
To see (2), by the compatibility of charts, there is a 1-1 correspondence between standard flats in $F_{2}$ that are parallel to $F_{4}$ and cosets of $\Bbb Z^{v(L(\Delta(F_{4})))}$ in $\Bbb Z^{v(L(\Delta(F_{2})))}$, but $\bar{h}_{2}(f_{12}(v(F_{4})))$ is such a coset by the first item of Lemma \ref{6.9} (2), thus $f_{12}(v(F_4))$ is the vertex set of some standard flat. Lemma \ref{6.9} (2) also implies that $f_{12}|_{F_4}$ is a CII. Assertion (3) follows from inductive assumption (2) and Lemma \ref{6.9} (2). Assertion (4) follows from inductive assumption (1) and Lemma \ref{6.9} (2). 
\end{proof}

%(5) is the claim before Lemma \ref{6.9}.
%Let $F$ be the flat in (3c) and suppose $F'=\phi(F)$ and $F'_{i}=\phi(F_{i})$. Then $F'\cap F'_{i}\neq\emptyset$ for $i=1,2$ and the convex hull of $F'$ and $F'_{1}$ is a flat. Let $g'\in G(\Gamma')$ be the element such that $\bar{\phi}_{g'}(F'_{1})=F'_{2}$ and $\bar{\phi}_{g'}|_{v(F'_{1})}$ coincides with the map induced by parallelism. Then $\phi_{g'}|_{v(F_{1})}$ is the CII between $\bar{F}_{1}=\phi_{g'}(F_{1})$ and $F_{1}$. Note that $g'\in \Stab(v(F'))=\Stab(v(F))$, so $\phi_{g'}(v(F)\cap v(F_{1}))=\phi_{g'}(v(F))\cap\phi_{g'}(v(F_{1}))=v(F)\cap v(\bar{F}_{1})$. Since $\phi(\bar{F}_{1})=\phi(F_{2})$, $f(v(F)\cap v(\bar{F}_{1}))=v(F)\cap v(F_{2})$ where $f:v(\bar{F}_{1})\to v(F_{2})$ is the CII. By (1), we have $f_{12}=f\circ\phi_{g'}$, then (5) follows.

Take $F_{1},F\in\mathcal{F}(\Gamma)$ with $\dim(F_{1})<k$ and $F_{1}\subset F$, we define a map $\pi_{1}:v(F)\to v(F_{1})$ as follows. For standard flat $K\subset F$ with $K$ parallel to $F_{1}$, we set $\pi|_{v(K)}=f_{K}$ where $f_{K}:v(K)\to v(F_{1})$ is the CII between $v(K)$ and $v(F_{1})$. We call $\pi_{1}$ a \textit{charts-induced projection} (CIP). 

\begin{lem}
\label{6.11}
Let $F'_{1}=\phi(F_{1})$ and $F'=\phi(F)$. Suppose $F'_{2}$ is an orthogonal complement of $F'_{1}$ in $F'$ and $\bar{h}_{1}$ is the chart for $F_{1}$. Then
\begin{enumerate}
\item $\pi_{1}\circ\bar{h}_{1}$ is $\Stab(v(F'_{2}))$-invariant and $\Stab(v(F'_{1}))$-invariant up to translation, hence is $\Stab(v(F'))$-invariant up to translation.
\item Pick $F_{3}\in\mathcal{F}(\Gamma)$ such that $F_{3}\subset F_{1}$. Let $\pi_{3}:v(F)\to v(F_{3})$ and $\pi_{13}:v(F_{1})\to v(F_{3})$ be CIPs. Then $\pi_{3}=\pi_{13}\circ\pi_{1}$.
\item Assume $\dim(F)<k$ and let $\bar{h}$ be the chart for $F$. Then $\pi_{1}$ coincides with the map induced by the natural projection from $\bar{h}(F)$ to $\bar{h}(F_{1})$ in $\Bbb Z^{v(L(\Delta(F)))}$.
\item Suppose $\pi'_{1}$ is the orthogonal projection $v(F')\to v(F'_{1})$. Then $$\phi\circ\pi_{1}(x)=\pi'_{1}\circ\phi(x)$$ for any $x\in v(F)$.
\item Let $F_{3}\in\mathcal{F}(\Gamma)$ be a standard flat in $F$. Then there exists stable standard flat $F_{4}\in F_{1}$ such that $\pi_{1}(v(F_{3}))=v(F_{4})$ ($F_{4}$ could be a point). Moreover, let $\pi_{4}:F\to F_{4}$ be the CIP. Then $$\pi_{1}|_{v(F_{3})}=\pi_{4}|_{v(F_{3})}.$$
\end{enumerate}
\end{lem}

\begin{proof}
To see (1), note that any element in $\Stab(v(F'_2))$ maps $F'_1$ to a flat parallel to $F'_1$, and this maps is exactly the parallelism map. It follows from Lemma \ref{6.9} (2) that $\pi_{1}\circ\bar{h}_{1}$ is $\Stab(v(F'_{2}))$-invariant. Lemma \ref{equivariant CII} and inductive assumption (1) imply $\pi_{1}\circ\bar{h}_{1}$ is $\Stab(v(F'_{1}))$-invariant up to translation. Assertion (2) follows from Assertion (1) and Lemma \ref{6.10} (2). Assertion (3) follows from Lemma \ref{6.10} (3). Assertion (4) follows from Lemma \ref{6.9} (2). To see (5), we first assume $F_{3}\cap F_{1}\neq\emptyset$ and take $F_{4}=F_{1}\cap F_{3}$, then $$\pi_{1}(v(F_{3}))=v(F_{4})$$ by Lemma~\ref{6.9} (1). In general, we pick a standard flat $\tilde{F}_{1}$ parallel to $F_{1}$ such that $\tilde{F}_{1}\cap F_{3}\neq\emptyset$. Let $f_{1}:\tilde{F}_{1}\to F_{1}$ be the CII and let $\tilde{\pi}_{1}:F\to \tilde{F}_{1}$ be the CIP. Then $f_{1}\circ\tilde{\pi}_{1}=\pi_{1}$, which reduces the problem to the previous case. The second assertion in (5) follows from (2).
\end{proof}

We will construct charts for elements in $\mathcal{F}_{k}(\Gamma)$ in three steps.

\textbf{Step 1: We construct chart for a single element in $\mathcal{F}_{k}(\Gamma)$.}

Pick a standard $k$-flat $F\in\mathcal{F}(\Gamma)$ and vertex $p\in F$. Let $F_{m}$ be the convex hull of all standard flats that are properly contained in $F$, pass through $p$ and belong to $\mathcal{F}(\Gamma)$. Then $F_{m}\in \mathcal{F}(\Gamma)$ by Lemma \ref{3.50}. We divide into three cases depending on the size of $F_m$. In each case we will construct a chart $\bar h: v(F)\to \Bbb Z^{v(L(\Delta(F)))}$ for $F$ and shall verify 
\begin{enumerate}
	\item $\bar h$ is compatible with charts of elements in $\mathcal{F}_{<k}(\Gamma)$;
	\item $\bar h$ is $\Stab(v(F'))$-invariant up to translation;
	\item inductive assumption (2) holds for $\bar h$.
\end{enumerate}

\textit{Case 1:} $F_{m}$ is a point. Let $F'=\phi(F)$ and let $$h':v(F')\to \Bbb Z^{v(L'(\Delta(F')))}$$ be the chart for $F'$. Define $$h=h'\circ\phi:v(F)\to \Bbb Z^{v(L'(\Delta(F')))}.$$ We assign an arbitrary bijection between $v(L'(\Delta(F')))$ and $\{1,2,\ldots, k\}$ with $$k=|v(L'(\Delta(F')))|,$$ which leads to an identification of $\Bbb Z^{v(L'(\Delta(F')))}$ with $\Bbb Z^k$. As $h$ is a map from $v(F)$ from $\Bbb Z^k$, we can write  $$h=(h_{1},h_{2},\cdots,h_{k})$$ where each $h_i$ is a coordinate component of $h$. Denote the identity element in $\Bbb Z^{v(L'(\Delta(F')))}$ by \textbf{0} and let $r=|h^{-1}(\textbf{0})|$. Since elements in $h^{-1}(\textbf{0})$ are representatives of the orbits of the action $\Stab(v(F'))\curvearrowright v(F)$, there is a natural map $v(F)\to h^{-1}(\textbf{0})$. By post-composing this map with a bijection $$h^{-1}(\textbf{0})\to\{0,1,\cdots,r-1\},$$ we obtain a $\Stab(v(F'))$-invariant map $$\chi:v(F)\to\{0,1,\cdots,r-1\}.$$ Now define $$\tilde{h}:v(F)\to \Bbb Z^{v(L'(\Delta(F')))}$$ by sending $x\in v(F)$ to $$(rh_{1}(x)+\chi(x),h_{2}(x),\cdots,h_{k}(x)),$$ then $\tilde{h}$ is a bijection and we have the following commutative diagram:
\begin{center}
\begin{tikzpicture}[scale=1.5]
\node (A) at (0,1) {$v(F)$};
\node (B) at (2.5,1) {$v(F')$};
\node (C) at (0,0) {$\Bbb Z^{v(L'(\Delta(F')))}$};
\node (D) at (2.5,0) {$\Bbb Z^{v(L'(\Delta(F')))}$};
\node (E) at (-2.5,0) {$\Bbb Z^{v(L(\Delta(F)))}$};
\path[->,font=\scriptsize,>=angle 90]
(A) edge node[above]{$\phi$} (B)
(A) edge node[right]{$\tilde{h}$} (C)
(B) edge node[right]{$h'$} (D)
(C) edge node[above]{$\phi'$} (D)
(A) edge node[above]{$h$} (D)
(A) edge node[above]{$\bar{h}$} (E)
(E) edge node[above]{$s'$} (C);
\end{tikzpicture}
\end{center}
Here $\phi'$ is the map induced by $\phi$, $s'$ is the bijection induced by $s:\Delta(F)\to\Delta(F')$ and $\bar{h}=s'^{-1}\circ \tilde{h}$. By construction, $\bar{h}$ is $\Stab(v(F'))$-invariant up to translation and satisfies inductive assumption (2). We choose $\bar{h}$ to be the chart for $F$, which is trivially compatible with the charts already defined.

\textit{Case 2:} $p\subsetneq F_{m}\subsetneq F$. Let $F'$ and $h$ be as before and let $F'_{m}=\phi(F_{m})$. Suppose $F_{c}$ (or $F'_{c}$) is a standard flat which is the orthogonal complement of $F_{m}$ (or $F'_{m}$) in $F$ (or $F'$). Then we have the following commuting diagram:
\begin{center}
$\begin{CD}
v(F)                         @>h>>        \Bbb Z^{v(L'(\Delta(F')))}\\
@VV\pi V                                                              @VV\pi'_{c} V\\
v(F_{c})                  @>h_{c}>>       \Bbb Z^{v(L'(\Delta(F'_{c})))}
\end{CD}$
\end{center}
Here $\pi$ and $\pi'_{c}$ are the natural projections. Note that $h$ maps fibres of $\pi$ to fibres of $\pi'_c$, which induces $h_{c}$. The action of $\Stab(v(F'_{c}))$ permutes the fibres of $\pi$, which induces an action $$\Stab(v(F'_{c}))\curvearrowright v(F_{c}).$$ As in case 1, we can obtain from $h_{c}$ a bijection $$\bar{h}_{c}:v(F_{c})\to\Bbb Z^{v(L(\Delta(F_{c})))}$$ which is $\Stab(v(F'_{c}))$-invariant up to translation. Then $\bar{h}_{c}\circ \pi$ is $\Stab(v(F'_{m}))$-invariant (since $\Stab(v(F'_m))$ stabilizes each fibre of $\pi$ by (\ref{6.3})) and $\Stab(v(F'_{c}))$-invariant up to translation.

Let $\bar{h}_{m}$ be the composition $$v(F)\to v(F_{m})\to\Bbb Z^{v(L(\Delta(F_{m})))}$$ of a CIP with a chart map. Then $\bar{h}_{m}$ is $\Stab(v(F'_{m}))$-invariant up to translation and $\Stab(v(F'_{c}))$-invariant by (1) of Lemma \ref{6.11}. Now we identify $\Bbb Z^{v(L(\Delta(F_{m})))}$ and $\Bbb Z^{v(L(\Delta(F_{c})))}$ as subgroups of $\Bbb Z^{v(L(\Delta(F)))}$ and define $$\bar{h}:v(F)\to \Bbb Z^{v(L(\Delta(F)))}$$ by $$\bar{h}=\bar{h}_{c}\circ \pi+\bar{h}_{m}.$$ It is clear that the bijection $\bar{h}$ is $\Stab(v(F'))$-invariant up to translation. We choose $\bar{h}$ to be the chart for $F$ and the compatibility follows from our construction.

Let $s_{0}$ and $\varphi$ be as in inductive assumption (2). We claim cosets of $\Bbb Z^{w}$ are mapped to cosets of $\Bbb Z^{s_{0}(w)}$ under $\varphi$ for any $w\in v(L(\Delta(F)))$. If $w\in v(L(\Delta(F_{m})))$, then the claim follows from Lemma \ref{6.9} (2) and inductive assumption (2) for $F_m$. If $w\in v(L(\Delta(F_{c})))$, then by Lemma \ref{6.11} (4), $\bar{h}_m$ maps $\Bbb Z^{w}$-cosets to points. The claim follows from the construction of $\bar{h}_{c}$. Thus $\varphi$ splits into products and $\bar{h}$ satisfies inductive assumption (2).

\textit{Case 3:} $F_{m}= F$. Then there exist standard flats $F_{1},F_{2}\in\mathcal{F}_{<k}(\Gamma)$ such that $F$ is the convex hull of $F_{1}$ and $F_{2}$ (a basic case to bare in mind is that $F_1\cap F_2$ is a single point). Let $F_{3}=F_{1}\cap F_{2}$. Suppose $F'=\phi(F)$ and $F'_{i}=\phi(F_{i})$. Take $\bar{h}_{i}$ to be the charts for $F_{i}$ for $1\le i\le 3$ and take $\bar{\pi}_{i}:F\to F_{i}$ to be the CIP for $1\le i\le 3$. Define $$\bar{h}:F\to \Bbb Z^{v(L(\Delta(F)))}$$ by $$\bar{h}=\bar{h}_{1}\circ\bar{\pi}_{1}+\bar{h}_{2}\circ\bar{\pi}_{2}-\bar{h}_{3}\circ\bar{\pi}_{3}.$$ Then by Lemma \ref{6.11} (1), $\bar{h}$ is $\Stab(v(F'))$-invariant up to translation. 

\begin{lem}
	\label{lem:properties}
The following holds true.
\begin{enumerate}
	\item The map $\bar{h}$ is a bijection.
	\item The map $\bar{h}$ is compatible with charts of elements in $\mathcal{F}_{<k}(\Gamma)$.
	\item The map $\bar{h}$ satisfies inductive assumption (2).
\end{enumerate}
\end{lem}

\begin{proof}
For Assertion (1), note that $\bar h$ induces a bijection between standard flats in $F$ which are parallel to $F_3$ and cosets of $\Bbb Z^{v(L(\Delta(F_{3})))}$ in $\Bbb Z^{v(L(\Delta(F)))}$. It suffices to show for any standard flat $\tilde{F}_{3}\subset F$ parallel to $F_{3}$, $\bar{h}$ maps $\tilde{F}_{3}$ bijectively to a coset of $\Bbb Z^{v(L(\Delta(F_{3})))}$. Note that by Lemma \ref{6.9} (2), if we change the standard flats $F_{1}$ and $F_{2}$ in the definition of $\bar{h}$ to some other flats parallel to them, then $\bar{h}$ would differ by translation, thus we can assume $\tilde{F}_{3}=F_{3}$. But $\bar{h}$ restricted to $F_{3}$ is of form $\bar{h}_{1}+\bar{h}_{2}-\bar{h}_{3}$, so what we need to prove is implied by the compatibility condition.

Now we prove Assertion (2). Let $F_{4}\subset F$ be an element in $\mathcal{F}_{<k}(\Gamma)$ and let $\bar{h}_{4}$ be its chart. We can assume $F_{i}\cap F_{4}\neq\emptyset$ by moving $F_{1}$ and $F_{2}$ appropriately as before. For $1\le i\le 3$, let $F_{4i}=F_{4}\cap F_{i}$, let $\bar{\pi}_{4i}:F\to F_{4i}$ be the CIP and let $\bar{h}_{4i}$ be the chart for $F_{4i}$. By (5) of Lemma \ref{6.11}, $\pi_{i}(F_{4})=F_{4i}$ for $1\le i\le 3$ and $$\bar{h}=\bar{h}_{1}\circ\bar{\pi}_{41}+\bar{h}_{2}\circ\bar{\pi}_{42}-\bar{h}_{3}\circ\bar{\pi}_{43}$$ when restricted on $F_{4}$. On the other hand, (3) of Lemma \ref{6.11} and the compatibility condition imply $$\bar{h}_{4}=\bar{h}_{41}\circ\bar{\pi}_{41}+\bar{h}_{42}\circ\bar{\pi}_{42}-\bar{h}_{43}\circ\bar{\pi}_{43}$$ up to translation. Now the compatibility of $\bar{h}_{4}$ and $\bar{h}$ follows from the compatibility of $\bar{h}_{i}$ and $\bar{h}_{4i}$ ($1\le i\le 3$).

For Assertion (3), it suffices to show $\bar{h}$ restricted on each $\Bbb Z$ coset has the desired property. Let $w\in v(L(\Delta(F)))$ and let $K$ be a $\Bbb Z^{w}$ coset. Then $K$ is contained in either a $\Bbb Z^{v(L(\Delta(F_{1})))}$ coset or a $\Bbb Z^{v(L(\Delta(F_{2})))}$ coset. Since $\bar{h}$ is compatible with other charts, there exists a standard flat $F_{5}\subset F$ which is parallel to either $F_{1}$ or $F_{2}$ such that $K\subset\bar{h}(v(F_{5}))$. Moreover, if $\bar{h}_{5}$ is the chart for $F_{5}$, then $\bar{h}_{5}\circ\bar{h}^{-1}(K)$ is again a $\Bbb Z^{w}$ coset. By applying the inductive assumption to $\bar{h}_{5}$, we know $\bar{h}|_{K}$ has the desired behavior.
\end{proof}

\begin{lem}
	\label{lem:translation}
If we choose different $F_{1}$ and $F_{2}$ in the definition of $\bar{h}$, then the resulting chart remains the same up to translation.
\end{lem}

\begin{proof}
Let $\bar h$ be the chart defined above using $F_1$ and $F_2$. 
The lemma follows if we know that for any $F_3\in\mathcal{F}_{<k}(\Gamma)$, the CIP from $v(F)$ to $v(F_3)$ coincides with the map induced by the natural projection from $\bar h(F)$ to $\bar h(F_3)$ in $\mathbb Z^{v(L(\Delta(F)))}$. To show this property of CIP holds, it suffices to show that for any pair of parallel elements $F_4,F_5\subset F$ in $\mathcal{F}_{<k}(\Gamma)$, the CII between $v(F_4)$ and $v(F_5)$ coincides with the map induced by parallelism between $\bar(v(F_4))$ and $\bar(v(F_5))$ in $\mathbb Z^{v(L(\Delta(F)))}$. This can be proved in the same way as Lemma \ref{6.10} (3), using Lemma~\ref{lem:properties} (2).
\end{proof}
\bigskip

\textbf{Step 2: We construct charts for flats with the same $\phi$-image as $F$.}

We define a graph $\Lambda(F)$. Its vertices are in 1-1 correspondence to standard flats that have the same $\phi$-image as $F$ and two vertices are joined by an edge if and only if the corresponding flats are \textit{bolted}, defined as follows.
\begin{definition}
Two parallel elements $H_{1},H_{2}\in \mathcal{F}(\Gamma)$ are \textit{bolted} if there exists $H\in \mathcal{F}(\Gamma)$ such that for $i=1,2$, $H\cap H_{i}\neq\emptyset$, $H\cap H_{i}\subsetneq H_{i}$ and the convex hull of $H$ and $H_{1}$ is a flat. The standard flat $H$ is called a $(H_{1},H_{2})-$\textit{bolt}, we will omit $(H_{1},H_{2})$ when they are clear. 
\end{definition}

The main goal of Step 2 is the following. 
\begin{lem}
	\label{lem:step 2}
Let $F'=\phi(F)$.  There exists a collection of charts, one for each vertex in $\Lambda(F)$, such that 
\begin{enumerate}
	\item each chart is compatible with charts of elements in $\mathcal{F}_{<k}(\Gamma)$;
	\item each chart is $\Stab(v(F'))$-invariant up to translation;
	\item inductive assumption (2) holds for each chart;
	\item inductive assumption (3) is satisfied for each pair of charts.
\end{enumerate}
\end{lem}

We choose a representative in each connected component of $\Lambda(F)$ (the representative in the component containing $F$ is chosen to be $F$), which gives a collection $\{F_\lambda\}_{\lambda\in\Lambda}$. We build a chart $\bar h_{\lambda}$ for each $F_\lambda$ as in Step 1. By Step 1, Lemma~\ref{lem:step 2} (1), (2) and (3) holds for each $\bar h_{\lambda}$. Now we arrange Lemma~\ref{lem:step 2} (4) for each pair of charts in $\{\bar h_\lambda\}_{\lambda\in\Lambda}$.

\begin{lem}
	\label{lem:same dimension}
Take $\{H_1,H_2\}\subset \{F_\lambda\}_{\lambda\in\Lambda}$	with their charts $\bar h_1$ and $\bar h_2$. Then for any $y\in v(F')$, $(\phi\circ\bar{h}^{-1}_{1})^{-1}(y)$ and $(\phi\circ\bar{h}^{-1}_2)^{-1}(y)$ are boxes of the same dimension.
\end{lem}

\begin{proof}	
Note that $H_1$ and $H_2$ must be in the same case of Step 1, so the lemma follows from Lemma \ref{6.2} (3) in case 1 of Step 1. In case 2 of Step 1, the lemma is a consequence of Lemma \ref{6.2} (3) and following obersavation which follows from Lemma~\ref{6.9} (2): for any two parallel standard flats $F_3,F_4\subset \mathcal{F}_{k-1}(\Gamma)$ with charts $\bar h_3$ and $\bar h_4$, $(\phi\circ\bar{h}^{-1}_{3})^{-1}(y_3)$ and $(\phi\circ\bar{h}^{-1}_4)^{-1}(y_4)$ are boxes of the same dimension where $y_i$ is a point in the image of $\phi\circ\bar{h}^{-1}_{i}$ for $i=3,4$.
In case 3 of Step 1, by Lemma~\ref{lem:translation}, (up a translation) the definition of $\bar{h}_i$ does not depend on the choice of the pair of stable flats in $H_i$. So we can choose them such that they are parallel to $F_{1},F_{2}\subset F$ as in Case 3 of Step 1. This, together with the previous observation imply the lemma.
\end{proof}
By Lemma~\ref{lem:same dimension}, there exists a unique identification $f:v(H_1)\to v(H_2)$ characterized by (3a) and (3b) of the inductive assumption. (3c) is trivially true for $f$ and $f$ is $\Stab(v(F'))$-equivariant since both $\bar{h}_1$ and $\bar{h}_{2}$ are $\Stab(v(F'))$-invariant up to translation. Thus Lemma~\ref{lem:step 2} (4) holds for each pair of charts in $\{\bar h_\lambda\}_{\lambda\in\Lambda}$.

%Now every representative has been identified with $F$ and induced identification between representatives obviously satisfies inductive assumption (3).

It remains to define charts for flats inside one connected component of $\Lambda(F)$, so we assume now $\Lambda(F)$ is connected. 

\begin{lem}
\label{construction of CII}
There is a collection of bijections between each pair of flats in $\Lambda(F)$, which is also called CII's, such that
\begin{enumerate}
\item These CII's are compatible under compositions.
\item Each CII is $\Stab(v(F'))$-equivariant and satisfies inductive assumption (3a) and (3c). 
\item Let $f:H_1\to H_2$ be a CII between flats $H_1$ and $H_2$ in $\Lambda(F)$. Suppose $S_{1}\in\mathcal{F}_{<k}(\Gamma)$ be a standard flat in $H_{1}$. Then there exists $S_{2}\in \mathcal{F}_{<k}(\Gamma)$ parallel to $S_{1}$ such that $f(v(S_{1}))=v(S_{2})$ and $f|_{v(S_{1})}$ is the CII between $v(S_{1})$ and $v(S_{2})$.
\end{enumerate}
\end{lem}

Assuming Lemma~\ref{construction of CII}, we can finish the proof of Lemma~\ref{lem:step 2} as follows.
For any flat $H$ in $\Lambda(F)$, we define the chart of $H$ to be the composition of the CII between $H$ and $F$, and the chart map of $F$. This chart satisfies inductive assumption (2) since $F$ also satisfies this condition and the CII satisfies (3a). Recall that the chart of $F$ is compatible with the charts for flats in $\F_{<k}(\Ga)$, so is the chart of $H$ by Lemma \ref{construction of CII} (3). Moreover, this chart is $\Stab(v(F'))$ invariant up to translation by Lemma \ref{construction of CII} (2). Under such definition of charts, the CII between $F$ and another flat in $\Lambda(F)$ trivially satisfies inductive assumption (3b), hence the CII between any two flats in $\Lambda(F)$ satisfies inductive assumption (3b) by Lemma \ref{construction of CII} (1). 

\begin{proof}[Proof of Lemma \ref{construction of CII}]
We will again follow the three cases of how large is $F_m$ in $F$ as in Step 1.	
In case 1, we define the CII between any two flats in $\Lambda(F)$ to be the map induced by parallelism, then (1) of Lemma \ref{construction of CII} is true. Let $F_{1}$ and $F_{2}$ be a pair of bolted flats and let $H$ be a bolt. Suppose $f_{12}:v(F_{1})\to v(F_{2})$ is the CII. Then for $i=1,2$, $H\cap F_{i}$ must be one point and we denote it by $p_{i}$. It is clear that $f_{12}(p_{1})=p_{2}$, thus inductive assumption (3c) follows. Moreover, 
\begin{align*}
\phi(p_{1})&=\phi(v(H)\cap v(F_{1}))=\phi(v(H))\cap\phi(v(F_{1}))=\phi(v(H))\cap\phi(v(F_{2}))\\
&=\phi(v(H)\cap v(F_{2}))=\phi(p_{2})=\phi\circ f_{12}(p_{1}).
\end{align*}
The second and fourth equality follows from Lemma \ref{6.2} (2). Thus (3a) is true for bolted pair of flats. By moving the bolt $H$ around using the action of $\Stab(F')$, we know $f_{12}$ is $\Stab(F')$-equivariant. The connectivity of $\Lambda(F)$ implies that (3a) and the equivariance are true for all pair of flats in $\Lambda(F)$. This finishes case 1.

We need the following terminology before case 2. 
Let $H$ be a standard flat. An \textit{$H$-fibre} is a standard flat parallel to $H$. Let $H_{1},H_{2}\in\mathcal{F}(\Gamma)$ be parallel elements that contain $H$-fibres and let $p:v(H_{1})\to v(H_{2})$ be the map induced by parallelism. We say a bijection $f:v(H_{1})\to v(H_{2})$ is \textit{parallel mod $H$-fibres} if $f(v(H'))=p(v(H'))$ for any $H$-fibre $H'$. For standard flat $S_{i}\in H_{i}$, we will write $f(S_{1})=S_{2}$ if $f(v(S_{1}))=v(S_{2})$.

In case 2, for flats $H_{1}$ and $H_{2}$ in $\Lambda(F)$, we define the CII $f:v(H_{1})\to v(H_{2})$ such that $f$ is parallel mod $F_{m}$-fibres and for each $F_{m}$-fibre $T\subset H_{1}$, $f|_{v(T)}$ is the CII between $v(T)$ and $f(v(T))$. (1) follows from parallelism and Lemma \ref{6.10} (1) for CII's between $F_{m}$-fibres. Let $F_{1}$, $F_{2}$, $f_{12}$ and $H$ be as in case 1. Then for $i=1,2$, there exist $F_{m}$-fibres $F_{im}\subset F_{i}$ such that $F_{i}\cap H\subset F_{im}$. Note that $$f_{12}(v(F_{1m}))=v(F_{2m})$$ and $H$ is also a bolt for $F_{1m}$ and $F_{2m}$ when $F_{1}\cap H\subsetneq F_{1m}$, thus inductive assumption (3c) follows. By Lemma \ref{3.50}, we can assume $H\cap F_{i}$ is actually a $F_{m}$-fibre for $i=1,2$, then the argument in the previous case implies that the image of any $F_m$-fiber in $F_1$ under $\phi$ and $\phi\circ f_{12}$ are the same. Then (3a) follows since we already know it is true for CII's between $F_m$-fibres. The $\Stab(F')$-equivariance follows by applying Lemma \ref{6.10} (4) to CII's between $F_m$-fibres. Note that any element of $\mathcal{F}_{<k}(\Gamma)$ that lies in $F_{1}$ must stay inside a $F_{m}$-fibre, then Lemma \ref{construction of CII} (3) follows from Lemma \ref{6.10} (2).

In case 3, let $H_{1}$ and $H_{2}$ be a bolted pair in $\Lambda(F)$. Pick a vertex $p_{0}\in H_{1}$ and let $H$ be the intersection of all $(H_{1},H_{2})$-bolts that contains $p_{0}$. Then $H$ is also a bolt. We define the CII $$f:v(H_{1})\to v(H_{2})$$ as in case 2 with $F_{m}$-fibres replaced by $H\cap H_{1}$-fibres. The inductive assumption (3c) for $f$ follows from the minimality of $H$ and we can prove (3a) and the $\Stab(v(F'))$-equivariance as before.

Now we prove Lemma \ref{construction of CII} (3) for $f$. It is clear if $S_{1}$ stays inside a $H\cap H_{1}$-fibre. In general, pick a $H\cap H_{1}$-fibre $T_{1}$ such that $$T_{1}\cap S_{1}=S_{11}\neq\emptyset$$ and a standard flat $S_{12}$ which is an orthogonal complement of $S_{11}$ in $S_{1}$. Since $f$ is parallel mod $T_{1}$-fibres, $f(v(S_{1}))$ belongs to a $(T_{1}\times S_{12})$-fibre $R$. Suppose $S_{21}=f(S_{11})$ and $T_{2}=f(T_{1})$. Let $\pi_{i}:H_{i}\to T_{i}$ be the CIP for $i=1,2$. Then $$\pi_{1}(v(S_{1}))=S_{11}$$ by Lemma \ref{6.11} (5), hence $$\pi_{2}(f(v(S_{1})))=S_{21}$$ by Lemma \ref{6.10} (1). But every two $T_{1}$-fibres in $R$ are bolted by $S_{1}$-fibres, then the CII between these two $T_{1}$-fibres is parallel mod $S_{11}$-fibres, which implies $f(v(S_{1}))$ actually stays inside a $(S_{11}\times S_{12})$-fibre. To see the second part of Lemma \ref{construction of CII} (3), note that $S_{1}$ and $S_{2}$ are bolted by $H\cap H_{1}$-fibres, then the CII between them is parallel mod $S_{11}$-fibres by inductive assumption (3c). Thus the CII coincides with $f$ by Lemma \ref{6.10} (2).

For arbitrary pair $H_{1}$ and $H_{2}$ in $\Lambda(F)$, we choose an edge path in $\Lambda(F)$ connecting $H_{1}$ and $H_{2}$, which would induce a CII from $H_{1}$ to $H_{2}$. This CII will automatically satisfies $\Stab(v(F'))$-equivariance, inductive assumption (3a) and Lemma \ref{construction of CII} (3), since these properties are true under compositions. For this CII to be well-defined, we need to show every edge loop in $\Lambda(F)$ induces the identity map. Let $F$ be a base point in the edge loop and let $$f:v(F)\to v(F)$$ be the bijection induced by the edge loop. Pick $F_{1}, F_{2}\in\mathcal{F}_{<k}$ inside $F$ such that their convex hull is $F$, then it follows from (3c) that for $i=1,2$, every CII between two $F_{i}$-fibres in $F$ is parallel mod $F_{1}\cap F_{2}$-fibres. We first assume $F_{1}\cap F_{2}$ is a point. By previous discussion, $f$ maps $F_{i}$-fibre to $F_{i}$-fibre, thus $f$ splits into product $f=f_{1}\times f_{2}$ where $f_{i}:F_{i}\to F_{i}$ are bijections. Moreover, if $g:f(v(F_{1}))\to v(F_{1})$ is the CII, then $$g\circ f|_{v(F_{1})}=\textmd{Id}$$ by (1) of Lemma \ref{6.10}, thus $f|_{v(F_{1})}$ is induced by parallelism and $f_{2}=\textmd{Id}$. Similarly we can prove $f_{1}=\textmd{Id}$, thus $f=\textmd{Id}$. In general, we can run the same argument mod $F_{1}\cap F_{2}$-fibres to show that $f$ sends every $F_{1}\cap F_{2}$-fibre to itself, then $f=\textmd{Id}$ follows by applying Lemma \ref{6.10} (1) to $F_{1}\cap F_{2}$-fibres. 
\end{proof}

\textbf{Step 3: we define charts for any element in  $\mathcal{F}_{k}(\Gamma)$.} 

Let $F$ and $F'=\phi(F)$ be as in the previous steps.
Let $H$ be an element in $\mathcal{F}_{k}(\Gamma)$ such that $\phi(H)$ is in the $G(\Ga')$-orbit of $F'$. Note that this is equivalent to $L'(\Delta(\phi(H)))=L'(\Delta(F'))$. Pick $g'\in G(\Gamma')$ with $\bar{\phi}_{g'}(\phi(H))=F'$, then $\phi_{g'}(H)$ is an element in $\Lambda(F)$. We define the chart of $H$ to be the composition of the chart map of $\phi_{g'}(H)$ and $\phi_{g'}$. If we choose a different $g'$, the resulting chart would differ by a translation, since $\bar{h}$ is $\Stab(v(F'))$-invariant up to translation. By (\ref{5.8}), this chart satisfies inductive assumption (2). Moreover, it is compatible with charts of elements in $\F_{<k}(\Ga)$, since these charts are $G(\Ga')$-invariant up to translations, and they are compatible with charts of flats in $\Lambda(F)$ by the previous step.

By now we have defined a $G(\Gamma')$-invariant (up to translations) collection of charts for flats that are $G(\Gamma')$ orbits of flats in $\Lambda(F)$. This collection corresponds to a stable clique of $k$ vertices in $\Gamma'$, namely the 1-skeleton of $L'(\Delta(F'))$. For each stable $k$-clique in $\Gamma'$, we run the same argument to define charts for the corresponding collection of $k$-flats in $G(\Gamma)$. This gives rise to charts defined for all elements in $\mathcal{F}_{k}(\Gamma)$ that satisfies all the requirements, hence finishes the induction step. In summary, we have constructed a $G(\Gamma')$-invariant (up to translations) $L$-atlas $\bar{\mathcal{A}}_{L}$ such that inductive assumption (2) is true for all charts in this atlas. This finishes the proof of Proposition~\ref{prop:atlas}.

\section{Shuffling tiers and branches}
\label{sec:shuffle}
\subsection{Motivating discussion and overview}
Let $G(\Gamma_1)$ be a RAAG of type II. 
The goal of this section is to understand the class of RAAGs that are quasi-isometric to $G(\Gamma_1)$. The first motivating example to consider is $\Gamma_1$ being a pentagon. In this case, it is known before that any RAAG quasi-isometric to $G(\Gamma_1)$ is isomorphic to a special subgroup of $G(\Gamma_1)$ (\cite{raagqi1}). A key point in the proof is that any quasi-isometry $q:G(\Gamma_1)\to G(\Gamma_2)$ induces a simplicial isomorphism $\alpha:\mathcal{P}(\Gamma_1)\to\mathcal{P}(\Gamma_2)$ that is visible, as for each vertex $v\in \mathcal{P}(\Gamma_1)$, each $v$-tier only contain one $v$-branches. This follows from Corollary~\ref{7.14} and Lemma~\ref{7.1}.

Here is a more interesting example. Suppose $\Lambda_{m,n}$ is obtained by gluing $m$ copies of pentagon and $n$ copies of hexagon along a common closed vertex star. Let $\bar v$ be the central vertex of the common vertex star (see the figure below for $\Lambda_{3,2}$). Let $v\in\mathcal{P}(\Gamma)$ be a lift of $\bar v$. In this case, each $v$-tier contains $(m+n)$ $v$-branches, corresponding to the $(m+n)$ connected components of $\Gamma_1\setminus St(\bar v)$ (cf. Corollary~\ref{7.14}). Thus a simplicial isomorphism $\alpha:\mathcal{P}(\Gamma_1)\to\mathcal{P}(\Gamma_2)$ does not necessarily send a $v$-tier to an $\alpha(v)$-tier. However, the hope is that if we post-compose $\alpha$ by a suitable permutation of the $\alpha(v)$-tiers, then it might be possible for modified $\alpha$ to send $v$-tiers to $\alpha(v)$-tiers, and we could still obtain a visible map. 

\begin{center}
	\includegraphics[scale=0.6]{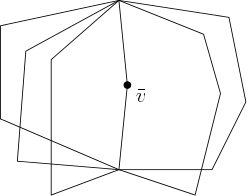}
\end{center}

However, there is an obstruction for this. For example, $G(\Lambda_{2,4})$ is an index 2 special subgroup of $G(\Lambda_{1,2})$. If $\alpha$ goes from $\mathcal{P}(\Lambda_{1,2})$ to $\mathcal{P}(\Lambda_{2,4})$, then it is possible to arrange as above such that it sends a $v$-tier to an $\alpha(v)$-tier. But this is impossible if the domain and range of $\alpha$ are exchanged, simply because in the domain, a $v$-tier has 6 $v$-branches, but in range an $\alpha(v)$-tier has 3 $v$-branches.

This is not the only obstruction. It turns out that as long as in the domain of $\alpha$, $m$ and $n$ are co-prime, then we can always modify $\alpha$ such that it preserves tiers. We generalize this to define a subclass of type II RAAGs, called prime RAAGs. At this point we would avoid give the full details of the definition of prime, and simply says that it involves symmetries of $v$-branches in the sense of Definition~\ref{8.1}. Some of this symmetries of $\mathcal{P}(\Gamma)$ are normal and expected, and they can be piecewisely defined using the action of $G(\Gamma)$ on $\mathcal{P}(\Gamma)$ (as the map $q$ and $q_*$ in the proof of Lemma~\ref{8.2}). However, there are other unexpected permutation of branches, which is the base for the definition of prime RAAGs. %This is what one should expect as in the end we need to be able to permute $v$-branches or $\alpha(v)$-branches.

The first step of the proof is to show if two prime RAAGs are quasi-isometric, then they are isomorphic. This is done in Section~\ref{subsec:prime}.

The second step is to show any RAAG of type II can be realized as a special subgroup of a prime RAAG. Putting these two steps together, we obtain the main theorem that any RAAG which is quasi-isometric to a RAAG of type II, is commensurable with this RAAG. 
The second step is more involved, and takes Section~\ref{subsec:cube}, Section~\ref{subsec:filtration}, and Section~\ref{subsec_wall space}. 

 Recall that if $G(\Gamma)$ were a special subgroup of $G(\Gamma')$, then $\Gamma'$ is a subgraph of $\Gamma$ and $\Gamma$ is a obtained by gluing multiple copies of $\Gamma'$ in a very specific way, and the gluing pattern is encoded in a compact CAT(0) cube complex. More precisely, there is a compact convex subcomplex $K\subset X(\Gamma')$ such that $F(\Gamma)$ is isomorphic to the full subcomplex of $\mathcal{P}(\Gamma)$ spanned by vertices that correspond to standard geodesic lines with non-trivial intersection with $K$. In particular, $F(\Gamma)$ is a union of subcomplexes of form $(F(\Gamma'))_x$ (the notation $(F(\Gamma'))_x$ is defined before Lemma~\ref{isometric embedding}) where $x$ ranges over all vertices of $K$. Each of these subcomplexes is isomorphic to $F(\Gamma')$.

Suppose $G(\Gamma)$ is a RAAG of type II such that it is not prime.
 This ensures that the branches of $\mathcal{P}(\Gamma)$ has certain kind of unexpected symmetry (as remarked before). The end game is how that these symmetries of $\mathcal{P}(\Gamma)$ actually implies $\Gamma$ has a very specific structure explained in the previous paragraph. This is done in three sub-steps. First we use the extra symmetries on $\mathcal{P}(\Gamma)$ to build a wall space structure on $F(\Gamma)$. The dual cube complex to this wall space will be the candidate compact cube complex as in the previous paragraph. Moreover, each vertex of this cube complex gives a subcomplex of $F(\Gamma)$, and $F(\Gamma)$ is a union of these subcomplexes (cf. Lemma~\ref{6.28}). This is done in Section~\ref{subsec:cube}. However, at this point, it is not clear what is the relationship between these subcomplexes. This is analyzed in Section~\ref{subsec:filtration}, where we also study in detail in what pattern are these subcomplexes assembled to give $F(\Gamma)$, and how is this related to the dual cube complex. In the last step, namely Section~\ref{subsec_wall space}, we conclude we indeed construct a prime RAAG such that $G(\Gamma)$ sits inside as a special subgroup.

\subsection{Prime right-angled Artin group}
\label{subsec:prime}
From now on, we assume $G(\Gamma)$ is a centerless RAAG of type II. We also label and orient edges of $X(\Ga)$ in a $G(\Ga)$-invariant way as before (see Section \ref{subsec_notation}). The goal of this subsection is to introduce the notion of prime RAAGs, and prove Theorem~\ref{8.9}.

Let $q:G(\Gamma)\to G(\Gamma')$ be a quasi-isometry and let $q_{\ast}:\mathcal{P}(\Gamma)\to\mathcal{P}(\Gamma')$ be the canonical simplicial isomorphism induced by $q$ (cf. Corollary~\ref{cor:unique}). Pick vertex $v\in\mathcal{P}(\Gamma)$, then $q_{\ast}$ induces a 1-1 correspondence between $v$-branches in $\mathcal{P}(\Gamma)$ and $q_{\ast}(v)$-branches in $\mathcal{P}(\Gamma')$. This correspondence is the starting point to understand the quasi-isometry $q$.

\begin{definition}
\label{8.1}
Let $v\in\mathcal{P}(\Gamma)$ be a vertex. Two $v$-branches $B_{1}$ and $B_{2}$ are \textit{quasi-isometrically indistinguishable} (QII) if there exist a quasi-isometry $f:X(\Gamma)\to X(\Gamma)$ and an induced simplicial isomorphism $f_{\ast}:\mathcal{P}(\Gamma)\to\mathcal{P}(\Gamma)$ such that 
\begin{enumerate}
\item $f_{\ast}$ fixes every vertex in $\mathcal{P}(\Gamma)\setminus(B_{1}\cup B_{2})$.
\item $f_{\ast}(B_{1})=B_{2}$ and $f_{\ast}(B_{2})=B_{1}$.
\end{enumerate}
Such $f$ or $f_{\ast}$ will be called an \textit{elementary permutation}. 
\end{definition}

%Recall that the simplicial isomorphism $f_\ast$ induced by $f$ as in Theorem~\ref{4.8} might not be canonical. However, the $f_\ast$-image of each maximal simplex is canonically determined by $f$. Thus the above notion of QII is well-defined.

\begin{lem}
	\label{lem:QII properties}
	Suppose $\Gamma$ is of type II. Let $f,B_1,B_2$ be as in Definition~\ref{8.1}. Let $q:G(\Gamma)\to G(\Gamma')$ be a quasi-isometry.
\begin{enumerate}
	\item  The map $q_{\ast}$ sends QII $v$-branches to QII $q_{\ast}(v)$-branches.
	\item Let $B_1,B_2,B_3$ be mutually different $v$-branches. If $B_1$ and $B_2$ are QII, $B_2$ and $B_3$ are QII, then $B_1$ and $B_3$ are QII.
\end{enumerate}
\end{lem}

\begin{proof}
For Assertion (1), let $g=q\circ f \circ q^{-1}$, where $q^{-1}$ is a quasi-isometry inverse of $q$. As $q_\ast$ and $ (q^{-1})_\ast$ are inverses of each other by Corollary~\ref{cor:unique}, we know that $g_*$ exchanges $q_*(B_1)$ and $q_*(B_2)$, and fixes all other vertices.

For Assertion (2), let $h:G(\Gamma)\to G(\Gamma)$ be the quasi-isometry witnessing the QII of $B_2$ and $B_3$ with $h_*(B_2)=B_3$. Then $h_*$ sends $\{B_1,B_2\}$ to $\{B_1,B_3\}$. Thus we are done by Assertion (1).
\end{proof}

\begin{lem}
\label{8.2}
Pick a $v$-tier $T$, then for any $v$-branch $B$ such that $B\nsubseteq T$, there exists a $v$-branch $B'\subset T$ such that $B'$ and $B$ are QII. 
\end{lem}

\begin{proof}
Pick standard geodesic $\ell\subset X(\Gamma)$ such that $\Delta(\ell)=v$ and suppose $\pi_{\Delta(\ell)}(B)=x$ and $\pi_{\Delta(\ell)}(T)=x'$ ($\pi_{\Delta(\ell)}$ is the map in Lemma \ref{5.17}). Recall that we have an action $G(\Gamma)\curvearrowright X(\Gamma)$, let $\alpha\in G(\Gamma)$ be the element such that $\alpha$ acts by translation along $\ell$ and $\alpha(x)=x'$. As $B\nsubseteq T$, we know $x\neq x'$, hence $\alpha$ is not the identity element.
Note that $\alpha$ induces a simplicial isomorphism $\alpha_{\ast}:\mathcal{P}(\Gamma)\to\mathcal{P}(\Gamma)$, moreover, $\alpha_{\ast}$ fixes every point in $St(v)$. Define $B'=\alpha_{\ast}(B)$. Let $L$ and $L'$ be the components of $X(\Gamma)\setminus P_{v}$ corresponding to $B$ and $B'$ respectively (Proposition \ref{7.23}). Then $\alpha(L)=L'$. Now we consider the following map $q: X(\Gamma)\to X(\Gamma)$ defined by
\begin{center}
$q(z)$=$\begin{cases}
z & \text{if $z\in X(\Gamma)\setminus(L\cup L')$}\\
\alpha(z) & \text{if $z\in L$}\\
\alpha^{-1}(z) & \text{if $z\in L'$}
\end{cases}$
\end{center}
One readily verifies that $q$ is a quasi-isometry and Proposition \ref{7.23} implies that $q_{\ast}$ satisfies the conditions in Definition \ref{8.1}, so $B$ and $B'$ are QII.
\end{proof}

Let $\bar{v}=\pi(v)$. It follows from Corollary \ref{7.14} (1) that each $v$-branches corresponds to a pair $(C,K)$ where $C$ is a component of  $F(\Gamma)\setminus St(\bar{v})$ and $K$ is a $v$-peripheral subcomplex with support $\partial C$ such that $\partial B=\Delta(K)$. We will denote $C=\Pi(B)$ in such case. As $\Ga$ has type II, Corollary \ref{7.14} implies that for each given height, there exists a unique $v$-branch $B'$ at this height such that $\partial B=\partial B'$ and $\Pi(B')=C$.

\begin{definition}
	\label{def:6.4}
We define two components $C_{1}$ and $C_{2}$ of $F(\Gamma)\setminus St(\bar{v})$ are \textit{quasi-isometrically indistinguishable} (QII) if there exist $v$-branches $B_{1}$ and $B_{2}$ which are QII such that $\Pi(B_{i})=C_{i}$ for $i=1,2$.
\end{definition}

This definition does not depend on the choice of $B_{1}$ and $B_{2}$ in the sense of the following lemma.

\begin{lem}
	\label{lem:well defined}
Suppose $C_{1}$ and $C_{2}$ are QII components of $F(\Gamma)\setminus St(\bar v)$. Then for any pair $B'_{1}$ and $B'_{2}$ such that $\partial B'_{1}=\partial B'_{2}$ and $\Pi(B'_{i})=C_{i}$, we know $B'_1$ and $B'_2$ are QII.
\end{lem}

\begin{proof}
We first look at the case $\partial B'_{1}=\partial B_1$. It follows from Corollary \ref{7.14} that there is $g\in G(\Ga)$ with its axis $\ell_g\subset X(\Ga)$ satisfying $\Delta(\ell_g)=v$ such that $g_*(B_1)=B'_1$. Thus $B_1$ and $B'_1$ are QII by Lemma~\ref{8.2}. Similarly, $B_2$ and $B'_2$ are QII. Thus $B'_1$ and $B'_2$ are QII by Lemma~\ref{lem:QII properties} (2). It remains to treat the case $\partial B'_{1}\neq\partial B_{1}$. Let $K$ and $K'$ be the standard subcomplexes in $P_{v}$ such that $\Delta(K')=\partial B'_{1}$ and $\Delta(K)=\partial B_{1}$. Then their supports satisfy $\Gamma_{K}=\Gamma_{K'}$ by Corollary~\ref{7.14}. Thus there exists $\alpha\in G(\Gamma)$ such that its action on $G(\Gamma)$ satisfying $$\alpha(K)=K'$$ and  $$\alpha(P_{v})=P_{v}.$$ Since $\alpha$ is label-preserving, we know $$\Pi(\alpha_{\ast}(B_{1}))=\Pi(B_{1}).$$ Moreover, $\alpha_{\ast}(v)=v$, so $\alpha_{\ast}(B_{1})$ and $\alpha_{\ast}(B_{2})$ is a pair of QII $v$-branches with $\partial (\alpha_{\ast}(B_{1}))=\partial B'_{1}$. Thus we can conclude the proof by using Lemma~\ref{lem:QII properties} (2) and the previous case.
\end{proof}

\begin{lem}
	\label{lem:transitive}
Let $C_1,C_2,C_3$ be three mutually distinct components of $F(\Gamma)\setminus St(\bar v)$. Suppose $C_1$ and $C_2$ are QII, and $C_2$ and $C_3$ are QII. Then $C_1$ and $C_3$ are QII.
\end{lem}

\begin{proof}
Let $B_1$ and $B_2$ be $v$-branches that are QII with $\Pi(B_i)=C_i$ for $i=1,2$. Let $B'_2$ and $B'_3$ be $v$-branches that are QII with $\Pi(B'_i)=C_i$ for $i=2,3$. As in the proof of Lemma~\ref{lem:well defined}, we can find $\alpha\in G(\Gamma)$ such that $\alpha(P_v)=P_v$, $\alpha_*(v)=v$ and $\alpha_*(B'_2)=B_2$. Thus by Lemma~\ref{lem:QII properties} (1), we can assume without loss of generality that $B_2=B'_2$. Now the lemma follows from Lemma~\ref{lem:QII properties} (2).
\end{proof}

Lemma~\ref{lem:transitive} implies that being QII among components of $F(\Gamma)\setminus St(\bar{v})$ is an equivalence relation, hence we can divide these components into QII equivalent classes
$\{\mathcal{C}_{i}\}_{i=1}^{k}$. We associate $\bar{v}$ with a k-tuple of positive integers $$(n_{1},n_{2},\cdots,n_{k}),$$ where each $n_{i}$ is the number of components of $F(\Gamma)\setminus St(\bar{v})$ in $\mathcal{C}_{i}$. The vertex $\bar v$ is \textit{prime} if $$\gcd(n_{1},n_{2},\cdots,n_{k})=1.$$

It follows from (1) of Corollary \ref{7.14} that if $C_{1}$ and $C_{2}$ are QII, then $\partial C_{1}=\partial C_{2}$, so every QII class $\mathcal{C}_{i}$ has a well-defined boundary, which will be denoted by $\partial \mathcal{C}_{i}$.

\begin{definition}
\label{8.3}
A right-angled Artin group $G(\Gamma)$ is \textit{prime} if and only if $F(\Gamma)$ is of type II and all vertices of $F(\Gamma)$ are prime.
\end{definition}

\begin{definition}
	\label{def:stretch factor}
For vertex $v\in \mathcal{P}(\Gamma)$, we define the stretch factor of $q_*$ at $v$ as follows. Take a $v$-branch $B$, and let $\{B_j\}_{j\in J}$ be the collection of $v$-branches that are QII to $B$. Let $n$ be the number of different elements in $\{\Pi(B_j)\}_{j\in J}$. Let $n'$ be the number of different elements in $\{\Pi(q_*(B_j))\}_{j\in J}$. Then the stretch factor of $q_*$ at $v$ is defined to be $n'/n$.	
\end{definition}

%Let $\bar{v}'=\pi(q_{\ast}(v))$. 
%
%
%take a component $C_1$ of $F(\Gamma)\setminus St(\bar v)$
%Note that take a component $C_1$ of $F(\Gamma)\setminus St(\bar v)$, then 
%
%
%By Lemma~\ref{lem:QII properties} and Lemma~\ref{lem:well defined}, $q_{\ast}$ induces a map from QII classes of $F(\Gamma)\setminus St(\bar{v})$ to QII classes of $F(\Gamma')\setminus St(\bar{v}')$. This map
%
%For $1\le i\le k$, let $\mathcal{C}_{i}'$ be the class corresponding to $\mathcal{C}_{i}$ under $q_{\ast}$, and let $n'_{i}$ be the number of components of $F(\Gamma')\setminus St(\bar{v}')$ in $\mathcal{C}_{i}'$. Note that $(n'_{1},n'_{2},\cdots,n'_{k})$ is the tuple associated with $\bar{v}'$.

\begin{lem}
\label{8.4}
The following are true.
\begin{enumerate}
	\item The definition of stretch factor does not depend on the choice of $B$.
	\item  If $G(\Gamma)$ is prime, then the stretch factor is an integer.
\end{enumerate}
\end{lem}

%Such $r$ will be called the \textit{stretch factor} of $q_{\ast}$ at $v$.
\begin{proof}
Let $\ell\subset X(\Gamma)$ and $\ell'\subset X(\Gamma')$ be standard geodesics such that $\Delta(\ell)=v$ and $\Delta(\ell')=q_{\ast}(v)$. Recall that the vertex set $v(\ell)$ has a natural ordering induced from the orientation of edges in $X(\Gamma)$. We identify $v(\ell)$ with $\Bbb Z$ in an order-preserving way. Let $B$ and $\{B_{j}\}_{j\in J}$ be as in Definition~\ref{def:stretch factor}. Suppose  $$\Pi(B)\in\mathcal{C}_{i}.$$ Then by (1) and (2) of Corollary \ref{7.14}, there are exactly $n_{i}$ elements of $\{B_{j}\}_{j\in J}$ in a given $v$-tier, where $n_i$ is the cardinality of $\mathcal C_i$. Pick a total order on elements in $\mathcal{C}_{i}$ and define a total order on $J$ by $j_{1}<j_{2}$ if and only if $$\pi_{v}(B_{j_{1}})<\pi_{v}(B_{j_{2}})\ \ (\pi_{v}\ \mathrm{is\ the\ map\ in\ Lemma\ \ref{5.17}})$$
 or 
 $$\pi_{v}(B_{j_{1}})=\pi_{v}(B_{j_{2}})\ \ \mathrm{and}\ \  \Pi(B_{j_{1}})<\Pi(B_{j_{2}}).$$ We identify $J$ with $\Bbb Z$ in an order-preserving way, then there is a natural map $g_{i}:J\to v(\ell)$ induced by $\pi_{v}$. Note that $$g_{i}(a)=\lfloor a/n_{i}\rfloor$$ up to translation.

Let $\{B'_{k}\}_{k\in K}$ be the collection of $q_{\ast}(v)$-branches such that $B'_{k}$ and $q_{\ast}(B)$ are QII. Then $$\Pi(\{B'_{k}\}_{k\in K})=\mathcal{C}'_{i}$$ for some CII class $\mathcal{C}'_{i}$ of $$F(\Gamma')\setminus St(\bar v')$$ where $\bar{v}'=\pi(q_{\ast}(v))$. Note that $q_{\ast}$ induces a bijection $f_{i}:J\to K$. We identify $v(\ell')$ with $\Bbb Z$ and $K$ with $\Bbb Z$ in the same way as before and let $$g'_{i}: K\to v(\ell')$$ be the natural map given by $$g'_{i}(a)=\lfloor a/n'_{i}\rfloor$$ where $n'_i$ is the cardinality of $\mathcal C'_i$.

We define another map $h_{i}:v(\ell)\to v(\ell')$ as follows. For $x\in v(\ell)$, pick a $B_{j}$ such that $\pi_{v}(B_{j})=x$ and define $$h_{i}(x)=\pi_{\Delta(\ell')}(q_{\ast}(B_{j})).$$ Up to bounded distance, the definition of $h_{i}$ is independent of the choice of $B_{j}$. We claim $h_{i}$ is a quasi-isometry. Pick $B_{j_{1}}$, $B_{j_{2}}$ in $\{B_{j}\}_{j\in J}$. For $m=1,2$, let $L_{j_{m}}$ be the subset of $X(\Gamma)$ as in Proposition \ref{7.23} such that $\Delta(\bar{L}_{j_{m}})=B_{j_{m}}$. Then $$d(L_{j_{1}},L_{j_{2}})=d(\pi_{v}(B_{j_{1}}),\pi_{v}(B_{j_{2}})).$$ Now it follows from Assertion (1) and (3) of Proposition \ref{7.23} that $h_{i}$ is a quasi-isometry. Note that the following diagram commutes up to bounded distance:
\begin{center}
$\begin{CD}
J                            @>f_{i}>>       K\\
@VVg_{i}V                                                              @VVg'_{i}V\\
v(\ell)          @>h_{i}>>        v(\ell')
\end{CD}$
\end{center}
thus $f_{i}$ is a bijective quasi-isometry from $\Bbb Z$ to $\Bbb Z$, hence $f_{i}$ is bounded distance from an isometry and $$h_{i}(x)=(n_{i}/n'_{i})x+b$$ up to bounded distant ($b$ is some constant). Now we pick a different QII class of $v$-branches which gives the QII class $\mathcal{C}_{i'}$ of $F(\Gamma)\setminus St(\bar v)$ and define $h_{i'}$ in similar way, then $$h_{i}=h_{i'}$$ up to bounded distant, but we also have $$h_{i'}(x)=(n_{i'}/n'_{i'})x+b'$$ up to bounded distance, so $n_{i}/n'_{i}=n_{i'}/n'_{i'}$.

To see (2), note that the previous discussion implies that the multiplication of the stretch factor with each $n_i$ (for $1\le i\le k$) is an integer. Thus the stretch factor must be an integer when $\gcd(n_1,\ldots, n_k)=1$.
\end{proof} 

In the rest of this subsection, we prove the following.
\begin{thm}
	\label{8.9}
	If $G(\Gamma)$ and $G(\Gamma')$ are prime right-angled Artin groups, then they are quasi-isometric if and only if they are isomorphic.
\end{thm}

Suppose $G(\Gamma)$ and $G(\Gamma')$ are prime right-angled Artin groups. Let $q:G(\Gamma)\to G(\Gamma')$ be a quasi-isometry and $q_{\ast}:\mathcal{P}(\Gamma)\to\mathcal{P}(\Gamma')$ be the induced simplicial isomorphism. Pick vertex $x\in X(\Gamma)$. Let $K=(F(\Gamma))_{x}$ and $K'=q_{\ast}(K)$. Suppose $\{v_{i}\}_{i=1}^{n}$ is the collection of vertices in $\mathcal{P}(\Gamma)$ such that $K\setminus St(v_{i})$ is disconnected, then $v_{i}\in K$ for all $i$ by Lemma \ref{7.8}. Let $v'_{i}=q_{\ast}(v_i)$. Then $\{v'_{i}\}_{i=1}^{n}$ are exactly the vertices in $\P(\Ga')$ such that $K'\setminus St(v'_{i})$ is disconnected.

Recall that $\pi:\mathcal{P}(\Gamma')\to F(\Gamma')$ is the canonical projection.
\begin{lem}
	\label{8.7}
	If for any vertex $v'\in\mathcal{P}(\Gamma)$, all vertices in $K'\setminus St(v')$ are in the same $v'$-tier, then $\pi|_{K'}$ is injective. Moreover, $\cap_{v\in K'} P_{v}\neq\emptyset$.
\end{lem}

\begin{proof}
	Since $\pi$ maps simplexes to simplexes of the same dimension, it suffices to show there do not exist vertex $w_{1},w_{2}\in K'$ such that $\pi(w_{1})=\pi(w_{2})$. Suppose the contrary is true. Then $$P_{w_{1}}\cap P_{w_{2}}=\emptyset.$$ Let $h$ be a hyperplane separating $P_{w_{1}}$ and $P_{w_{2}}$, and let $\ell$ be a standard geodesic dual to $h$. Then $$\pi_{\Delta(\ell)}(w_{1})\neq\pi_{\Delta(\ell)}(w_{2})$$ ($\pi_{\Delta(\ell)}$ is the map in Lemma \ref{5.17}), hence $w_{1}$ and $w_{2}$ are in different $\Delta(\ell)$-tier, which yields a contradiction. The second statement follows from the Lemma \ref{2.1} and the previous argument.
\end{proof}

The following is the main ingredient for Theorem~\ref{8.9}.

\begin{lem}
	\label{lem:main}
Under the assumption of Theorem~\ref{8.9}, it is possible to modify $q_{\ast}$ by post-composing $q_*$ with finitely many elementary permutations such that the assumption of Lemma \ref{8.7} is true.
\end{lem}

Assuming Lemma~\ref{lem:main}, we can finish the proof of Theorem~\ref{8.9} as follows. Lemma~\ref{8.7} and Lemma~\ref{lem:main} gives a simplicial embedding $F(\Gamma)\to F(\Gamma')$.
By considering the quasi-isometry inverse, we have a simplicial embedding $$s':F(\Gamma')\to F(\Gamma),$$ thus $F(\Gamma)$ and $F(\Gamma')$ have the same number of vertices. Note that $s(F(\Gamma))$ is a full subcomplex of $F(\Gamma')$, so $s$ is actually a simplicial isomorphism and Theorem~\ref{8.9} follows. In the rest of this subsection, we prove Lemma~\ref{lem:main}.

\begin{remark}[Informal discussion of the proof of Lemma~\ref{lem:main}.]
We look at some simple special cases of Lemma~\ref{lem:main} before we discuss the proof in full detail. 

The simplest case to consider is that $\{v'_i\}_{i=1}^n$ has only one element. Let $\{B_j\}_{j=1}^k$ be the collection of $v_1$-branches that have non-trivial intersection with $K$. 
Then $\{B_j\}_{j=1}^k$ are in the same $v_1$-tier (as they have the same height as $x$ with respect to $v_1$). By Corollary~\ref{7.13}, $\{B_j\}_{j=1}^k$ are in 1-1 correspondence with connected components of $$K\setminus St(v_1)\cong F(\Gamma)\setminus St(\bar v_1),$$ and the correspondence is giving by sending $B_j$ to $\Pi(B_j)$.
Let $B'_j=q_*(B_j)$. 
As $G(\Gamma)$ is prime, the stretch factor of $q_*$ at $v_1$ is an integer $\ge 1$. Thus there exists a quasi-isometry $g$ such that $g_*:\mathcal{P}(\Gamma)\to\mathcal{P}(\Gamma)$ is a composition of finitely many elementary permutations of $v'_1$-branches with the property that $\{g_*(B'_j)\}_{j=1}^k$ are in the same $v'_1$-tier. Thus Lemma~\ref{lem:main} follows. We refer the procedure in this case as the basic move at $v'_1$.

The next case to consider is that $\{v'_i\}_{i=1}^n$ has two elements, and their distance is $\ge 2$. One naturally want to first preform the basic move at $v'_1$, and then perform the basic move at $v'_2$. The second step will maintain the outcome of the first step by Lemma~\ref{8.5} below, as long as in the second step we send the $v'_2$-branch that contains $v'_1$ to itself, but this is always possible to arrange.

A slightly more complicated case to consider is that $\{v'_i\}_{i=1}^n$ has mutual distance $\ge 2$. We need to pick the correct order to perform basic moves, namely when we perform the basic move at $v'_i$, we want to send the $v'_i$-branch that contains $\{v'_1,\ldots,v'_{i-1}\}$ to itself to maintain the outcome of previous moves. This is possible only if we can order $\{v'_i\}_{i=1}^n$ such that $\{v'_1,\ldots,v'_{i-1}\}$ are in the same $v'_i$-branch, which motivates the notion of \emph{tight subset} of $\{v'_i\}_{i=1}^n$ as in the proof of Lemma~\ref{lem:tight}.

Another case to consider is that $\{v'_i\}_{i=1}^n$ has two elements, and their distance $=1$. Again the key point is to make sure that doing the basic move at $v'_2$ maintains the outcome of the basic move at $v'_1$. Take a $v'_1$-branch $B$. If $v'_2\in \partial B$, then doing basic move at $v'_2$ will send $B$ to itself, as $B$ contains a vertex in $St(v'_2)$, which is fixed pointwise under the basic move at $v'_2$. Thus we need to handle the case when $v'_2\notin \partial B$, which leads to the notion of $v'_2$-non-crossing as in the proof of Lemma~\ref{8.6}. This lemma gives a way to correct the situation if doing basic move at $v'_2$ destroys the outcome of basic move at $v'_1$.
\end{remark}

We start the proof of Lemma~\ref{lem:main} with two simple observations.
\begin{lem}
	\label{8.5}
	Assume $d(v_{1},v_{2})\ge 2$ and let $B$ be any $v_{1}$-branch such that $v_{2}\notin B$. Then $B$ and $v_{1}$ are in the same $v_{2}$-branch, in particular, all such $B$ are in the same $v_{2}$-branch. 
\end{lem}

\begin{proof}
	Note that $B\cap St(v_{2})=\emptyset$ (otherwise $v_{2}\in B$). We also have $$\partial B\nsubseteq lk(v_{1})\cap lk(v_{2})=St(v_{1})\cap St(v_{2}),$$ otherwise $B$ and $v_1$ are in different connected components of $\mathcal{P}(\Gamma)\setminus(lk(v_1)\cap lk(v_2))$, which contradicts Corollary \ref{7.14} (5). As $$\partial B\subset St(v_1),$$ we know there is a vertex $w'\in\partial B$ with $w'\notin St(v_{2})$, hence $B$ can be connected to $v_{1}$ via $w'$ outside $St(v_{2})$.
\end{proof}

From now on, we will write $v_{j}|_{v_{i}}v_{k}$ if $v_{j}$ and $v_{k}$ are in different $v_{i}$-branches and write $v_{j}v_{k}|_{v_{i}}$ if $v_{j}$ and $v_{k}$ are in the same $v_{i}$-branch.

\begin{lem}
	\label{8.8}
	Suppose $F(\Gamma)$ is of type II. Let $v_{1},v_{2},v_{3}$ be vertices in $\mathcal{P}(\Gamma)$. If $v_{1}|_{v_{2}}v_{3}$, then $v_{1}v_{2}|_{v_{3}}$ and $v_{3}v_{2}|_{v_{1}}$.
\end{lem}

\begin{proof}
	Let $B$ be the $v_{2}$-branch that contains $v_{3}$. Since $\mathcal{P}(\Gamma)\setminus (lk(v_{1})\cap lk(v_{2}))$ is connected, we have $$\partial B\nsubseteq lk(v_{1})\cap lk(v_{2})=St(v_{1})\cap St(v_{2}).$$ Then there exists vertex $w\in\partial B$ with $w\notin St(v_{1})$, which implies that $v_{2}$ and $v_{3}$ can be connected via $w$ outside $St(v_{1})$. 
\end{proof}

%Since $q_{\ast}$ does not necessarily map $v$-tier to $q_{\ast}(v)$-tier, we want to modify $q_{\ast}$ by post-composing $q_{\ast}$ with an appropriate permutation of $q_{\ast}(v)$-branches such that $v$-tier is send to $q_{\ast}(v)$-tier. It is easy to do this for a single vertex $v$ when $G(\Ga)$ is prime by Lemma \ref{8.4}, but in general we need to deal with more than one vertices simultaneously, so it is necessary to figure out how to deal with each vertex independently.

%In order to do this, we introduce an auxiliary order. 

Let $E_n=\{v'_j\}_{j=1}^n$. For $1\le i\le n$, define $E_i=\{v'_j\}_{j=1}^i$ and define $E_0=\emptyset$.

\begin{lem}
	\label{lem:tight}
It is possible to order the elements $\{v'_i\}_{i=1}^n$ of $E_n$ such that for each $1\le i\le n$ and any $v'\in E_n\setminus E_i$, we have all elements of $E_i\setminus St(v')$ contained in the same $v'$-branch.
\end{lem}
 
\begin{proof}
 Pick $E\subset\{v'_{i}\}_{i=1}^{n}$ and denote $\{v'_{i}\}_{i=1}^{n}\setminus E$ by $E^{c}$. We say $E$ is \textit{tight} if for any $v'_{i}\in E^{c}$, $E\setminus St(v'_{i})$ is inside a $v'_{i}$-branch. Pick $v'_{i},v'_{j}\in E^{c}$, we define $v'_{i}<_{E}v'_{j}$ if and only if there exists $v'_{k}\in E$ such that $v'_{j}$ and $v'_{k}$ are in different $v'_{i}$-branches. 
 
 We claim that if $E$ is tight, then $\le_{E}$ is a partial order on $E^{c}$. Now we prove the claim. If $v'_{i}<_{E} v'_{j}$ and $v'_{j}<_{E} v'_{i}$, then there exist $v'_{k_{1}}$ and $v'_{k_{2}}$ in $E$ such that $$v'_{j}|_{v'_{i}} v'_{k_{1}}\ \mathrm{and}\ v'_{i}|_{v'_{j}}v'_{k_{2}}.$$ By Lemma \ref{8.8}, we have $$v'_{k_{2}}v'_{j}|_{v'_{i}},$$ so $$v'_{k_{2}}|_{v'_{i}}v'_{k_{1}},$$ which contradicts the tightness of $E$. Thus the relation $\le_{E}$ is antisymmetric. It remains to check the transitivity. Suppose $v'_{i}<_{E} v'_{j}$ and $v'_{j}<_{E} v'_{k}$ for $v'_{i},v'_{j},v'_{k}\in E^{c}$. Then there exist $v'_{\ell}$ and $v'_{m}$ in $E$ such that $$v'_{\ell}|_{v'_{i}}v'_{j}\ \mathrm{and}\ v'_{m}|_{v'_{j}}v'_{k}.$$ Since $$v'_{\ell}\notin St(v'_{j})\ \mathrm{and}\ v'_{m}\notin St(v'_{j}),$$ then $v'_{m}v'_{\ell}|_{v'_{j}}$. We also have $v'_{i}v'_{\ell}|_{v'_{j}}$ by Lemma \ref{8.8}, so $v'_{m}v'_{i}|_{v'_{j}}$. This, together with $v'_{m}|_{v'_{j}}v'_{k}$ imply $v'_{i}|_{v'_{j}}v'_{k}$, hence we have $$v'_{k}v'_{j}|_{v'_{i}}.$$ However, we also know $v'_{\ell}|_{v'_{i}}v'_{j}$, so $v'_{\ell}|_{v'_{i}}v'_{k}$ and $v'_{i}<_{E}v'_{k}$. 

If $E$ is tight, let $v'_{i}\in E^{c}$ be a minimal element in $E^{c}$ with respect to $\le_{E}$. Then $E\cup\{v'_{i}\}$ is also tight. Let $E_{1}=\{v'_{1}\}$. $E_{1}$ is clearly tight, so it is possible to add a vertex in $E_{1}^{c}$ to $E_{1}$ to obtain a tight set $E_{2}$. By repeating this process for $n-1$ times, we obtain a filtration $E_{1}\subsetneq E_{2}\subsetneq\cdots\subsetneq E_{n-1}\subsetneq E_{n}=\{v'_{i}\}_{i=1}^{n}$ such that the requirements of the lemma are met.
\end{proof}

Suppose we have already obtained a quasi-isometry $q_{\ast}$ such that for every vertex $v'\in E_{i}$, vertices of $K'\setminus St(v')$ are in the same $v'$-tier. Suppose $v'_{i+1}=E_{i+1}\setminus E_{i}$ and let $B'$ be the $v'_{i+1}$-branch that contains all points in $E_{i}\setminus St(v'_{i+1})$ (if $E_{i}\setminus St(v'_{i+1})=\emptyset$, we pick an arbitrary $v'_{i+1}$-branch). Let $\{B_{j}\}_{j=1}^{k}$ be the collection of $v_{i+1}$-branches that have non-trivial intersection with $K$. Then $\{B_j\}_{j=1}^k$ are in the same $v_{i+1}$-tier (as they have the same height as $x$ with respect to $v_{i+1}$). By Corollary~\ref{7.13}, $\{B_j\}_{j=1}^k$ are in 1-1 correspondence with connected components of $K\setminus St(v_{i+1})\cong F(\Gamma)\setminus St(\bar v_{i+1})$, and the correspondence is giving by sending $B_j$ to $\Pi(B_j)$. Let $B'_{j}=q_{\ast}(B_{j})$. Since both $G(\Gamma)$ and $G(\Gamma')$ are prime, the stretch factor of $q_{\ast}$ at $v_{i+1}$ is 1, then there exists $g_{\ast}:\mathcal{P}(\Gamma')\to\mathcal{P}(\Gamma')$ such that
\begin{enumerate}
	\item $g_{\ast}$ is a composition of finitely many elementary permutations of $v'_{i+1}$-branches, hence $g_{\ast}$ fixes every point in $St(v'_{i+1})$;
	\item $g_{\ast}$ fixes every point in $B'$;
	\item $\{g_{\ast}(B'_{j})\}_{j=1}^{k}$ are in the same $v'_{i+1}$-tier.
\end{enumerate}
By Lemma \ref{8.6} below, we can assume in additional that 
\begin{enumerate}[resume]
	\item for any $v'\in E_{i}\cap St(v'_{i+1})$ and any $v'$-branch $D$ such that $D\cap K'\neq\emptyset$, $g_{\ast}(D)$ and $D$ are in the same $v'$-tier.
\end{enumerate}

By (1) and (2), $g_{\ast}$ fixes every point in $E_{i+1}$. We claim that vertices of $g_{\ast}(K')\setminus St(v')$ are in the same $v'$-tier for any $v'\in g_{\ast}(E_{i+1})$. There are three cases to consider as follows.
\begin{itemize}
	\item The case $v'=v'_{i+1}$ follows from property (3) of $g_*$.
	\item The case $v'\in E_{i}\cap St(v'_{i+1})$ follows from the inductive assumption and property (4) of $g_*$.
	\item Let $v'\in E_{i}\setminus St(v'_{i+1})$ and $D$ be a $v'$-branch. If $v'_{i+1}\notin D$, then $g_{\ast}(D)=D$ by (2) and Lemma \ref{8.5}; if $v'_{i+1}\in D$, $g_{\ast}(D)=D$ is still true since $g_{\ast}$ fixes $v'$ and $v'_{i+1}$. Thus $g_{\ast}$ does not permute the $v'$-branches and the claim follows.
\end{itemize}

Let $q'_{\ast}=g_{\ast}\circ q_{\ast}$, $K''=g_{\ast}(K')=q'_{\ast}(K)$, $E'_{i}=g_{\ast}(E_{i})$ and $v''_{i}=g_{\ast}(v'_{i})$. Then $\{v''_{i}\}_{i=1}^{n}$ are exactly the vertices in $\mathcal{P}(\Gamma')$ such that $K''\setminus St(v''_{i})$ is disconnected. Note that $$E'_{1}\subsetneq E'_{2}\subsetneq\cdots\subsetneq E'_{n}$$ still satisfies Lemma~\ref{lem:tight}. Moreover, vertices of $K''\setminus St(v'')$ are in the same $v''$-tier for any $v''\in E'_{i+1}$. So we can repeat the previous process to deal with $E'_{i+2}$. 

After finitely many steps, we can assume for every point $v'\in\{v'_{i}\}_{i=1}^{n}$, vertices of $K'\setminus St(v')$ are in the same $v'$-tier, thus $K'$ satisfies the assumption of Lemma \ref{8.7} and $\pi\circ q_{\ast}$ induces a simplicial embedding $s:F(\Gamma)\to F(\Gamma')$. This finishes the proof of Lemma~\ref{8.7} modulo Lemma~\ref{8.6}.

\begin{lem}
\label{8.6}
Take $F(\Gamma)$ which is of type II. Let $w\in\mathcal{P}(\Gamma)$ be a vertex. Let $\{B_i\}_{i=1}^{n}$ be a collection such that each $B_i$ is a $v_{i}$-branch for some vertex $v_i\in\mathcal{P}(\Gamma)$ satisfying $d(v_{i},w)=1$. We assume $B_i\neq B_j$ for $i\neq j$ (however $v_{i}=v_{j}$ is allowed for $i\neq j$).
Let $q:X(\Gamma)\to X(\Gamma)$ be a quasi-isometry such that $q_{\ast}$ fixes every point in $St(w)$. Then there exists a quasi-isometry $q':X(\Gamma)\to X(\Gamma)$ such that $q'_{\ast}$ satisfies:
\begin{enumerate}
\item $q'_{\ast}$ fixes every point in $St(w)$.
\item $q'_{\ast}(B)=q_{\ast}(B)$ for any $w$-branch $B$.
\item $B_{i}$ and $q'_{\ast}(B_{i})$ are in the same $v_{i}$-tier.
\item If $q_{\ast}$ fixes every point in a $w$-branch $B$, then $q'_{\ast}$ also fixes every point in $B$.
\end{enumerate}
\end{lem}

\begin{proof}
We start by introducing an auxiliary notion. Take vertices $w,v\in\mathcal{P}(\Gamma)$ and a $v$-peripheral complex $K\subset \P(\Ga)$. The pair $(v,K)$ is $w$-\textit{non-crossing} if $d(v,w)=1$ and $w\notin K$. In this case, $B\cap St(w)=\emptyset$ for any $v$-branch $B$ such that $\partial B=K$. Moreover, for any other $v$-branch $B'$ with $\partial B'=K$, $B'$ and $B$ are in the same $w$-branch. To see this, note that $$K=\partial B\nsubseteq lk(v)\cap lk(w),$$ otherwise $B$ and $v$ will be in different connected components of $$\mathcal{P}(\Gamma)\setminus(lk(v)\cap lk(w)),$$ which contradicts Corollary \ref{7.14} (5). On the other hand, $K\subset lk(v)$. So $K$ contains a vertex $w'\in lk(v)\setminus St(w)$ such that $B'$ can be connected with $B$ outside $St(w)$ via $w'$. We refer to Remark~\ref{rmk:non-crossing} for a comment on the naming of ``$w$-\textit{non-crossing}''.

If $B_i$ and $q_*(B_i)$ are not in the same $v_i$-tier, we wish to post-composing $q_*$ with elementary permutations to arrange Lemma~\ref{8.6} (3). Suppose in step 1 we already arranged $B_1$ and $q'_*(B_1)$ to be in the same $v_1$-tier. Then in step 2 when arranging the position of $q'_*(B_2)$ we want to maintain the outcome of step 1. One ideal situation is that all the $v_2$-branches upon which we want to perform elementary permutations are contained in the same $v_1$-branch. Then whatever happens in step 2 only takes place within one particular $v_1$-branch and each $v_1$-branch is mapped to itself. This leads us to define the following binary relation, which will guide us on the order of treating elements in $\{B_i\}_{i=1}^n$.

We define a binary relation $\le$ on the set of $w$-non-crossing pairs by $(v_{1},K_{1})\le(v_{2},K_{2})$ if there exist $v_{1}$-branch $B_{1}$ with $\partial B_{1}=K_{1}$ and $v_{2}$-branch $B_{2}$ with $\partial B_{2}=K_{2}$ such that $B_{1}\subset B_{2}$. If $(v_{1},K_{1})<(v_{2},K_{2})$, then $d(v_{1},v_{2})=1$. To see this, note that if $v_{1}=v_{2}$, we must have $B_{1}=B_{2}$ and $K_{1}=K_{2}$. Suppose $d(v_{1},v_{2})=2$. Since $v_{2}\notin B_{1}$, $B_{1}$ must belong to the $v_{2}$-branch that contains $v_{1}$ by Lemma \ref{8.5}. Hence $v_1\in B_2$ and $w\in\partial B_{2}=K_{2}$, which yields a contradiction.

Now we show the relation $<$ is a partial order.
Suppose $(v_{1},K_{1})\le (v_{2},K_{2})$. Since $B_1\subset B_2$, we know $B_1\cap St(v_2)=\emptyset$, thus $$(B_1\cup \partial B_1)\cap \{v_2\}=\emptyset,$$ in particular $$v_2\notin \partial B_1=K_1.$$
Thus $(v_1,K_1)$ is $v_2$-non-crossing and we deduce as before that $B'_{1}\subset B_{2}$ for any $v_{1}$-branch $B'_{1}$ with $\partial B'_{1}=K_{1}$. Thus the relation $\le$ is transitive. If $(v_{1},K_{1})\le (v_{2},K_{2})$, $(v_{2},K_{2})\le (v_{1},K_{1})$ and $(v_{1},K_{1})\neq(v_{2},K_{2})$, then it follows from previous discussion that all $v_{1}$-branches with boundary $=K_{1}$ stay inside one particular $v_2$-branch, and all $v_{2}$-branches with boundary $=K_{2}$ stay inside one particular $v_1$-branch. This implies all $v_{1}$-branches with boundary $=K_{1}$ stay inside one particular $v_1$-branch, which is impossible.
So $\le$ is antisymmetric.

Now we begin to arrange all the requirements of Lemma~\ref{8.6}.
We only need to consider the case when $B_{i}\subset \P(\Gamma)\setminus St(w)$ for all $i$, otherwise $B_{i}$ will contain a vertex fixed by $q_{\ast}$ and (3) is automatic. Let $K_{i}=\partial B_{i}$. Then $(v_{i},K_{i})$ is a $w$-non-crossing pair. Suppose $(v_{1},K_{1})$ is a maximal element in $\{(v_{i},K_{i})\}_{i=1}^{n}$ with respect to the order defined above and suppose $(v_{1},K_{1})=(v_{i},K_{i})$ if and only if $1\le i\le m$. Let $K'_{1}=q_{\ast}(K_{1})$ and let $\{A_{i}\}_{i\in\Bbb Z}$ (or $\{A'_{i}\}_{i\in\Bbb Z}$) be the collection of $v_{1}$-branches with boundary $K_{1}$ (or $K'_{1}$). Then $q_{\ast}$ induces a bijection between $\{A_{i}\}_{i\in\Bbb Z}$ and $\{A'_{i}\}_{i\in\Bbb Z}$. Since $q_{\ast}$ fixes $v_{1}$, the stretch factor of $q_{\ast}$ at $v_{1}$ is 1 by Lemma \ref{8.4}, so we can post-compose $q_{\ast}$ with a finite sequence of elementary permutations of elements in $\{A'_{i}\}_{i\in\Bbb Z}$ such that (3) is true for $1\le i\le m$. Note that $(v_{1},K'_{1})$ is also $w$-non-crossing, so $\{A'_{i}\}_{i\in\Bbb Z}$ are in the same $w$-branch and each of the elementary permutations we post-compose before is supported on this particular $w$-branch (and is identity outside this $w$-branch), hence (1) and (2) still hold.

Pick $i_{0}>m$ and let $D_{1}$ and $D_{2}$ be two QII $v_{i_{0}}$-branches such that $$\partial D_{1}=\partial D_{2}=q_{\ast}(K_{i_{0}}).$$ Let $f_{\ast}$ be an elementary permutation of $D_{1}$ and $D_{2}$. We claim $$f_{\ast}(q_{\ast}(B_{i}))=q_{\ast}(B_{i})$$ for $1\le i\le m$, then the lemma follows by induction on the number of $B_{i}$. To see the claim, note that $$(v_{1},K'_{1})\nleq(v_{i_{0}},q_{\ast}(K_{i_{0}}))$$ (since $(v_{1},K_{1})\nleq(v_{i_{0}},K_{i_{0}})$). Then for any $v_{1}$-branch $E$ such that $\partial E=K'_{1}$, we know $E$ contains a vertex $u\in \mathcal{P}(\Gamma)\setminus(D_{1}\cup D_{2})$, otherwise we would have $E\subset D_{1}$ or $E\subset D_{2}$. Recall that $f_{\ast}(u)=u$, so $f_{\ast}(E)=E$, in particular $f_{\ast}(q_{\ast}(B_{i}))=q_{\ast}(B_{i})$ for $1\le i\le m$.

Property (4) is true since we only need to consider those $B_{i}$'s that are not contained in $w$-branches which are fixed by $q_{\ast}$ pointwise. 
\end{proof}

\begin{remark}
	\label{rmk:non-crossing}
The naming of ``$w$-non-crossing'' comes from the fact that if $(v,K)$ is $w$-non-crossing, then for any connected component $L$ of $X(\Gamma)\setminus P_v$ such that $\Delta(\partial L)=K$ (cf. Section~\ref{sec:correspondence}), $L$ does not cross any hyperplanes of $X(\Gamma)$ whose dual edge can extend to a standard line $\ell$ with $\Delta(\ell)=w$. In other words, all components $L$ of $X(\Gamma)\setminus P_v$ with $\Delta(\partial L)=K$ are in the same height with respect to $\ell$. Though we will not need this fact for later part of the paper.
\end{remark}

\begin{remark}
	\label{rmk:support}
We record a direct consequence of the construction of Lemma~\ref{8.6} that will be used later.	
The map $q'$ in Lemma~\ref{8.6} is obtained by postcomposing $q$ with a finite sequence of elementary permutations such that each of these elementary permutation induces identity map on $\mathcal{P}(\Gamma)\setminus q_*(\mathcal{B})$ where $\mathcal{B}$ is the union of all $w$-branches that contains at least one element from $\{B_i\}_{i=1}^{n}$.
\end{remark}

\subsection{Prime partition, sub-tiers, prime factors and dual cube complexes}
\label{subsec:cube} Given right-angled Artin group $G(\Gamma)$ of type II (not necessarily prime), our goal in the next three subsections is to find a prime right-angled Artin group $G(\Gamma')$ which is quasi-isometric to $G(\Gamma)$. Such $G(\Gamma')$, if exists, must be unique by Theorem \ref{8.9}. In this subsection, we introduce a wall space structure on $F(\Gamma)$ and prove several basic properties of this wall space for later use.

Pick vertex $\bar{v}\in F(\Gamma)$, let $\{\mathcal{C}_{i}\}_{i=1}^{k}$ be the collection of QII classes in $F(\Gamma)\setminus St(\bar{v})$ and let $(n_{1},n_{2},\cdots,n_{k})$ be the associated tuple. Let $\{C_{ij}\}_{j=1}^{n_{i}}$ be the components in $\mathcal{C}_{i}$ and let $$d=\gcd(n_{1},n_{2},\cdots,n_{k}).$$ For each $i$, we choose a map $$f_{i}:\{C_{ij}\}_{j=1}^{n_{i}}\to \{1,2,\cdots,d\}$$ such that for each $1\le m\le d$, there are $n_{i}/d$ elements in $f^{-1}_{i}(m)$. For $1\le m\le d$, let $$\mathfrak{C}_{m}=\cup_{i=1}^{k}f^{-1}_{i}(m).$$ This partition of components of $F(\Gamma)\setminus St(\bar{v})$ into $\{\mathfrak{C}_{m}\}_{m=1}^d$ is called a \textit{prime partition at $\bar{v}$}. Each $\mathfrak{C}_{m}$ is called a \textit{prime factor} at $\bar{v}$. The prime partition comes together with an order, namely, we define $\mathcal{C}_{i}\le \mathcal{C}_{j}$ if $i\le j$. Note that the prime partition is trivial if $\bar{v}$ is prime. Now we fix a prime partition for every non-prime vertex in $F(\Gamma)$. 

\begin{remark}
\label{equal number of prime factors}
Let $\alpha:F(\Ga)\to F(\Ga)$ be a simplicial automorphism. By consider the group automorphism of $G(\Ga)$ induced by $\alpha$, we deduce that the number of prime factors at $\bar{v}$ and the number of prime factors at $\alpha(\bar{v})$ are the same. However, $\alpha$ may not map prime factors at $\bar{v}$ to prime factors at $\alpha(\bar{v})$.
\end{remark}

Let $v\in\mathcal{P}(\Gamma)$ be a vertex such that $\pi(v)=\bar{v}$ and let $T$ be a $v$-tier. Recall that we have a map $\Pi$ which maps $v$-branches to components of $F(\Gamma)\setminus St(\bar{v})$. This would give rise to a partition $$\{\Pi^{-1}(\mathfrak{C}_{m})\cap T\}_{m=1}^{d}$$ of $v$-branches in $T$. Each element in the partition is called a \textit{$v$-sub-tier}. 

The following lemma follows directly from definition.
\begin{lem}
\label{correspondence}
Pick vertex $x\in X(\Ga)$ and let $i_x:F(\Ga)\to \P(\Ga)$ be the natural embedding. Then for vertices $\bar{u},\bar{v},\bar{w}\in F(\Ga)$, $\bar{u}$ and $\bar{v}$ are in different prime factors at $\bar{w}$ if and only if $i_x(\bar{u})$ and $i_x(\bar{v})$ are in different $i_x(\bar{w})$-sub-tiers.
\end{lem}

\begin{lem}
\label{premuting sub-tiers}
Let $S_{1}$ and $S_{2}$ be two $v$-sub-tiers. Then there exists a quasi-isometry $q:X(\Gamma)\to X(\Gamma)$ such that the induces simplicial isomorphism $q_{\ast}:\mathcal{P}(\Gamma)\to\mathcal{P}(\Gamma)$ satisfies
\begin{enumerate}
\item $q_{\ast}$ fixes every vertex in $\mathcal{P}(\Gamma)\setminus(S_{1}\cup S_{2})$.
\item $q_{\ast}(S_{1})=S_{2}$ and $q_{\ast}(S_{2})=S_{1}$.
\item For every $v$-branch $B\subset S_{1}$, $q_{\ast}(B)$ and $B$ are QII.
\end{enumerate}
\end{lem}

\begin{proof}
To see this, note that there exist unique $v$-tiers $T_{1},T_{2}$ and $1\le m_{1},m_{2}\le d$ such that $$S_{i}=T_{i}\cap \Pi^{-1}(\mathfrak{C}_{m_{i}})$$ for $i=1,2$. For each $1\le i\le k$, pick a bijection between $f^{-1}_{i}(m_{1})$ and $f^{-1}_{i}(m_{2})$, and this induces a bijection $\bar{\Lambda}$ from components in $\mathfrak{C}_{m_{1}}$ to components in $\mathfrak{C}_{m_{2}}$. By Corollary \ref{7.14} (1), $\bar{\Lambda}$ induces a bijection $\Lambda$ from $v$-branches in $S_{1}$ to $v$-branches in $S_{2}$ such that $B$ and $\Lambda(B)$ are QII. We define $q$ as follows. Set $q(x)=x$ if $x\in P_{v}$. If $x\notin P_{v}$, let $D$ be the component of $X(\Gamma)\setminus P_{v}$ with $x\in D$ and let $B$ be the $v$-branch corresponding to $D$ (see Proposition \ref{7.23}). If $B$ is not inside $S_{1}\cup S_{2}$, then set $q(x)=x$. Otherwise we assume $B\subset S_{1}$. Let $B'=\Lambda(B)$ and let $D'$ be the associated component of $X(\Gamma)\setminus P_{v}$. Let $f$ be the elementary permutation (Definition \ref{8.1}) of $B$ and $B'$. We can assume $f(D)=D'$ (Proposition \ref{7.23}) and $f$ is a $(L,A)$-quasi-isometry with $L$ and $A$ independent of $B\subset S_{1}$ (see the discussion after Lemma \ref{8.2}). Set $q(x)=f(x)$ in this case. Then $q$ is a quasi-isometry and satisfies all the requirements.
\end{proof}

The reader may check that the same proof of Lemma~\ref{8.6} works to give the following lemma.
\begin{lem}
	\label{8.10}
Lemma \ref{8.6} is still true if we replaced $v_{i}$-tier by $v_{i}$-sub-tier in (3).
\end{lem}

In the rest of this subsection, we show the prime partitions on $F(\Gamma)$ give rise to a pocset structure on $F(\Gamma)$, and construct the dual cube complex.
\begin{definition}
	\label{def:basic}
	Pick non-prime vertex $\bar{v}\in F(\Gamma)$ and let $\{\mathfrak{C}_{j}\}_{j=1}^{d}$ be the prime factors at $\bar{v}$. A \textit{$\bar{v}$-halfspace} of $F(\Gamma)$ is a full subcomplex of form $$St(\bar{v})\cup(\cup_{j=1}^{m}\mathfrak{C}_{j})$$ or $$St(\bar{v})\cup(\cup_{j=m+1}^{d}\mathfrak{C}_{j})$$ with $1\le m<d$. Let $$H=St(\bar{v})\cup(\cup_{j=1}^{m}\mathfrak{C}_{j})\ (\mathrm{or}\ St(\bar{v})\cup(\cup_{j=m+1}^{d}\mathfrak{C}_{j})).$$ We define the \textit{complement} of $H$, denoted by $H^{c}$, to be $$St(\bar{v})\cup(\cup_{j=m+1}^{d}\mathfrak{C}_{j})\ (\mathrm{or}\ St(\bar{v})\cup(\cup_{j=1}^{m}\mathfrak{C}_{j})).$$ A \textit{$\bar{v}$-wall} of $F(\Gamma)$ is a pair of halfspaces $(H,H^{c})$. 
	
	Let $\mathcal{H}(\Gamma)$ be the collection of pairs $(\bar{v},H)$ such that $\bar{v}$ is non-prime and $H$ is a $\bar{v}$-halfspace. If there is another pair $(\bar{v}',H')\in \mathcal{H}(\Gamma)$ such that $H=H'$ and $\bar{v}\neq\bar{v}'$, then $(\bar{v}',H')$ and $(\bar{v},H)$ are viewed as different elements in $\mathcal{H}(\Gamma)$.

	Let $\mathcal{W}(\Gamma)$ be the collection of triples $(\bar{v},H,H^{c})$ such that $(H,H^{c})$ is a $\bar{v}$-wall. Occasionally, we will omit $\bar{v}$ when there is no ambiguity.
\end{definition}

\begin{definition}
	We say two halfspaces $(\bar{v}_{1},H_{1}),(\bar{v}_{2},H_{2})\in\mathcal{H}(\Gamma)$ are \textit{compatible} if $d(\bar{v}_{1},\bar{v}_{2})=1$ or $(H_{1}\cap H_{2})\nsubseteq St(\bar{v}_{1})$. 
\end{definition}

\begin{lem}
	\label{lem:6.25}
	Suppose $d(\bar v_1,\bar v_2)\ge 2$.  Let $C_{1}$ (or $C_{2}$) be the component of $F(\Gamma)\setminus St(\bar{v}_{1})$ (or $F(\Gamma)\setminus St(\bar{v}_{2})$) that contains $\bar{v}_{2}$ (or $\bar{v}_{1}$). Then the following holds:
	\begin{enumerate}
		\item $(H_{1}\cap H_{2})\nsubseteq St(\bar{v}_{1})$ implies $(H_{1}\cap H_{2})\nsubseteq St(\bar{v}_{2})$ and vice versa;
		\item assuming $(H_{1}\cap H_{2})\nsubseteq St(\bar{v}_{1})$, then exactly one of the following three possibilities is true: (1) $\bar{v}_{1}\in H_{2}$ and $\bar{v}_{2}\in H_{1}$; (2) $H_{2}\subsetneq H_{1}$; (3) $H_{1}\subsetneq H_{2}$;
		\item assuming $(H_{1}\cap H_{2})\nsubseteq St(\bar{v}_{1})$, then
		\begin{itemize}
			\item case (1) holds if and only if $C_i\subset H_i$ for $i=1,2$;
			\item case (2) holds if and only if $C_1\subset H_1$ and $C_2\cap H_2=\emptyset$;
			\item case (3) holds if and only if $C_2\subset H_2$ and $C_1\cap H_1=\emptyset$;
		\end{itemize}
		\item in Case (1) of Assertion (2), we have $H_1\nsubseteq H_2$ and $H_2\nsubseteq H_1$.
	\end{enumerate}
\end{lem}

%if $d(\bar{v}_{1},\bar{v}_{2})\ge 2$, then $H_{1}$ and $H_{2}$ are compatible if and only if $C_{1}\cap H_{1}=C_{2}\cap H_{2}=\emptyset$ is not true,

\begin{proof}
	We will first prove Assertions (2) and (3).
	
	We claim $C_{1}\cap H_{1}=C_{2}\cap H_{2}=\emptyset$ is impossible. Indeed, if $C_{1}\cap H_{1}=C_{2}\cap H_{2}=\emptyset$, then $\bar v_1\notin H_2$. Lemma~\ref{8.5} implies that all the components of $F(\Gamma)\setminus St(\bar v_2)$ that are in $H_2$, as well as  $\bar v_2$, are contained in a single component of $F(\Gamma)\setminus St(\bar v_1)$. Then $H_2\setminus E\subset C_1$ where $E$ is defined to be $$St(\bar v_1)\cap St(\bar v_2)=lk(\bar v_1)\cap lk(\bar v_2).$$
	Thus $$H_{2}\subset C_{1}\cup St(\bar{v}_{2}).$$ Hence $$H_{1}\cap H_{2}\subset H_{1}\cap(C_{1}\cup St(\bar{v}_{2}))=H_{1}\cap St(\bar{v}_{2}).$$ As $C_{1}\cap H_{1}=\emptyset$, we know $C_1$ is disjoint from all components of $F(\Gamma)\setminus St(\bar v_1)$ that are in $H_1$. Then $$H_{1}\cap St(\bar{v}_{2})=St(\bar{v}_{1})\cap St(\bar{v}_{2})\subset St(\bar{v}_{1}),$$ which contradicts $$(H_1\cap H_2)\nsubseteq St(\bar v_1).$$ Thus the claim is proved. Thus in order to prove Assertion (2), it suffices to prove Assertion (3).
	
	If $C_{i}\cap H_{i}\neq\emptyset$ for $i=1,2$, then actually $C_{i}\subset H_{i}$ and case (1) holds. The converse is clear. 
	%neither $H_{2}\subset H_{1}$ nor $H_{1}\subset H_{2}$ is true in this case by Lemma \ref{8.5}. 
	If $C_{1}\subset H_{1}$ and $C_{2}\cap H_{2}=\emptyset$, then $\bar v_1\notin H_2$ and $H_{2}\subset C_{1}\cup St(\bar{v}_{2})$ as before. Note that $C_1\subset H_1$ and $\bar v_2\subset C_1$ implies that $$(C_{1}\cup St(\bar{v}_{2}))\subset H_1.$$
	Moreover $\bar{v}_1\in H_1\setminus H_2$, hence case (2) is true. Conversely, if $H_2\subsetneq H_1$, then $\bar v_2\in H_2\subset H_1$. Thus $C_1\subset H_1$. Now we prove $C_2\cap H_2=\emptyset$. If this is not true, then $C_2\subset H_2$. Let $D$ be a component of $F(\Gamma)\setminus St(\bar v_1)$ such that $D\cap H_1=\emptyset$. Then $D\cap C_1=\emptyset$, hence $\bar v_2\notin D$. Then Lemma~\ref{8.5} implies $D$ and $\bar v_1$ are in the same connected component of $F(\Gamma)\setminus St(\bar v_2)$. Thus $D\subset C_2\subset H_2$. Then $H_2$ contains vertices which are not in $H_1$, which contradicts $H_2\subset H_1$. Thus $C_2\cap H_2=\emptyset$. Similarly, we can prove $C_{2}\subset H_{2}$ and $C_{1}\cap H_{1}=\emptyset$ iff case (3) holds. Thus Assertion (3) holds.
	Assertion (1) follows as in cases (1) and (3) of Assertion (2), we have $\bar v_1\in H_1\cap H_2$ and $\bar v_1\notin St(\bar v_2)$. In case (2), $H_2=H_1\cap H_2$, thus $(H_{1}\cap H_{2})\nsubseteq St(\bar{v}_{2})$ is clear.
	Assertion (4) also follows as we already showed that if $C_i\subset H_i$ for $i=1,2$, then $H_2$ contains vertices that are not in $H_1$. Similarly,  $H_1$ contains vertices that are not in $H_2$.
\end{proof}

\begin{lem}
	We define $(\bar{v}_{1},H_{1})\le (\bar{v}_{2},H_{2})$ if $d(\bar{v}_{1},\bar{v}_{2})\neq 1$ and $H_{1}\subset H_{2}$. Then $\le$ gives a partial order on $\mathcal{H}(\Gamma)$. Moreover, $(\mathcal{H}(\Gamma),\le)$ with the complement operation defined before form a pocset.
\end{lem}

\begin{proof}
	If $(\bar{v}_{1},H_{1})\le (\bar{v}_{2},H_{2})$ and $d(\bar v_1,\bar v_2)\ge 2$, then $(H_1\cap H_2)\nsubseteq St(\bar v_1)$. Thus Lemma~\ref{lem:6.25} that implies that $H_1\subsetneq H_2$, which means $(\bar{v}_{1},H_{1})\ge (\bar{v}_{2},H_{2})$ is impossible. Thus if $(\bar{v}_{1},H_{1})\le (\bar{v}_{2},H_{2})$ and $(\bar{v}_{1},H_{1})\ge (\bar{v}_{2},H_{2})$, then $\bar v_1=\bar v_2$. Thus the relation $\le$ is antisymmetric. Now we show transitivity. Pick $(\bar{v}_{3},H_{3})\in\mathcal{H}(\Gamma)$ such that $(\bar{v}_{2},H_{2})\le (\bar{v}_{3},H_{3})$, if two of $\bar{v}_{1},\bar{v}_{2},\bar{v}_{3}$ are the same, then $$(\bar{v}_{1},H_{1})\le (\bar{v}_{3},H_{3})$$ by definition. If $\bar{v}_{1},\bar{v}_{2},\bar{v}_{3}$ are pairwise distinct, let $C_{1}$ and $C_{2}$ be as in Lemma~\ref{lem:6.25} and let $C'_{2}$ be the component of $F(\Gamma)\setminus St(\bar{v}_{2})$ that contains $\bar{v}_{3}$. Since $H_{1}\subsetneq H_{2}$ and $H_{2}\subsetneq H_{3}$, then $C_{1}\cap H_{1}=\emptyset$, $C_{2}\subset H_{2}$ and $C'_{2}\cap H_{2}=\emptyset$ by  Lemma~\ref{lem:6.25}. Thus $\bar{v}_{1}$ and $\bar{v}_{3}$ are in different component of $F(\Gamma)\setminus St(\bar{v}_{2})$ and $d(\bar{v}_{1},\bar{v}_{3})\ge 2$, which implies $(\bar{v}_{1},H_{1})\le (\bar{v}_{3},H_{3})$. It follows that $\le$ is a partial order.
	
	It remains to show $(\bar{v}_{1},H_{1})\le (\bar{v}_{2},H_{2})$ implies $(\bar{v}_{2},H^{c}_{2})\le (\bar{v}_{1},H^{c}_{1})$. The case $\bar{v}_{1}=\bar{v}_{2}$ is clear. If $d(\bar{v}_{1},\bar{v}_{2})\ge 2$, then $H_{1}\cap C_{1}=\emptyset$ and $C_{2}\subset H_{2}$ by Lemma~\ref{lem:6.25}, hence $C_{1}\subset H^{c}_{1}$ and $C_{2}\cap H^{c}_{2}=\emptyset$, which implies $(\bar{v}_{2},H^{c}_{2})\le (\bar{v}_{1},H^{c}_{1})$ by Lemma~\ref{lem:6.25}.
\end{proof}

\begin{lem}
	\label{lem:6.27}
	A subset $U\subset \mathcal{H}(\Gamma)$ is an ultrafilter in the sense of Definition~\ref{2.4} if and only if $U$ satisfies both of the following conditions:
	\begin{enumerate}
		\item for each pair $(\bar{v},H)$ and $(\bar{v},H^{c})$, $U$ contains exactly one of them; 
		\item every pair of halfspaces in $U$ is compatible.
	\end{enumerate}
	
\end{lem}

\begin{proof}
	We fist claim the following are equivalent.
	\begin{enumerate}
		\item $(\bar{v}_{1},H_{1})$ and $(\bar{v}_{2},H_{2})$ are not compatible.
		\item $d(\bar{v}_{1},\bar{v}_{2})\neq 1$ and $(\bar{v}_{1},H_{1})\le (\bar{v}_{2},H^{c}_{2})$.
		\item $d(\bar{v}_{1},\bar{v}_{2})\neq 1$ and $(\bar{v}_{2},H_{2})\le (\bar{v}_{1},H^{c}_{1})$.
	\end{enumerate}
	To see the claim, let us assume $d(\bar{v}_{1},\bar{v}_{2})\ge 2$. Let $C_1$ and $C_2$ be as in Lemma~\ref{lem:6.25}. Then $(\bar{v}_{1},H_{1})$ and $(\bar{v}_{2},H_{2})$ are not compatible $\Leftrightarrow$ $C_{1}\cap H_{1}=C_{2}\cap H_{2}=\emptyset\Leftrightarrow C_{1}\cap H_{1}=\emptyset$ and $C_{2}\subset H^{c}_{2}\Leftrightarrow(\bar{v}_{1},H_{1})\le (\bar{v}_{2},H^{c}_{2})$, where the first step and the last step follow from Lemma~\ref{lem:6.25}. Similarly we can establish the equivalence of (1) and (3) in the claim.
	
	For if direction of the lemma, we need to show if $(\bar v_1,H_1)\le (\bar v_2,H_2)$ and $(\bar v_1,H_1)\in U$, then $(\bar v_2, H_2)\in U$. Indeed, if $(\bar v_2,H_2)\notin U$, then $(\bar v_2,H^c_2)\in U$. Then $$H_1\cap H^c_2\subset H_2\cap H^c_2\subset St(\bar v_2),$$ which contradicts that $(\bar v_1,H_1)$ and $(\bar v_2,H^c_2)$ are compatible. Now we prove the only if direction. Suppose $U$ is an ultrafilter. Then Lemma~\ref{lem:6.27} (1) is clear. If $U$ contains a pair of non-compatible halfspaces $(\bar v_1, H_1)$ and $(\bar v_2, H_2)$, then the claim in the previous paragraph implies that $$(\bar v_1,H_1)\le (\bar v_2,H^c_2).$$ Now Definition~\ref{2.4} (2) implies that $(\bar v_2,H^c_2)\in U$, which contradicts Definition~\ref{2.4} (1). This proves the only if direction of the lemma.
\end{proof}

Let $X$ be the $CAT(0)$ cube complex obtained from the pocset $\mathcal{H}(\Gamma)$ as in Theorem \ref{2.5}. Let $\Phi$ be the pocset isomorphism from the collection of halfspaces in $X$ to $\mathcal{H}(\Gamma)$ as in Theorem \ref{2.5}. Then $\Phi$ induces a bijective map from hyperplanes of $X$ to $\mathcal{W}(\Gamma)$ (cf. Definition~\ref{def:basic}), which is also denoted by $\Phi$. 

Denote the collection of vertices in $X$ by $\{x_{i}\}_{i=1}^{r}$ be the collection of vertices in $X$, and let $\{U(x_{i})\}_{i=1}^{r}$ be the corresponding ultrafilters. 

Let $\Phi(x_{i})$ be the intersection of halfspaces in $U(x_{i})$. For each subcomplex $A\subset X$, we define $\Phi(A)=\cup_{x\in A}\Phi(x)$ where $x$ ranges over vertices in $A$.

\begin{lem}
	\label{6.28}
	Then following are true.
	\begin{enumerate}
		\item For any vertex $\bar{u}\in F(\Gamma)$, $\Phi(x_{i})\setminus St(\bar{u})$ is contained in a prime factor at $\bar{u}$.
		\item For arbitrary simplex $g\subset F(\Gamma)$, there exists an ultrafilter $U$ such that the intersection of halfspaces in $U$ contains $g$. In particular, $\cup_{i=1}^{r}\Phi(x_{i})=F(\Gamma)$.
		\item $\Phi(x_{i})\neq\emptyset$ for all $i$.
		\item  If $A$ is convex, then $\Phi(A)$ is a full subcomplex.
	\end{enumerate}
\end{lem}

\begin{proof}
	Assertion (1) is true as each $\bar u$ wall has a $\bar u$-halfspace containing $\Phi(x_i)$, and all these $\bar u$-halfspaces are compatible.  Now we prove Assertion (2).
	Let $E'$ be the collection of non-prime vertices in $F(\Gamma)$ and let $G$ be the collection of vertices in $g$. Let $$E=E'\cup G=\{\bar u_1,\bar u_2,\ldots,\bar u_n\}.$$ For $1\le i\le n$, define $E_i=\{\bar u_1,\ldots,\bar u_i\}$. We order the elements in $E$ such that for each $1\le i\le n$ and any $\bar u_j\in E\setminus E^c_i$, we have $E_i\setminus St(\bar u_j)$ is inside a single connected component of $F(\Gamma)\setminus St(\bar u_j)$. This can be arranged in the same way as in Lemma~\ref{lem:tight}. 
	We can assume in addition that $\bar{u}_{i}\in G$ if and only if $i\le n_{1}$ and $\bar{u}_{i}\in E'$ if and only if $i\ge n_{2}$. For $i\ge n_{2}$, if $E_{i-1}\setminus St(\bar{u}_{i})\neq\emptyset$, let $C_{i}$ be the component of $F(\Gamma)\setminus St(\bar{u}_{i})$ that contains $E_{i-1}\setminus St(\bar{u}_{i})$ (this is possible by our choice of $E_{i}$). If $E_{i-1}\setminus St(\bar{u}_{i})=\emptyset$, let $C_{i}$ be an arbitrary component. We define $U$ by choosing the unique halfspace that contains $C_{i}$ in each $\bar{u}_{i}$-wall for $i\ge n_{2}$. It clear that the intersection of halfspaces in $U$ contains $g$. It remains to show two halfspaces $$(\bar{u}_{i},H_{1}),(\bar{u}_{j},H_{2})\in U$$ are compatible. The case $d(\bar{v}_{1},\bar{v}_{2})\le 1$ is trivial. We assume $d(\bar{u}_{i},\bar{u}_{j})\ge 2$. Suppose $i<j$, then $\bar{u}_{i}\subset C_{j}\subset H_{2}$, hence $$\bar{u}_{i}\in (H_{1}\cap H_{2})\setminus St(\bar{u}_{j}).$$ It follows that $U$ is an ultrafilter. This justifies (2).
	
	We now prove Assertion (3). 
	Let $U=\{(\bar{u}_{\lambda},H_{\lambda})\}_{\lambda\in\Lambda}$ be an ultrafilter and let $A=\cap_{\lambda\in\Lambda}H_{\lambda}$. To prove Assertion (3), it suffices to justify that if $(\bar{u}_{\lambda},H_{\lambda})$ is minimal in $U$, then $\bar{u}_{\lambda}\in A$. 
	Suppose the contrary is true, then there exists $(\bar{u}_{\lambda'},H_{\lambda'})\in U$ such that $\bar{u}_{\lambda}\notin H_{\lambda'}$, in particular $d(\bar{u}_{\lambda'},\bar{u}_{\lambda})\ge 2$. By Lemma~\ref{lem:6.27}, $H_{\lambda}$ and $H_{\lambda'}$ are compatible. Now Lemma~\ref{lem:6.25} (2) and (3) imply that we must have $H_{\lambda'}\subsetneq H_{\lambda}$, which contradicts the minimality of $(\bar{u}_{\lambda},H_{\lambda})$. 
	
	It remains to prove (4). Let $\{\mathfrak{h}_{i}\}_{i=1}^{t}$ be the collection of halfspaces in $X$ with $A\subset \mathfrak{h}_{i}$ and let $\Phi(\mathfrak{h}_{i})=(\bar{w}_{i},\mathfrak{h'}_{i})$. Suppose $K=\cap_{i=1}^{t}\mathfrak{h'}_{i}$. Since each $\mathfrak{h'}_{i}$ is a full subcomplex, so is $K$. It suffices to show $\Phi(A)=K$. The inclusion $\Phi(A)\subset K$ is clear. Let $\mathcal{W}'(\Gamma)$ be the $\Phi$-image of hyperplanes in $X$ that intersect $A$ and let $\mathcal{H}'(\Gamma)$ be the corresponding collection of halfspaces. Then $\mathcal{H}'(\Gamma)$ is a sub-pocset of $\mathcal{H}(\Gamma)$. We claim $U'\subset \mathcal{H}'(\Gamma)$ is an ultrafilter of $\mathcal{H}'(\Gamma)$ if and only if $U'\cup\{\mathfrak{h'}_{i}\}_{i=1}^{t}$ is an ultrafilter of $\mathcal{H}(\Gamma)$. To see this, we can use the pocset isomorphism $\Phi$ between the halfspaces of $X$ and $\mathcal{H}(\Ga)$ to translate this statement to a statement about halfspaces of $X$, which becomes obvious. We also deduce that $U'\cup\{\mathfrak{h'}_{i}\}_{i=1}^{t}$ corresponds to a vertex in $A$. Thus there is an isometric embedding from the $CAT(0)$ cube complex associated with $\mathcal{H}'(\Gamma)$ to $X$, whose image is exactly $A$. Let $\{U'_{i}\}_{i=1}^{\ell}$ be the collection of ultrafilters on $\mathcal{H}'(\Gamma)$ and let $K_{i}$ be the intersection of halfspaces in $U'_{i}$. Then we can prove $$\cup_{i=1}^{\ell}K_{i}=F(\Gamma)$$ as in Assertion (2). It follows that $$K=K\cap(\cup_{i=1}^{\ell}K_{i})=\cup_{i=1}^{\ell}(K\cap K_{i}),$$ but $K\cap K_{i}=U(x)$ for some vertex $x\in A$, so $K\subset \Phi(A)$.
\end{proof}

%Let $U'=(U\setminus\{(\bar{u}_{\lambda},H_{\lambda})\})\cup\{(\bar{u}_{\lambda},H^{c}_{\lambda})\}$. Then $U'$ is also an ultrafilter since $(\bar{u}_{\lambda},H_{\lambda})$ is minimal in $U$. Let $B$ be the intersection of halfspaces in $U'$. Then $\bar{u}\in B$, in particular, $B\neq\emptyset$.

%note that the discussion in the previous paragraph implies that if $x_{i},x_{j}$ are adjacent and $\Phi(x_{i})\neq\emptyset$, then $\Phi(x_{j})\neq\emptyset$. Moreover, there exists $i_{0}$ such that $\Phi(x_{i_{0}})\neq\emptyset$ by Lemma \ref{8.13}. So (2) follows from the connectedness of $X$.

Recall that two distinct walls $(\bar{v}_{1},H_{1},H^{c}_{1}),(\bar{v}_{2},H_{2},H^{c}_{2})\in\mathcal{W}(\Gamma)$ are \textit{transverse} if none of $(\bar{v}_{1},H_{1})< (\bar{v}_{2},H_{2})$, $(\bar{v}_{1},H_{1})< (\bar{v}_{2},H^{c}_{2})$, $(\bar{v}_{2},H_{2})< (\bar{v}_{1},H_{1})$ and $(\bar{v}_{2},H_{2})< (\bar{v}_{1},H^{c}_{1})$ is true. Thus two such walls are transverse if and only if $d(\bar{v}_{1},\bar{v}_{2})=1$ (note that when $d(\bar{v}_{1},\bar{v}_{2})=1$, even if $H_1\subset H_2$, we still have $(\bar{v}_{1},H_{1})\nleq (\bar{v}_{2},H_{2})$ by our definition). It follows that if $h'_{1}$ and $h'_{2}$ is a pair of crossing hyperplanes in $X$ and $\Phi(h'_{i})$ is a $\bar{v}_{i}$-wall for $i=1,2$, then $d(\bar{v}_{1},\bar{v}_{2})=1$.

\subsection{A filtration for $X$ and $F(\Gamma)$}
\label{subsec:filtration}
In this subsection, our goal is to understand the relationship between different $\Phi(x)$ with $x$ ranging over the $X$, and how $F(\Gamma)$ is assembled from these $\Phi(x)$. The main result of this subsection is Lemma~\ref{lem:key}. As the material in this subsection is comparably technical, The reader might want to assume Lemma~\ref{lem:key} and go ahead to Section~\ref{subsec_wall space} to see how it finishes the proof, before coming back to this subsection.

We now define a filtration for $X$ as well as for $F(\Ga)$. Such filtration is motivated by the generalized star extension introduced in \cite[Section 6.3]{raagqi1}.

We define a chain of convex subcomplexes in $X$ by induction. Pick a vertex $x\in X$ and set $L_{1}=\{x\}$. Suppose we have already defined $L_{i}$. If $L_{i}=X$, then we stop; if $L_{i}\subsetneq X$, pick an edge $e_{i}$ such that $e_{i}\cap L_{i}$ is a vertex and let $L_{i+1}$ be the convex hull of $L_{i}\cup e_{i}$. Let $\{L_{i}\}_{i=1}^{s}$ be the resulting collection of convex subcomplexes. Here is an alternative way of describing $L_{i+1}$. Suppose $h_{i}$ is the hyperplane dual to $e_{i}$ and $N_{i}$ is the carrier of $h_{i}$. Then $h_{i}\cap L_{i}=\emptyset$ by the convexity of $L_{i}$. Let $M_{i}$ be the copy of $(L_{i}\cap N_{i})\times[0,1]$ inside $N_{i}$. Then $L_{i+1}=L_{i}\cup M_{i}$.

Now we look at the relation between $\Phi(L_{i})$ and $\Phi(L_{i+1})$. For $j=1,2$, let $M_{ij}$ be the subcomplex of $M_i$ of form $$(L_{i}\cap N_{i})\times\{j-1\}.$$ We assume $M_{i1}=L_{i}\cap N_{i}$ and let $p:M_{i1}\to M_{i2}$ be the map induced by parallelism. Suppose $(\bar{v},H_{i})\in \mathcal{H}(\Ga)$ is the element corresponding to the halfspace of $h_{i}$ that contains $L_{i}$. Then $$\Phi(M_{i1})\subset \Phi(K)\subset H_{i}$$ and $\Phi(M_{i2})\subset H^{c}_{i}$. For any vertex $x\in M_{i1}$, $(\bar{v},H_{i})$ is a minimal element in $U(x)$, so $\bar{v}\in \Phi(x)\subset\Phi(M_{i1})$. Similarly, $\bar{v}\in\Phi(M_{i2})$.

\begin{lem}
	\label{lem:key}
	The following are true.
	\begin{enumerate}
		\item There is a simplicial isomorphism $\bar h_*:\Phi(M_{i1})\to \Phi(M_{i2})$ with $\bar h_*(\Phi(x))=\Phi(p(x))$ for any vertex $x\in \Phi(M_{i1})$ and $(\bar h_*)^{-1}(\Phi(x))=\Phi(p^{-1}(x))$ for any vertex $x\in \Phi(M_{i2})$.
		\item $\Phi(L_i)\cap \Phi(M_{i1})=\Phi(L_i)\cap \Phi(M_{i2})=St(\bar v, \Phi(M_{i1}))$.
	\end{enumerate}
Thus $\Phi(L_{i+1})$ can be obtained by taking $\Phi(L_{i})$ and $\Phi(M_{i1})\cong \Phi(M_{i2})$, and gluing them along $St(\bar{v},\Phi(M_{i1}))$. 
\end{lem}

We extracted a technical part of the proof as Lemma~\ref{8.12} below. Now we prove Lemma~\ref{lem:key}, assuming Lemma~\ref{8.12}.
\begin{proof}
We claim there exist $\mathfrak{C}_{1}$ and $\mathfrak{C}_{2}$ which are prime factors at $\bar{v}$ such that $\mathfrak{C}_{1}\subset H_{i}$, $\mathfrak{C}_{2}\subset H^{c}_{i}$ and $\Phi(M_{ij})\setminus St(\bar{v})\subset\mathfrak{C}_{j}$ for $j=1,2$. Pick adjacent vertices $x_{1},x_{2}\in M_{i1}$, then there exists $(\bar{v},H'_{i})\in U(x_{1})$ such that $(H_{i}\cap H'_{i})\setminus St(\bar{v})$ is a prime factor at $\bar{v}$. Denote this prime factor by $\mathfrak{C}_{1}$, then $$\Phi(x_{1})\setminus St(\bar{v})\subset \mathfrak{C}_{1}.$$ Let $h$ be the hyperplane dual to the edge joining $x_{1}$ and $x_{2}$ and let $\Phi(h)=(\bar{w},H,H^{c})$. Then $$U(x_{2})=(U(x_{1})\setminus \{H\})\cup\{H^{c}\}.$$ Since $h$ crosses $h_{i}$, $d(\bar{w},\bar{v})=1$. Thus $$(\bar{v},H_{i}),(\bar{v},H'_{i})\in U(x_{2}),$$ which implies $\Phi(x_{2})\setminus St(\bar{v})\subset \mathfrak{C}_{1}$. Now $\Phi(M_{i1})\setminus St(\bar{v})\subset\mathfrak{C}_{1}$ follows from the connectedness of $M_{i1}$. We can choose $\mathfrak{C}_{2}$ in a similar way.

The above argument also implies for any vertex $\bar{u}$ such that $d(\bar{u},\bar{v})\neq 1$, $\Phi(M_{i1})\setminus St(\bar{u})$ is contained in a prime factor at $\bar{u}$. Note that if $M_{i1}\subset St(\bar{v})$, then $\Phi(x)=\Phi(p(x))$ for any vertex $x\in M_{i1}$, hence $M_{i2}=M_{i1}$. Now we assume $M_{i1}\nsubseteq St(\bar{v})$, then we can set up as in Lemma \ref{8.12} below with respect to $K=M_{i1}$, $\bar{v}\in K$, $\mathfrak{C}_{1}$ and $\mathfrak{C}_{2}$. Note that $K$ is a full subcomplex of $F(\Gamma)$ by Lemma~\ref{6.28}. Let $h_{\ast}$ and $\bar{h}_{\ast}$ be the maps in Lemma \ref{8.12}. We claim $\bar{h}_{\ast}(M_{i1})=M_{i2}$. 

Pick vertex $x\in M_{i1}$. We claim for any vertex $\bar{u}$ with $d(\bar{u},\bar{v})\ge 1$, $\Phi(x)\setminus St(\bar{u})\neq\emptyset$ if and only if $\bar{h}_{\ast}(\Phi(x))\setminus St(\bar{u})\neq\emptyset$, and in this case $\Phi(x)\setminus St(\bar{u})$ and $\bar{h}_{\ast}(\Phi(x))\setminus St(\bar{u})$ are in the same prime factor at $\bar{u}$. This follows from Lemma \ref{8.12} (3b) when $d(\bar{u},\bar{v})= 1$. When $d(\bar{u},\bar{v})>1$, note that $\bar{v}\in\Phi(x)$ and $\bar{v}\in\bar{h}_{\ast}(\Phi(x))$, then the claim follows from Lemma \ref{8.12} (3a). Thus for any $(\bar{u},H)\in U(x)$ with $d(\bar{u},\bar{v})\ge 1$, $\bar{h}_{\ast}(\Phi(x))\subset H$. Moreover, (1) of Lemma \ref{8.12} implies $$\bar{h}_{\ast}(\Phi(x))\setminus St(\bar{v})\subset\mathfrak{C}_{2},$$ so $$\bar{h}_{\ast}(\Phi(x))\subset\Phi(p(x))$$ by our choice of $\mathfrak{C}_{2}$ (recall that $p:M_{i1}\to M_{i2}$ is the parallelism map). Denote the number of vertices in $\Phi(x)$ by $|\Phi(x)|$, then $$|\Phi(x)|\le |\Phi(p(x))|.$$ By reversing the role of $M_{i1}$ and $M_{i2}$ and apply Lemma \ref{8.12} with $K=M_{i2}$, we have $|\Phi(p(x))|\le |\Phi(x)|$, hence $$|\Phi(x)|=|\Phi(p(x))|.$$ But $\bar{h}_{\ast}(\Phi(x))$ and $\Phi(p(x))$ are both full subcomplexes, so $$\bar{h}_{\ast}(\Phi(x))=\Phi(p(x)).$$ Thus $\bar{h}_{\ast}(\Phi(M_{i1}))=\Phi(M_{i2})$. This also implies $(\bar h_*)^{-1}(\Phi(x))=\Phi(p^{-1}(x))$, which finished the proof of Assertion (1).

Since $\Phi(M_{ij})$ is a full subcomplex for $j=1,2$ (Lemma~\ref{6.28}), we have $$\bar{h}_{\ast}(St(\bar{v})\cap\Phi(M_{i1}))=\bar{h}_{\ast}(St(\bar{v},\Phi(M_{i1})))=St(\bar{v},\Phi(M_{i2}))=St(\bar{v})\cap\Phi(M_{i2}).$$ However, $\bar{h}_{\ast}$ fixes every point in $St(\bar{v})\cap\Phi(M_{i1})$, so $$St(\bar{v})\cap\Phi(M_{i1})=St(\bar{v})\cap\Phi(M_{i2}).$$ Recall that $$\Phi(L_{i})\cap\Phi(M_{i2})\subset St(\bar{v}),$$ so
\begin{align*}
\Phi(M_{i1})\cap St(\bar{v})&\subset\Phi(M_{i1})\cap\Phi(M_{i2})\subset\Phi(L_{i})\cap\Phi(M_{i2})\\
&\subset\Phi(M_{i2})\cap St(\bar{v})=\Phi(M_{i1})\cap St(\bar{v})
\end{align*}
and all these sets are equal.
\end{proof}

It remains to prove Lemma~\ref{8.12}. We start with an auxiliary result needed in the proof of Lemma~\ref{8.12}.
\begin{lem}
\label{8.11}
If $G(\Gamma)$ is of type II, then given any two vertices $v_{1},v_{2}\in\mathcal{P}(\Gamma)$, there only exists finitely many vertices $w$ such that $v_{1}|_{w}v_{2}$.
\end{lem}

\begin{proof}
We first prove a preparatory claim as follows: pick vertex $x\in X(\Gamma)$ and $v\in\mathcal{P}(\Gamma)\setminus (F(\Gamma))_{x}$, let $\bar{w}\in\Gamma$ and let $S_{\bar{w}}\subset(F(\Gamma))_{x}$ be the lift of $St(\bar{w})\subset F(\Gamma)$. Then $S_{\bar{w}}\setminus St(v)\neq\emptyset$. Now we prove this claim. Suppose the contrary is true. Put $\bar{v}=\pi(v)$. Then $St(\bar{w})\subset St(\bar{v})$, hence $\bar{v}\in St(\bar{w})$. Let $v'\in (F(\Gamma))_{x}$ be the lift of $\bar{v}$. Then $v'\in S_{\bar{w}}\subset St(v)$. Note that $d(v',v)=1$ is impossible since $\pi(v')=\pi(v)$, so $v'=v$, which is contradictory to $v\notin (F(\Gamma))_{x}$. Thus the claim is proved.

Now we prove the lemma.
Pick an edge path $\omega$ which connects a vertex in $P_{v_{1}}$ to a vertex in $P_{v_{2}}$. Let $\{x_{i}\}_{i=0}^{n}$ be consecutive vertices in $\omega$, and let $\ell_{i}$ be the standard geodesic containing $x_{i-1}$ and $x_{i}$. Let $K_{i}=(F(\Gamma))_{x_{i}}$ and $K=\cup_{i=0}^{n}K_{i}$. Then $v_{1}\in K_{0}$ and $v_{2}\in K_{n}$. It suffices to show that for any vertex $v\notin K$, $v_{1}$ and $v_{2}$ are in the same $v$-branch. To see this, note that $$\pi(K_{i-1}\cap K_{i})=St(\pi(\Delta(\ell_{i}))),$$ so for $1\le i\le n$, there exists vertex $w_{i}$ such that $$w_{i}\in(K_{i-1}\cap K_{i})\setminus St(v)$$ by the auxiliary result above. By Lemma \ref{7.8}, $w_{i}w_{i+1}|_{v}$ for $1\le i\le n-1$, $v_{1}w_{1}|_{v}$ and $w_{n}v_{2}|_{v}$, so $v_{1}v_{2}|_{v}$.
\end{proof}

\begin{lem}
\label{8.12}
Let $\bar{v}\in F(\Gamma)$ be a non-prime vertex and let $K\subset F(\Gamma)$ be a full subcomplex containing $\bar{v}$ such that for any vertex $\bar{u}\in F(\Gamma)\setminus lk(\bar{v})$, $K\setminus St(\bar{u})$ is inside a prime factor at $\bar{u}$. Suppose in addition that $K\setminus St(\bar{v})\neq\emptyset$ and let $\mathfrak{C}_{1}$ be the prime factor at $\bar{v}$ that contains $K\setminus St(\bar{v})$. Let $\mathfrak{C}_{2}$ be a different prime factor at $\bar{v}$. 
Pick vertex $x\in X(\Gamma)$ and let $K',v,\mathfrak{C}'_{1}$ and $\mathfrak{C}'_{2}$ be the lift of $K,\bar{v},\mathfrak{C}_{1}$ and $\mathfrak{C}_{2}$ in $(F(\Gamma))_{x}\subset\mathcal{P}(\Gamma)$ respectively. For $i=1,2$, let $S_{i}$ be the $v$-sub-tier that contains $\mathfrak{C}'_{i}$. Let $q:X(\Gamma)\to X(\Gamma)$ be a quasi-isometry such that $q_{\ast}$ permutes $S_{1}$ and $S_{2}$ and fixes every points in $\mathcal{P}(\Gamma)\setminus(S_{1}\cup S_{2})$ (Lemma \ref{premuting sub-tiers}).

There exists a quasi-isometry $h:X(\Gamma)\to X(\Gamma)$ such that
\begin{enumerate}
\item $h_{\ast}$ permutes $S_{1}$ and $S_{2}$ and fixes every vertex in $\mathcal{P}(\Gamma)\setminus(S_{1}\cup S_{2})$.
\item The projection map $\pi:\mathcal{P}(\Gamma)\to F(\Gamma)$ restricted on $K'\cup h_{\ast}(K')$ is injective.
\item Let $M=\pi(K'\cup h_{\ast}(K'))$. Let $\bar{h}_{\ast}:K\to\pi(h_{\ast}(K'))$ be the simplicial isomorphism induced by $h_{\ast}$. Pick vertex $\bar{u}\in F(\Gamma)$. Then
\begin{enumerate}
\item If $d(\bar{u},\bar{v})\ge 2$, then $M\setminus St(\bar{u})$ is contained in one prime factor at $\bar{u}$. 
\item If $d(\bar{u},\bar{v})=1$, then $\bar{r}\in K\setminus St(\bar{u})$ if and only if $\bar{h}_{\ast}(\bar{r})\in \bar{h}_{\ast}(K)\setminus St(\bar{u})$. In this case, $\bar{r}$ and $\bar{h}_{\ast}(\bar{r})$ are in the same prime factor at $\bar{u}$. 
\item The subcomplex $\bar{h}_{\ast}(K)$ is full.
\end{enumerate}
\end{enumerate}
\end{lem}

\begin{proof}
We assume $K$ is a full subcomplex. The general case follows from this special case by considering the full subcomplex spanned by $K$.

Let $L=K'\cup q_{\ast}(K')$. By Lemma \ref{8.11}, there are only finitely many vertices $w\in\mathcal{P}(\Gamma)$ such that $St(w)$ separates two vertices in $L$. Denote these vertices by $\{w_{i}\}_{i=1}^{n}$. We claim if $St(w)$ separates two vertices in $K'$ and $d(w,v)\ge 2$, then vertices of $q_{\ast}(K')\setminus St(w)$ are in the same $w$-branch. To see this, suppose $v_{1}|_{w}v_{2}$ for $v_{1},v_{2}\in K'$, then either $v_{1}|_{w}v$ or $v_{2}|_{w}v$, so $v_{1}w|_{v}$ or $v_{2}w|_{v}$ by Lemma \ref{8.8}. Then the claim follows from Lemma \ref{8.5}. Note that the claim is also true if we switch the role of $K'$ and $q_{\ast}(K')$.

By the above claim, we can reorder and divide $\{w_{i}\}_{i=1}^{n}$ into the following 4 groups. 
\begin{enumerate}
\item $w_{1}=v$.
\item $w_{i}\in lk(v)$ for $2\le i\le n_{1}$.
\item $St(w_{i})$ separates $v$ from some vertex in $K'$ for $n_{1}+1\le i\le n_{2}$.
\item $St(w_{i})$ separates $v$ from some vertex in $q_{\ast}(K')$ for $n_{2}+1\le i\le n$.
\end{enumerate}
Note that $q_{\ast}$ induces a bijection between $\{w_{i}\}_{i=n_{1}+1}^{n_{2}}$ and $\{w_{i}\}_{n_{2}+1}^{n}$. Let $$k=n_{2}-n_{1}=n-n_{2}.$$ We also assume $q_{\ast}(w_{i})=w_{i+k}$ for $n_{1}+1\le i\le n_{2}$. 

Let $D=\{w_{i}\}_{i=1}^{n_{2}}$. We claim $D\setminus St(w_{i})$ stays inside a $w_{i}$-branch for $i>n_{2}$. To see this, let $B$ be the $w_{i}$-branch that contains $v$ and let $w_{i_{0}}\in D\setminus St(w_{i})$. It is clear that $w_{i_{0}}\in B$ if $i_{0}\le n_{1}$. If $n_{1}<i_{0}\le n_{2}$, by above discussion, there exists $u\in K'$ such that $w_{i_{0}}u|_{v}$, similarly, there exists $u'\in q_{\ast}(K')$ such that $w_{i}u'|_{v}$. But $u|_{v}u'$, so $w_{i_{0}}|_{v}w_{i}$, and by Lemma \ref{8.8}, we have $w_{i_{0}}v|_{w_{i}}$ and $w_{i_{0}}\in B$. This discussion also implies $\{w_{i}\}_{i=n_{1}+1}^{n_{2}}\subset S_{1}$ and $\{w_{i}\}_{i=n_{2}+1}^{n}\subset S_{2}$.

Let $\{B_{\lambda}\}_{\lambda\in\Lambda}$ be the collection of $w_{i}$-branches that contain vertices of $K'$, here $i$ ranges over all values between $2$ and $n_{1}$. By Lemma \ref{8.6}, we can post-compose $q_{\ast}$ with a simplicial isomorphism $f_{\ast}$ to obtain a map $q'_{\ast}=f_{\ast}\circ q_{\ast}$ that satisfies the conclusions of Lemma \ref{8.6}. By Lemma \ref{8.10}, we can assume that for vertex $u\in K'\setminus St(w_{i})$ ($2\le i\le n_{1}$), $u$ and $q'_{\ast}(u)$ are in the same $w_{i}$-sub-tier. Let $$L'=K'\cup q'_{\ast}(K').$$ Note that if $B_{\lambda}\subset \mathcal{P}(\Gamma)\setminus St(v)$, then $B_{\lambda}\subset S_{1}$. So by Remark~\ref{rmk:support}, $f_{\ast}$ is a composition of elementary permutations that happen inside $S_{2}$, hence $f_{\ast}$ fixes every point in $S_{1}$, in particular, $f_{\ast}(K')=K'$ and $f_{\ast}(L)=L'$. Let $w'_{i}=f_{\ast}(w_{i})$. Then $\{w'_{i}\}_{i=1}^{n}$ is the collection of vertices such that $St(w)$ separates two vertices in $L'$. We divide $\{w'_{i}\}_{i=1}^{n}$ into 4 groups as before and this partition coincides with the partition induced by $f_{\ast}$. Since $w_{i}\in S_{1}$ for $n_1+1\le i\le n_{2}$, so $w'_{i}=w_{i}$ for $i\le n_{2}$. Moreover, the claim in the previous paragraph also holds for $\{w'_{i}\}_{i=1}^{n}$.

We claim for $n_{1}< i\le n_{2}$, vertices of $L'\setminus St(w'_{i})$ are in the same $w'_{i}$-sub-tier. Actually, by Lemma \ref{7.8}, $w'_{i}\in (F(\Gamma))_{x}$ for $n_{1}<i\le n_{2}$. Recall that $K\setminus St(\pi(w'_{i}))$ is inside a prime factor at $\pi(w'_{i})$, so vertices of $K'\setminus St(w'_{i})$ are inside a $w'_{i}$-sub-tier. We know $S_{2}$ contains vertices of $q'_{\ast}(K')\setminus St(v)$, but $w'_{i}\in S_{1}$, so vertices of $q'_{\ast}(K')\setminus St(v)$ and $v$ are in the same $w'_{i}$-branch by Lemma \ref{8.5}. The claim follows.

Let $E_{0}=\{w'_{i}\}_{i=1}^{n_{2}}$ and $E=\{w'_{i}\}_{1}^{n}$. Then $E_{0}$ is tight in $E$ by previous discussion. Let $E_{0}\subsetneq E_{1}\subsetneq E_{2}\subsetneq\cdots\subsetneq E_{k}=E$ be a filtration such that each $E_i$ is tight in $E$ (this can be arranged as in Lemma~\ref{lem:tight}). Up to reordering, we assume $w'_{i+n_{2}}=E_{i}\setminus E_{i-1}$ for $1\le i\le k$ and $q'_{\ast}(w'_{i})=w'_{i+k}$ for $n_{1}< i\le n_{2}$. Suppose there is an integer $m$ ($0\le m\le k$) such that
\begin{enumerate}
\item $q'_{\ast}$ permutes $S_{1}$ and $S_{2}$ and fixes every vertex in $\mathcal{P}(\Gamma)\setminus(S_{1}\cup S_{2})$.
\item For $2\le i\le n_{1}$ and vertex $u\in K'\setminus St(w'_{i})$, $u$ and $q'_{\ast}(u)$ are in the same $w'_{i}$-sub-tier.
\item For $n_{1}< i\le n_{2}+m$, vertices of $L'\setminus St(w'_{i})$ are contained in one $w'_{i}$-sub-tier.
\end{enumerate}
Such $m$ always exists, since (1) and (2) is always true and (3) is true when $m=0$. Our goal to is modify the map $q'_{\ast}$ such that $m=k$. Now we assume $m<k$ and argue by induction.

Let $a=n_{1}+m+1$ and $b=n_{2}+m+1$. Since vertices of $K'\setminus St(w'_{a})$ stay inside a $w'_{a}$-sub-tier and $w'_{b}=q'_{\ast}(w'_{a})$, there is a simplicial isomorphism $g_{\ast}$ which is a composition of finitely many elementary permutations of $w'_{b}$-branches such that 
\begin{enumerate}[label=(\alph*)]
	\item vertices in $g_{\ast}(q'_{\ast}(K'))\setminus St(w'_{b})$ are in the same $w'_{b}$-sub-tier; 
	\item let $B'$ be the $w'_{b}$-branch that contains $v$, then $g_{\ast}$ fixes every point in $B'$.
\end{enumerate}
 Lemma \ref{8.5} implies vertices of $K'\setminus St(w'_{b})$ are in $B'$, so $g_{\ast}$ fixes every point in $K'$. Moreover, the tightness of $E_{m}$ implies $g_{\ast}(w'_{i})=w'_{i}$ for $1\le i\le b$. By Lemma \ref{8.6} and Lemma~\ref{8.10}, we can assume in addition that $g_*$ satisfies
 \begin{enumerate}[resume,label=(\alph*)]
 	\item for any vertex $t\in St(w'_{b})\cap E_{m}$ and any $t$-branch $A$ with $A\cap L'\neq\emptyset$, $g_{\ast}(A)$ and $A$ are in the same $t$-sub-tier. 
 \end{enumerate}

Let $$L''=K'\cup g_{\ast}(q'_{\ast}(K')).$$ Then $g_{\ast}(L')=L''$. Let $w''_{i}=g_{\ast}(w'_{i})$. Then $\{w''_{i}\}_{i=1}^{n}$ is the collection of vertices such that $St(w''_{i})$ separates two vertices in $L''$. Moreover, $g_{\ast}(E_{0})\subsetneq g_{\ast}(E_{1})\subsetneq\cdots\subsetneq g_{\ast}(E_{k})=g_{\ast}(E)$ satisfies that each $g_*(E_i)$ is tight in $g_*(E)$. Note that $g_{\ast}(E_{i})=E_{i}$ for $i\le m+1$. We claim 
\begin{enumerate}[label=(\roman*)]
\item $g_{\ast}\circ q'_{\ast}$ permutes $S_{1}$ and $S_{2}$ and fixes every vertex in $\mathcal{P}(\Gamma)\setminus(S_{1}\cup S_{2})$.
\item For $2\le i\le n_{1}$ and vertex $u\in K'\setminus St(w''_{i})$, $u$ and $g_{\ast}(q'_{\ast}(u))$ are in the same $w''_{i}$-sub-tier.
\item For $n_{1}< i\le b$, vertices of $L''\setminus St(w''_{i})$ are contained in one $w''_{i}$-sub-tier.
\end{enumerate}

Part (i) follows from property (b) of $g_{\ast}$ and Lemma \ref{8.5}. 

We now verify (ii). Assume $2\le i\le n_{1}$, then $w''_{i}=w'_{i}$. If $d(w'_{i},w'_{b})\ge 2$. As $g_{\ast}$ fixes every point in $B'$, we know $g_{\ast}$ induces trivial permutation of $w'_{i}$-branches (indeed, for a $w'_i$-branch $D$, if $w'_b\notin D$, then $g_{\ast}(D)=D$ by Lemma \ref{8.5} and  $g_*|_{B'}=\mathrm{Id}$; if $w'_b\in D$, then $g_{\ast}(D)=D$ is still true since $g_{\ast}$ fixes $w'_b$ and $w'_i$). Now (2) follows from the inductive assumption. If $d(w'_{i},w'_{b})=1$, since $q'_{\ast}(u)\in L'$, $q'_{\ast}(u)$ and $g_{\ast}(q'_{\ast}(u))$ are in the same $w'_{i}$-sub-tier by property (c) of $g_{\ast}$. But $u$ and $q'_{\ast}(u)$ are in the same $w'_{i}$-sub-tier by induction, thus (ii) follows. 

It remains to verify (iii). Suppose $n_{1}< i\le b$. Then $w''_{i}=w'_{i}$ and $w''_{b}=w'_{b}$. Since $g_{\ast}(L')=L''$, the case $i<b$ and $d(w'_{i},w'_{b})=1$ follows from inductive assumption and property (c) of $g_{\ast}$. If $i<b$ and $d(w'_{i},w'_{b})=2$, then we argue as before to see that $g_{\ast}$ induces trivial permutation of $w'_{i}$-branches. Then (iii) follows from the inductive assumption. If $i=b$, by Lemma \ref{8.5}, vertices of $K'\setminus St(w'_{b})$ and $v$ are in the same $w'_{b}$-branch, then (iii) follows from property (a) of $g_{\ast}$. 

After applying the above induction process for finitely many times, we obtain a simplicial isomorphism $h_{\ast}:\mathcal{P}(\Gamma)\to\mathcal{P}(\Gamma)$ which satisfies (1) in Lemma \ref{8.12}. Moreover, let $$\tilde{L}=K'\cup h_{\ast}(K')$$ and let $\{\tilde{w}_{i}\}_{i=1}^{n}$ be the collection of vertices such that $St(\tilde{w}_{i})$ separates two vertices of $\tilde{L}$. Then
\begin{enumerate}
\item For $2\le i\le n_{1}$ and vertex $u\in K'\setminus St(\tilde{w}_{i})$, $u$ and $h_{\ast}(u)$ are in the same $\tilde{w}_{i}$-sub-tier.
\item For $n_{1}< i\le n$, vertices of $\tilde{L}\setminus St(\tilde{w}_{i})$ are contained in one $\tilde{w}_{i}$-sub-tier.
\end{enumerate}
Let $2\le i\le n_{1}$. Since $h_{\ast}(\tilde{w}_{i})=\tilde{w}_{i}$ and vertices of $K'\setminus St(\tilde{w}_{i})$ are contained in one $\tilde{w}_{i}$-tier, (1) implies actually vertices of $\tilde{L}\setminus St(\tilde{w}_{i})$ are contained in one $\tilde{w}_{i}$-tier. Thus $\tilde{L}$ satisfies the assumption of Lemma \ref{8.7} and (2) of Lemma \ref{8.12} follows. Note that $K'$ is a full subcomplex, so is $h_{\ast}(K')$, then $$\bar{h}_{\ast}(K)=\pi(h_{\ast}(K'))$$ is a full subcomplex. 

Pick vertex $x_{0}\in\cap_{w\in\tilde{L}}P_{w}$ ($x_{0}$ may not equal to $x$), then $\tilde{L}\subset (F(\Gamma))_{x_{0}}$. Let $$i_{x_0}:F(\Ga)\to(F(\Gamma))_{x_{0}}\subset \P(\Ga)$$ be the natural embedding. Then $i_{x_0}(M)=\tilde{L}$ and (3a) follows from property (2) of $h_{\ast}$ and Lemma~\ref{correspondence}. Now we look at (3b). For vertex $\bar{k}\in K$,
\begin{align*}
d(\bar{k},\bar{u})&=d(i_{x_0}(\bar{k}),i_{x_0}(\bar{u}))=d( h_{\ast}\circ i_{x_0}(\bar{k}),i_{x_0}(\bar{u})) \\
&=d(\pi\circ h_{\ast}\circ i_{x_0}(\bar{k}),\pi\circ i_{x_0}(\bar{u}))=d(\bar{h}_{\ast}(\bar{k}),\bar{u}).
\end{align*}
The first and the third equality follow from Lemma \ref{isometric embedding}, and the second equality holds since $h_{\ast}$ fixes $i_{x_0}(\bar{u})$. Thus the first part of (3b) is true. The rest of $(3b)$ follows from property (1) of $h_{\ast}$.
\end{proof}

\subsection{Recognizing special subgroups}
\label{subsec_wall space}
Let $L_1\subsetneq L_2\subsetneq \ldots\subsetneq L_s=X$ be the filtration of $X$ defined in Section~\ref{subsec:filtration}.
Let $\Gamma'$ be the 1-skeleton of $\Phi(L_{1})$. Note that $L_1$ is a single vertex and Lemma~\ref{lem:key} implies that the isomorphism type of $\Ga'$ does not depend on the choice of $L_1$. The main goal of this subsection is the following.
\begin{prop}
	\label{prop:special subgroup}
$G(\Ga')$ is a special subgroup of $G(\Ga)$ and $G(\Ga')$ is prime.
\end{prop}

The following is an immediate consequence of this proposition.

\begin{thm}
	\label{8.15}
	Let $G(\Gamma)$ be a right-angled Artin group of type II. Then there exists a prime right-angled Artin group $G(\Gamma')$ such that $G(\Gamma)$ is isomorphic to a special subgroup of finite index in $G(\Gamma')$.
\end{thm}
Now we prove Proposition~\ref{prop:special subgroup}.
\begin{proof}
For convex subcomplex $E\subset X(\Gamma')$, let $\{\ell_{\lambda}\}_{\lambda\in\Lambda}$ be the collection of standard geodesics in $X(\Gamma')$ such that $\ell_{\lambda}\cap E\neq\emptyset$. Denote the full subcomplex in $\mathcal{P}(\Gamma')$ spanned by $\{\Delta(\ell_{\lambda})\}_{\lambda\in\Lambda}$ by $\hat{E}$. An edge $e\subset X(\Gamma')$ is called a \textit{$v$-edge} for $v\in\mathcal{P}(\Gamma)$ if $\Delta(\ell_{e})=v$ ($\ell_{e}$ is the standard geodesic containing $e$). An edge $e\subset X$ is called a \textit{$\bar{v}$-edge} for $\bar{v}\in F(\Gamma)$ if the ultrafilters corresponding to two vertices of $e$ differ by a $\bar{v}$-halfspace. 

We are going to define a sequence of simplicial embeddings $f_{i}:\Phi(L_{i})\to \mathcal{P}(\Gamma')$ and cubical embeddings $g_{i}:L_{i}\to X(\Gamma')$ with $$f_{i+1}|_{\Phi(L_{i})}=f_{i}\ \mathrm{and}\ g_{i+1}|_{L_{i}}=g_{i}$$ which satisfy the following compatibility conditions:
\begin{enumerate}
\item $g_{i}(L_{i})$ is a compact convex subcomplex of $X(\Gamma')$.
\item For any vertex $x\in L_{i}$, $f_{i}(\Phi(x))$=$\widehat{g_{i}(x)}$. In particular, $f_{i}(\Phi(L_{i}))=\widehat{g_{i}(L_{i})}$.
\item $g_{i}$ sends a $\bar{v}$-edge to a $f_{i}(\bar{v})$-edge.
\end{enumerate}
Note that the existence of the embedding when $i=s$ will imply $G(\Gamma')$ is a special subgroup of $G(\Gamma)$.

We need several observations before the construction of $g_{i}$ and $f_{i}$. Pick vertex $v\in\mathcal{P}(\Gamma')$, then the vertices of $P_{v}$ are exactly those vertices $x\in X(\Gamma')$ with $v\in\hat{x}$. Let $\ell_{v}$ be a standard geodesic such that $\Delta(\ell_{v})=v$ and let $h_{v}$ be a hyperplane dual to $\ell_{v}$. We identify $P_{v}$ with $\ell_{v}\times h_{v}$. Then $e\subset X(\Gamma')$ is a $v$-edge if and only if $e\in P_{v}$ and $e$ has trivial projection to the $h_{v}$-factor. Actually, these statements have their analogues in $X$.

%Pick vertex $x,x'\in X$ and let $\{\bar{v},H\}\in U(x)$ be a minimal halfspace. Then $\bar{v}\subset\Phi(x)$. We claim $\bar{v}\in\Phi(x')$ if and only if for any $(\bar{v}',H',H'^{c})\in\mathcal{W}(\Gamma)$ with $(\bar{v}',H')\in U(x)$ and $(\bar{v}',H'^{c})\in U(x')$, we have $d(\bar{v},\bar{v}')\le 1$. The only if direction is clear. The other direction is true because if $d(\bar{v},\bar{v}')\le 1$, then $\bar{v}\in St(\bar{v}')\subset H'\cap H'^{c}$. 

Let $\mathcal{W}_{\bar{v}}(\Gamma)$ be the collection of $\bar{v}'$-walls with $d(\bar{v},\bar{v}')\le 1$ and let $\mathcal{H}_{\bar{v}}(\Gamma)$ be corresponding collection of halfspaces. Denote the corresponding $CAT(0)$ cube complex by $X_{\bar{v}}$. Let $\Sigma\subset\mathcal{H}(\Gamma)$ be the subset made of elements $(\bar{w},R)$ such that $d(\bar{w},\bar{v})\ge 2$ and $\bar{v}\in R$. Pick an ultrafilter $U_{\bar{v}}$ of $\mathcal{H}_{\bar{v}}(\Gamma)$, it is easy to see every pair of halfspaces in $\Sigma\cup U_{\bar{v}}$ are compatible, thus $\Sigma\cup U_{\bar{v}}$ is an ultrafilter of $\mathcal{H}(\Gamma)$ and this induces a cubical embedding $$i_{\bar{v}}:X_{\bar{v}}\to X.$$ Note that $i_{\bar{v}}(X_{\bar{v}})$ is convex in $X$ since two walls in $\mathcal{W}_{\bar{v}}(\Gamma)$ are transverse in $\mathcal{W}_{\bar{v}}(\Gamma)$ if and only if they are transverse in $\mathcal{W}(\Gamma)$. Since every $\bar{v}$-wall is transverse to all $\bar{w}$-walls with $d(\bar{w},\bar{v})=1$, $X_{\bar{v}}$ admits a canonical splitting $$X_{\bar{v}}=h_{\bar{v}}\times[0,d_{\bar{v}}-1],$$ here $h_{\bar{v}}$ is isomorphic to the hyperplane in $X$ corresponding to a $\bar{v}$-wall, and $d_{\bar{v}}$ is the number of prime factors at $\bar{v}$. We will view $X_{\bar{v}}$ as a convex subcomplex of $X$. Note that vertices of $X_{\bar{v}}$ are exactly those vertices $x$ with $\bar{v}\subset\Phi(x)$. Moreover, $e\subset X$ is a $\bar{v}$-edge if and only if $e\in X_{\bar{v}}$ and $e$ has trivial projection to the $h_{\bar{v}}$-factor. 

Suppose we have already constructed $g_{i}$ and $f_{i}$. Let $e_{i},h_{i},N_i$ and $(\bar{v},H_{i},H^{c}_{i})=\Phi(h_{i})$ be as in step 2 and let $v=f_{i}(\bar{v})$. Pick vertex $x\in L_{i}$. Then $$x\in X_{\bar{v}}\Leftrightarrow\bar{v}\in\Phi(x)\Leftrightarrow v\in f_{i}(\Phi(x))\Leftrightarrow v\in \widehat{g_{i}(x)}\Leftrightarrow g_{i}(x)\in P_{v}.$$ Thus $g_{i}$ induced an isomorphism between $X_{\bar{v}}\cap L_{i}$ and $P_{v}\cap g_{i}(L_{i})$. Let $$X_{\bar{v}}\cap L_{i}=\bar{K}_{i}\times \bar{I}_{i}$$ and $$P_{v}\cap g_{i}(L_{i})=K_{i}\times I_{i}$$ be the splitting induced from the splitting of $X_{\bar{v}}$ and $P_{v}$ as above ($\bar{K}_{i}\subset h_{\bar{v}}$, $K_{i}\subset h_{v}$, $\bar{I}_{i}\subset [0,d_{\bar{v}}-1]$ and $I_{i}\subset \ell_{v}$). By (3), $g_{i}|_{X_{\bar{v}}\cap L_{i}}=g_{i1}\times g_{i2}$ with $g_{i1}:\bar{K}_{i}\to K_{i}$ and $g_{i2}:\bar{I}_{i}\to I_{i}$. Suppose $\bar{I}_{i}=[0,a]$, we identify $I_{i}$ with $[0,a]$ via $g_{i2}$ and consistently identify $\ell_{v}$ with $\Bbb R$.

Since $e_{i}$ is a $\bar{v}$-edge, $e_{i}\subset X_{\bar{v}}$. We assume without loss of generality that $$x_{i}=e_{i}\cap L_{i}\in \bar{K}_{i}\times\{a\}.$$ Then $$M_{i1}=L_{i}\cap N_{i}=\bar{K}_{i}\times\{a\}$$ and $$N_{i}=\bar{K}_{i}\times[a,a+1].$$ Similarly, $g_{i}(M_{i1})=K_{i}\times\{a\}$. Note that $g_{i1}$ induces an isomorphism from $\bar{K}_{i}\times[a,a+1]$ to $K_{i}\times[a,a+1]$, this defines $$g_{i+1}:L_{i+1}=L_{i}\cup N_{i}\to g_{i}(L_{i})\cup (K_{i}\times[a,a+1]).$$ Moreover, $h_{v}\times[a,a+1]$, which is the carrier of the hyperplane $h_{v}\times\{a+1/2\}$, satisfies $$(h_{v}\times[a,a+1])\cap g_{i}(L_{i})=K_{i}\times\{a\},$$ so $$g_{i}(L_{i})\cup (K_{i}\times[a,a+1])$$ is a compact convex subcomplex in $X(\Gamma')$.

We consider the left action $G(\Ga')\acts X(\Ga')$ and let $\alpha\in G(\Ga')$ be the translation along $\ell_{v}$ such that $$\alpha(K_{i}\times\{a\})=K_{i}\times\{a+1\}.$$ Then $\alpha$ induces an isomorphism $$\alpha_{\ast}:\widehat{K_{i}\times\{a\}}\to\widehat{K_{i}\times\{a+1\}}.$$ It is clear that $\alpha_{\ast}(\hat{x})=\widehat{\alpha(x)}$ for vertex $x\in K_{i}\times\{a\}$ and $\alpha$ sends $v$-edge to $\alpha_{\ast}(v)$-edge. We define $f_{i+1}$ by
\begin{center}
$f_{i+1}(z)$=$\begin{cases}
f_{i}(z) & \text{if $z\in \Phi(L_{i})$}\\
(\alpha_{\ast}\circ f_{i}\circ(\bar{h}_{\ast})^{-1})(z) & \text{if $z\in \Phi(M_{i2})$}
\end{cases}$
\end{center}
Note that $$f_{i}(z)=(\alpha_{\ast}\circ f_{i}\circ(\bar{h}_{\ast})^{-1})(z)$$ for $$z\in\Phi(L_{i})\cap\Phi(M_{i2})=St(\bar{v},\Phi(M_{i1}))$$ (cf. Lemma~\ref{lem:key} (2)), so $f_{i+1}$ is well-defined. Now we show $f_{i+1}$ and $g_{i+1}$ satisfy the compatibility conditions (2) and (3). Since $$g_{i+1}|_{M_{i2}}=\alpha\circ g_{i}\circ p^{-1},$$ (here $p^{-1}:M_{i2}\to M_{i1}$ is the parallelism map), it suffices to check $p^{-1}$ and $(\bar{h}_{\ast})^{-1}$ satisfy the corresponding compatibility conditions. However, $$(\bar{h}_{\ast})^{-1}(\Phi(x))=\Phi(p^{-1}(x))$$ for vertex $x\in M_{i2}$ by Lemma~\ref{lem:key} (1). Let $e\subset M_{i2}$ be a $\bar{w}$-edge. Then $p^{-1}(e)$ is also a $\bar{w}$-edge. We also deduce that $d(\bar{w},\bar{v})=1$, hence $(\bar{h}_{\ast})^{-1}(\bar{w})=\bar{w}$. It follows that $p^{-1}$ sends $\bar{w}$-edge to $(\bar{h}_{\ast})^{-1}(\bar{w})$-edge.

Let $f:F(\Gamma)\to\mathcal{P}(\Gamma')$ and $g:X\to X(\Gamma')$ be the simplicial embedding and the cubical embedding obtained by the above induction. Then $E=g(X)$ is a compact convex subcomplex of $X(\Gamma')$ and $$\hat{E}=f(F(\Gamma))\cong F(\Gamma).$$ Thus $G(\Ga)$ is isomorphic to a special subgroup of $G(\Ga')$. In particular, $G(\Gamma')$ and $G(\Gamma)$ are quasi-isometric, so $\Gamma'$ is also of type II by Corollary \ref{7.22}. Next we show $G(\Gamma')$ is prime.

%\begin{lem}
%\label{8.14}
%Pick vertex $x\in E$ and $v\in\hat{x}$. Let $\bar{v}=f^{-1}(v)$. Then 
%\begin{equation*}
%\frac{\textmd{number\ of\ components\ in\ }\hat{x}\setminus St(v) }{\textmd{number\ of\ components\ in\ }\hat{E}\setminus St(v) }\le \frac{1}{d_{\bar{v}}}.
%\end{equation*}
%Recall that $d_{\bar{v}}$ is the number of prime factors at $\bar{v}$.
%\end{lem}
%
%\begin{proof}
%We use $c(K)$ to denote the number of components in $K$. By consider the $g$-image of $X_{\bar{v}}\cong h_{\bar{v}}\times [0,d_{\bar{v}}-1]$, we deduce that there is a segment $I_v\subset E$ of length $=d_{\bar{v}}-1$ such that it is made of $v$-edges and it contains $x$. It follows from Corollary \ref{7.14} that two vertices of $\hat{I_v}$ are in the same $v$-branch if and only if they are in the same component of $\hat{y}\setminus St(v)$ for some $y\in I_v$. Thus $\hat{I_v}$ intersects at least $d_{\bar{v}} \cdot c(\hat{x}\setminus St(v))$ many $v$-branches. The same is true for $\hat{E}$ and the lemma follows.
%\end{proof}

Take a vertex in $\bar{u}\in F(\Gamma')$. Let $x\in E$ be a vertex and let $v\in\hat{x}$ be the lift of $\bar{u}$.  Put $\bar{v}=f^{-1}(v)$.
Let $r:X(\Ga')\to X(\Ga)$ be the map in Theorem \ref{retraction map} and let $r_{\ast}:\mathcal{P}(\Gamma')\to\mathcal{P}(\Gamma)$ be the induced simplicial isomorphism.  We claim the stretch factor of $r_*$ at $v$ is $d_{\bar v}$. Note that this claim and Lemma~\ref{8.4} imply $\bar u$ is prime.

It remains to prove the claim. Lemma \ref{8.4} implies that this stretch factor is upper bounded by the number of prime factors at $\pi\circ r_{\ast}(v)$, which is equal to $d_{\bar{v}}$ by Remark \ref{equal number of prime factors} (note that the composition $$F(\Ga)\stackrel{f}{\to}\P(\Ga')\stackrel{r_{\ast}}{\to}\P(\Ga)\stackrel{\pi}{\to}F(\Ga)$$ is a simplicial isomorphism). Now we produce the lower bound.

By considering the $g$-image of $X_{\bar{v}}\cong h_{\bar{v}}\times [0,d_{\bar{v}}-1]$, we deduce that there is a segment $I_v\subset E$ of length $=d_{\bar{v}}-1$ such that it is made of $v$-edges and it contains $x$. Let $\{x_1,\ldots,x_{d_{\bar v}}\}$ be vertices of $I_v$. Then each component $C$ of $F(\Gamma')\setminus St(\bar u)$ gives rise to a component $C_i$ of $\hat x_i\setminus St(v)$ via $$F(\Gamma')\setminus St(\bar u)\cong \hat x\setminus St(v).$$ Let $B_{C_i}$ be the $v$-branch containing $C_i$.
It follows from Corollary \ref{7.14} that $B_{C_i}\neq B_{C_j}$ unless $i=j$, moreover, for a component $C'$ of $F(\Gamma')\setminus St(\bar u)$ with $C'\neq C$, we have $$B_{C_i}\neq B_{C'_j}$$ for $1\le i,j\le d_{\bar v}$. Let $\Pi$ be the map defined before Definition~\ref{def:6.4}. Then $$\Pi(B_{C_i})=C$$ for $1\le i\le d_{\bar v}$. However, as each $B_{C_i}$ contains a component of $$\hat I_v\setminus St(v)\subset \hat E\setminus St(v),$$ we know that $$\Pi(r_*(B_{C_i}))\neq \Pi(r_*(B_{C_j}))$$ for $i\neq j$ and $$\Pi(r_*(B_{C_i}))\neq \Pi(r_*(B_{C'_j}))$$ when $C'\neq C$ (note that $r(E)$ is a vertex in $X(\Ga)$, and $r_{\ast}(\hat{E})=\widehat{r(E)}$ by (1) and (2) of Theorem \ref{retraction map}). This means stretch factor of $r_{\ast}$ at $v$ (Lemma \ref{8.4}) is $\ge d_{\bar{v}}$. Thus the claim is proved.
\end{proof}

%Thus $\hat{I_v}$ intersects at least $d_{\bar{v}} \cdot c(\hat{x}\setminus St(v))$ many $v$-branches. The same is true for $\hat{E}$ and the lemma follows.
%two vertices of $\hat{I_v}$ are in the same $v$-branch if and only if they are in the same component of $\hat{y}\setminus St(v)$ for some $y\in I_v$. Take a connected component of . Then it gives a component of $\hat x\setminus St(v)$ via $F(\Gamma')\setminus St(\bar u)\cong \hat x\setminus St(v)$. Let $B$ be the $v$-branch containing $\hat x\setminus St(v)$. 

\begin{remark}
\label{8.16}
Suppose $\Ga$ is of type II. For each vertex $v\in\mathcal{P}(\Gamma)$, we pick an identification $f_v$ between the collection of $v$-sub-tiers and a copy of integers $\Z_v$. A $v$-halfspace is a subcomplex of $\P(\Ga)$ of form $St(v)\cup f^{-1}_{v}([a,\infty))$ or $St(v)\cup f^{-1}_{v}((-\infty,a])$, where $a\in\Z_v$. We can put a pocset structure on the collection of all these halfspaces in a similar way as before. Then the $CAT(0)$ cube complex associated with this pocset is isomorphic to $X(\Ga')$.

We will not use this fact, so we will not give the detailed argument. However, it is instructive to think about the case when $\out(G(\Gamma))$ is finite. Then the cube complex associated with the above pocset is actually isomorphic to $X(\Ga)$. So the quasi-isometry rigidity/flexibility of $G(\Gamma)$ is reflected in how hard it is to reconstruct $X(\Ga)$ from $\P(\Ga)$ via cubulation. 
\end{remark}

It follows from Corollary \ref{7.22}, Theorem \ref{8.15} and Theorem \ref{8.9} that

\begin{thm}
\label{8.17}
If $G(\Gamma_{1})$ is a right-angled Artin group of type II, then $G(\Gamma_{2})$ is quasi-isometric to $G(\Gamma_{1})$ if and only if $G(\Gamma_{2})$ is commensurable with $G(\Gamma_{1})$. Moreover, there exists a unique prime right-angled Artin group $G(\Gamma)$ such that $G(\Gamma_{1})$ and $G(\Gamma_{2})$ are isomorphic to special subgroups of finite index in $G(\Gamma)$.
\end{thm}

Now we discuss several related examples.

\begin{example}
\label{8.18}
Let $\Gamma_{1}$ be a 5-gon and let $\Gamma_{2}$ be a 6-gon. We glue $\Gamma_{1}$ and $\Gamma_{2}$ along a vertex star to form $\Gamma$ and claim $\Gamma$ is prime. Let $\bar{v}$ be the only vertex of $\Gamma$ such that $St(\bar{v})$ separates $\Gamma$ and for $i=1,2$ and let $C_{i}=\Gamma_{i}\setminus St(\bar{v})$. Pick $v\in\mathcal{P}(\Gamma)$ such that $\pi(v)=\bar{v}$ and let $B_{i}$ be a $v$-branch such that $\Pi(B_{i})=C_{i}$. It suffices to show $B_{1}$ and $B_{2}$ are not QII.

Suppose the contrary is true and let $q$ be the quasi-isometry such that $q_{\ast}(B_{1})=B_{2}$. By Corollary \ref{7.14}, there exist vertices $x_{1},x_{2}\in P_{v}$ such that $(\Gamma_{i})_{x_{i}}\setminus St(v)\subset B_{i}$ for $i=1,2$ (here $(\Gamma_{i})_{x_{i}}\subset(\Gamma)_{x_{i}}$ is the lift of $\Gamma_{i}$). Note that for any $v\in \mathcal{P}(\Gamma)$, $(\Gamma_1)_{x_1}\setminus St(v)$ is contained in a single $v$-branch - this follows from Lemma \ref{7.8}. Thus for any $v\in \mathcal{P}(\Gamma)$, $q_{\ast}((\Gamma_{1})_{x_{1}})\setminus St(v)$ is contained in a single $v$-branch. Then Lemma \ref{8.7} implies that $\cap P_{v_i}\neq\emptyset$ with $v_i$ ranges over vertices in $q_{\ast}((\Gamma_{1})_{x_{1}})$. As $v\in q_{\ast}((\Gamma_{1})_{x_{1}})$, we know $q_{\ast}((\Gamma_{1})_{x_{1}})=(\Gamma_{1})_{x_{3}}$ for some vertex $x_{3}\in P_{v}$. Take a vertex $u$ in $(\Gamma_1)_{x_1}\setminus St(v)$. Then $P_{q_*(u)}\cap P_v\neq\emptyset$  and $\pi(q_*(u))\subset \pi((\Gamma_{1})_{x_{3}})=\Gamma_1$. On the other hand, $(\Gamma_{2})_{x_{2}}\setminus St(v)\subset B_{2}$.
Then Lemma \ref{7.9} implies that $q_{\ast}(B_{1})\neq B_{2}$, which is a contradiction.

Theorem \ref{8.17} implies that any $G(\Gamma')$ quasi-isometric $G(\Gamma)$ is isomorphic to a finite index subgroup of $G(\Gamma)$. Note that by the same proof, this statement is true in the case when $\Gamma$ is obtained by gluing two distinct graphs $\Gamma_{1}$ and $\Gamma_{2}$ ($\out(G(\Gamma_{i}))$ is finite for $i=1,2$) along an isomorphic vertex star.
\end{example}

\begin{example}
\label{8.19}
Let $\Gamma$ be a pentagon
and a hexagon glued together over the star of a vertex. Let $\Gamma'$ be a pentagon
and two hexagons glued together over the star of a vertex. See the picture below. We claim $\mathcal{P}(\Gamma)$ and $\mathcal{P}(\Gamma')$ are isomorphic, but $G(\Gamma)$ and $G(\Gamma')$ are not quasi-isometric. 
\begin{center}
\includegraphics[scale=0.6]{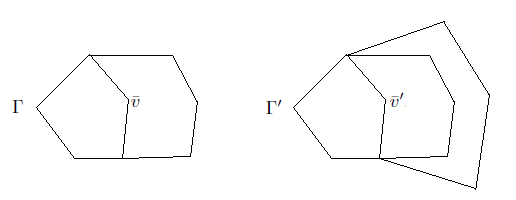}
\end{center}

First we show $G(\Gamma)$ and $G(\Gamma')$ are not quasi-isometric. Suppose the contrary is true and let $q:G(\Gamma)\to G(\Gamma')$ be a quasi-isometry. Let $\pi:\mathcal{P}(\Gamma)\to F(\Gamma)$ and $\pi':\mathcal{P}(\Gamma')\to F(\Gamma')$ be the canonical projection map defined after Lemma~\ref{coarse contain of standard subcomplexes}.

Pick vertex $v\in\mathcal{P}(\Gamma)$ such that $\pi(v)=\bar{v}$ and let $v'=q_{\ast}(v)$. Then $\pi'(v')=\bar{v}'$. This follows from the fact that $\pi(v)=\bar{v}$ (or $\pi'(v')=\bar{v}'$) if and only if there are at least two QII classes among all the $v$-branches (or $v'$-branches). This fact follows from the discussion in Example \ref{8.18} and Corollary \ref{7.14}.

Now we compute the stretch factor of $q_*$ at $v$. Corollary \ref{7.14} implies that a $v$-tier contains two $v$-branches $B_1$ and $B_2$ such that $\Pi(B_1)$ and $\Pi(B_2)$ gives the two connected components of $\Gamma\setminus St(\bar v)$. Suppose $\Pi(B_1)$ is contained in the pentagon subgraph of $\Gamma$, and $\Pi(B_2)$ is contained in the hexagon subgraph of $\Gamma$.
Corollary \ref{7.14} also implies that a $v'$-tier contains three $v'$-branches $B'_1,B'_2$ and $B'_3$ such that $\Pi(B'_1),\Pi(B'_2)$ and $\Pi(B'_3)$ gives the three connected components of $\Gamma'\setminus St(\bar v')$. Suppose $\Pi(B'_2)$ and $\Pi(B'_3)$  are inside the two hexagons in $\Gamma'$. Then using the symmetry of $\Gamma'$ that exchanges the two hexagons, we know $B'_2$ and $B'_3$ are QII. Thus the two hexagons give components of $\Gamma'\setminus St(\bar v')$ that are QII.

We claim $\Pi(q_*(B_1))=\Pi(B'_1)$. Let $x_1$ and $(\Gamma_1)_{x_1}$ be as in Example~\ref{8.18}. We argue as in the second paragraph of Example~\ref{8.18} to see that there exists $x'_1\in P_{v'}$ such that $q_*((\Gamma_1)_{x_1})=(\Gamma'_1)_{x'_1}$ where $\Gamma'_1$ is the pentagon subgraph of $\Gamma'$ and  $(\Gamma'_1)_{x'_1}$ are defined in the same way as $(\Gamma_1)_{x_1}$. Thus $\Pi(q_*(B_1))=\Pi(B'_1)$. Similarly we can show $\Pi(q_*(B_2))=\Pi(B'_2)$ or $\Pi(B'_3)$, $\Pi(q^{-1}_*(B'_1))=\Pi(B_1)$, and $\Pi(q^{-1}_*(B'_2))=\Pi(q^{-1}_*(B'_3))=\Pi(B_2)$. Thus if we use $B_1$ to compute the stretch factor as in Definition~\ref{def:stretch factor}, we will conclude that the factor is $1$. If we use $B_2$ to compute the stretch factor as in Definition~\ref{def:stretch factor}, we obtain the factor is $2$. This contradicts Lemma~\ref{8.4}. Thus the quasi-isometry $q$ does not exist.

It remains to show $\mathcal{P}(\Gamma)$ and $\mathcal{P}(\Gamma')$ are isomorphic.  Let $f_{1}$ and $f_{2}$ be two simplicial embeddings from $\Gamma$ to $\Gamma'$ such that (1) they cover different 6-gons in $\Gamma'$; (2) $f_{1}=f_{2}$ when restricted to the 5-gon in $\Gamma$. We also use $f_{i}$ to denote the group monomorphism induced by $f_{i}$. Let $\omega\in G(\Gamma)$ be a geodesic word and write $\omega=\omega_{1}a_{1}\cdots\omega_{n}a_{n}\omega_{n+1}$, here $a_{i}$ is a product of powers of elements in $St(\bar{v})$ for all $i$, but $\omega_{i}$ does not contain any powers of elements in $St(\bar{v})$ ($\omega_{1}$ or $\omega_{n+1}$ may be trivial). By permuting letters in $a_{i}$, we have $a_{i}=\bar{v}^{k_{i}}b_{i}$, where $b_{i}$ does not contain any power of $\bar{v}$. 

Define a map $h:G(\Gamma)\to G(\Gamma')$ by mapping $\omega$ to $\omega'_{1}a'_{1}\cdots\omega'_{n}a'_{n}\omega'_{n+1}$ such that (1) $\omega'_{i}=f_{1}(\omega_{i})$ if and only if $k_{i-1}/2$ is an integer, otherwise $\omega'_{i}=f_{2}(\omega_{i})$; (2) $a'_{i}=f_{1}(a_{i})$ if and only if the first letter of $w_{i+1}$ is inside the 5-gon, otherwise $a'_{i}=\bar{v}'^{\lfloor k_{i}/2\rfloor}\cdot f_{1}(b_{i})$. Given a different geodesic word $\omega_{1}=\omega$, we can obtain $\omega_{1}$ from $\omega$ by using the commutator relations to permute the letters in $\omega$, moreover, each word in the middle is also a geodesic word. Now it is easy to check that $h$ is well-defined, and for each $S$-geodesic $\ell\subset G(\Gamma)$, there exists a unique $S'$-geodesic $\ell'\subset G(\Gamma')$ such that $h(\ell)=\ell'$ up to finitely many points, moreover, if two $S$-geodesic are parallel (or orthogonal), then the corresponding $h$-images are parallel (or orthogonal), thus $h$ induces a simplicial map $h_{\ast}:\mathcal{P}(\Gamma)\to\mathcal{P}(\Gamma')$. We can define a map $h':G(\Gamma')\to G(\Gamma)$ in a similar fashion which serves as the inverse of $h$, which would imply that $h_{\ast}$ is actually a simplicial isomorphism.
\end{example}

\bibliographystyle{alpha}
\bibliography{1}
\end{document}